\documentclass[11pt]{amsart}
\usepackage{}
\usepackage{amssymb}
\usepackage{amsmath}
\usepackage{amsfonts}

\usepackage{amsmath,amssymb,amsthm}
\usepackage{version, tabularx, multicol}
\usepackage{graphicx,float}
\usepackage{color,latexsym,amsfonts}
\usepackage{extarrows}
\usepackage[shortlabels]{enumitem}
\usepackage{bbm}
\usepackage{bm}

\theoremstyle{plain}
\newtheorem{theo*}{Theorem}
\newtheorem{theo}{Theorem}[section]
\newtheorem{defi}[theo]{Definition}
\newtheorem{cor}[theo]{Corollary}
\newtheorem{prop}[theo]{Proposition}
\newtheorem{lem}[theo]{Lemma}

\theoremstyle{remark}
\newtheorem{rem}[theo]{Remark}
\newtheorem{rem*}{Remark}

\newlist{propenum}{enumerate}{1}
\setlist[propenum]{label=(\roman*)}

\newcommand{\ca}{{\mathcal A}}

\newcommand{\cc}{{\mathcal C}}
\newcommand{\cd}{{\mathcal D}}

\newcommand{\cj}{{\mathcal J}}

\newcommand{\cl}{{\mathcal L}}

\newcommand{\ct}{{\mathcal T}}
\newcommand{\cu}{{\mathcal U}}

\newcommand{\E}{{\mathbb E}}

\newcommand{\N}{{\mathbb N}}
\renewcommand{\P}{{\mathbb P}}

\newcommand{\R}{{\mathbb R}}

\newcommand{\T}{{\mathbb T}}
\newcommand{\Z}{{\mathbb Z}}

\newcommand{\ba}{\mathbf a}
\newcommand{\bb}{\mathbf b}

\newcommand{\be}{\mathbf e}
\newcommand{\bk}{\mathbf k}

\newcommand{\bn}{\mathbf n}

\newcommand{\bt}{{\mathbf t}}

\newcommand{\bs}{{\mathbf s}}

\newcommand{\bv}{{\mathbf v}}

\newcommand{\bx}{\mathbf x}

\newcommand{\ind}{{\mathbbm 1}}
\newcommand{\un}{\bm{1}}
\newcommand{\zero}{\bm{0}}
\newcommand{\roott}{\emptyset}
\newcommand{\bp}{\mathbf{p}}

\newcommand{\qc}{\theta_{\alpha}}

\newcommand{\supp}{{\rm supp}\,}
\newcommand{\Card}{{\rm Card}\;}

\newcommand{\dist}{{\rm dist}\;}

\newcommand{\val}[1]{\mathop{\left| #1 \right|}\nolimits}
\newcommand{\inv}[1]{\mathop{\frac{1}{ #1}}\nolimits}

\newcommand{\param}{{\rm Pa}}

\newcommand{\II}[1]	{[\![#1]\!]}
\newcommand{\lb}{[\![}
\newcommand{\rb}{]\!]}

\DeclareFontFamily{U}{mathb}{\hyphenchar\font45}
\DeclareFontShape{U}{mathb}{m}{n}{ <-6> matha5 <6-7> matha6 <7-8>
mathb7 <8-9> mathb8 <9-10> mathb9 <10-12> mathb10 <12-> mathb12 }{}
\DeclareSymbolFont{mathb}{U}{mathb}{m}{n}

\DeclareMathAccent{\abxring}{0}{mathb}{"38}

\DeclareFontFamily{U}{mathb}{\hyphenchar\font45}
\DeclareFontShape{U}{mathb}{m}{n}{ <-6> matha5 <6-7> matha6 <7-8>
mathb7 <8-9> mathb8 <9-10> mathb9 <10-12> mathb10 <12-> mathb12 }{}
\DeclareSymbolFont{mathb}{U}{mathb}{m}{n}

\title[Conditioning BGW trees to have large
sub-populations]{Conditioning Bienaym\'{e}-Galton-Watson trees to have
  large sub-populations}

\date{\today}

\begin{document}
\author{Romain Abraham}
\address{Romain Abraham, Institut Denis Poisson,
Universit\'{e} d'Orl\'{e}ans,
Universit\'e de Tours, CNRS, France}
\email{romain.abraham@univ-orleans.fr}

\author{Hongwei Bi}
\address{Hongwei Bi, University of international
business and economics, China}
\email{bihw@uibe.edu.cn}

\author{Jean-Fran\c{C}ois Delmas}
\address{Jean-Fran\c{c}ois Delmas,  CERMICS, \'{E}cole des Ponts, France}
\email{jean-francois.delmas@enpc.fr}



\begin{abstract}
We study the local limit in distribution  of
Bienaym\'{e}-Galton-Watson trees conditioned on having large
sub-populations. Assuming a generic and aperiodic condition on the offspring
distribution, we prove the existence of a limit given by a  Kesten's  tree
associated with a certain critical offspring distribution.
\end{abstract}

\maketitle

\section{Introduction}

Local  limit of large
Bienayme-Galton-Watson (in short, BGW) trees have been extensively
studied in recent years.
The classical result of Kesten \cite{k86} describes the local limit
of a critical or subcritical BGW tree conditioned on reaching at least
height $h$ locally converges in distribution as $h$ goes to infinity to the
so-called size-biased tree or Kesten's tree, which is a tree with an
infinite spine. We refer to Section~\ref{sec:kesten} for a precise
description of the Kesten's tree.

Over  the years,  motivated by  various point  of view  from theoretical
probability, combinatorics, biology or physics, other conditionings have
been considered,  such as  large total  progeny \cite{k75,  gk04}, large
number  of leaves  \cite{js11, ck14},  large number  of protected  nodes
\cite{abd17},  existence  of  an  individual  with  a  large  number  of
out-degree or  children \cite{he17,  he22}.  Janson  \cite{j12} surveyed
the  local  limit  of  BGW  trees when  conditioned  on  a  large  total
population size, and  Abraham and Delmas \cite{ad14a,  ad14b} provided a
general framework, which describes in full generality the local limit of
critical  or  subcritical  BGW  trees  conditioned  on  having  a  large
sub-population.   Notice that  in \cite{js11,j12,ad14b,  s19} the  local
limit may exhibit a condensation phenomenon, as one node of the limiting
tree has an  infinite number of children. With  other conditioning, such
as  a large  size  of  a late  generation  \cite{abd20,  ad19}, or  with
exponential weight given by the total height of the tree among tree with
given large size \cite{du23}, the local limit is a tree with an infinite
backbone. Local limit  of large multi-type Galton-Watson  trees has also
been considered in \cite{p16,  adg18}, and also in~\cite{s18,thevenin23}
when conditioning on a linear combination  of the sizes of the sub-populations with a given type.

\medskip

One  can also  consider  scaling limits  of BGW  trees  (seen as  metric
space), the  so called  L\'evy continuum trees,  as initiated  by Aldous
\cite{a91} and generalized by Duquesne  and Le Gall \cite{dlg02}. Let us
mention that  there is  a large  recent literature  on this  subject. In
particular  Marzouk \cite{marzouk22}  considered  the  scaling limit  of
random trees  with a prescribed  degree sequence, and  Kargin \cite{k23}
and Kortchemski and Marzouk \cite{km23} the scaling limit
of BGW trees conditioned on its total progeny and the number of leaves.

Motivated by those  last works, we shall investigate the  local limit of
BGW tree when conditioning on its total progeny and the number of leaves
being large.   More generally, let  $\ca=(A_i)_{i\in \II{1,J}}$, where
$\II{a,b}=[a,b]\cap \N$,   be a
finite    collection     of    pairwise    subsets    of     $\N$    and
$\alpha=(\alpha_i)_{i\in   \II{1,   J}}$   a  probability   measure   on
$\II{1, J}$. We  denote by $A_0$ the  complementary of
$\bigcup_{i\in \II{1 ,J}} A_i$ in  $\N$, which might be  empty or
not.   We shall  then generalize  \cite{ad14a} and  consider the  local
limit of  BGW trees conditioned  to have  $\lfloor \alpha_i n  \rfloor $
nodes with out degree  in $A_i$ for all $i\in \II{1, J}$  as $n$ goes to
infinity. We stress that there is no condition on the nodes with out-degree in $A_0$.

More precisely let $p=(p(n))_{n\in \N}$ be a probability distribution on
$\N$; its support  is $\supp(p)=\{n\in \N\, \colon\,  p(n)>0\}$. We assume
that  $p$   is  non-trivial  in   the  sense  that   $p(0)>0$  and
$p(\{0, 1\})<1$, where $p(A)=\sum_{n\in A}  p(n)$ for $A\subset \N$.  We
denote by  $\mu(p)\in (0, +\infty ]$  its mean.  We denote  by $\ct_p$ a
BGW tree  with offspring distribution  $p$. Notice we don't  assume that
$\mu(p)$  is even  finite, however  since  $p(0)$ is  positive the  tree
$\ct_p$     is     finite     with    positive     probability.      Set
$\bar \N=\N \cup\{\infty \}$.  For a  tree $\bt$ we denote by $L_A(\bt)$
the  number of  its nodes  with out-degree  (or number  of children)  in
$A$. We simply set:
\[
  L_\ca(\bt)=(L_{A_i}(\bt))_{i\in \II{1, J}}\in \bar \N^J.
\]

In Theorem~\ref{th:p,a-compatibility} we completely characterize the
non-trivial probability distributions $p'$ on $\N$ such that
$\supp(p')\subset \supp(p)$ and for all $\bn=(n_i)_{i\in \II{1, J}}  \in \N^J$
such that
    $\P(      L_{\ca}(\ct_{p'})     =     \bn)>0$      (and     thus
$\P( L_{\ca}(\ct_p) = \bn)>0$), we have:
\[
\dist\left(\ct_p\, \big|\,  L_{\ca}(\ct_p) = \bn\right)
= \dist\left(\ct_{p'}\, \big|\,  L_{\ca}(\ct_{p'}) = \bn\right).
\]
Such probability distributions are called \emph{$(p, \ca)$-compatible}.
The $(p, \ca)$-compatible probability distributions can be continuously
parametrized by a parameter $(\theta, \beta)$ in a subset of $[0,
+\infty ]\times \R_+^J$. When the parameter $\theta$ is positive and
finite, then the $(p, \ca)$-compatible  probability distribution $\tilde p_{\theta,
  \beta}$ associated with the parameter $(\theta, \beta)$ is given by:
\[
  \tilde p_{\theta, \beta}(n) = \beta_i \theta^n \, p(n)
  \quad\text{for}\quad n\in A_i
  \quad\text{and}\quad i\in \II{0, J}
  ,
  \quad\text{where}\quad
  \beta_0=\theta^{-1}.
\]
The fact  that such  exponentially tilted probability  distributions are
$(p, \ca)$-compatible was already observed in~\cite{ad14a} for $J=1$ and
in Th\'evenin~\cite{thevenin23} for  the multi-type BGW tree  setting.  
In comparison with those two papers, we give here an exhaustive description of the $(p, \ca)$-compatible probability distributions. 
In particular, it
is possible to observe degenerate  cases when the parameter $\theta$ can
take   the   values   $0$    and   $\infty   $,   see~\eqref{eq:def-pt0}
and~\eqref{eq:def-pt-infty}.  In both   cases,  when  possible,   we get
  that  $0\not   \in   A_0$  and  for  the  latter   that
$A_0\cap \supp(p)$ is either empty  or reduced to $\{1\}$.  As suggested
by  this remark,  we  shall indeed  distinguish in  most  of the  proofs
according to  $0\not \in A_0$ (the  leaves are directly involved  in the
conditioning) or not.

\medskip

For $\bx=(x_i)_{i\in \II{1, J}}\in \R^J$, we set
$|\bx|=\sum_{i\in  \II{1, J}} |x_i|$ the $L^1$ norm of $\bx$.
As in~\cite{p16}, it  is
interesting to have a  fixed (asymptotic) proportion of sub-populations,
and  thus  consider   the  local  limit  of   $\ct_p$  conditionally  on
$\{  L_{\ca}(\ct_p)   =  \bn\}$  when  $\bn/|\bn|$   converges  to  some
$\alpha\in \R_+ ^J$ such that $|\alpha|=1$,  as $|\bn|$
goes to infinity.
For  $p$ non-trivial, intuitively we have that $L_A(\ct_p)$ is of order
$p(A)$     times     the     total    size     of     the     population
$\sharp  \ct_p=L_\N(\ct_p)$ when the tree is large, see
\cite{janson16,thevenin20} for precise statement. 
Thus, it is natural to consider among the $(p, \ca)$-compatible
probability distribution those which are in the direction $\alpha=(\alpha_i)_{i\in \II{1, J}}$, that
is:
\[
\tilde p_{\theta, \beta} (A_i)  \propto\, 
\alpha_i
\quad\text{for all}\quad
i\in \II{1, J}.
\]
From this,  we can write  $\alpha$ as  a function of  $(\theta, \beta)$,
provided that $\tilde p_{\theta,  \beta} (A_0^c)>0$ or equivalently that
$\beta\neq  \zero$  (see   Remark~\ref{rem:theta=0}  for  details),  and
similarly $\beta$  as a function  of $(\theta, \alpha)$.  This  gives an
elementary  reparametrization  $(p_{\theta,   \alpha})$  of  the  family
$(\tilde p_{\theta, \beta})$ by the parameter $\theta$ and its direction
$\alpha$,      provided     $\beta\neq      \zero$.      The      family
$(\tilde    p_{\theta,     \beta})$    is    explicitely     given    in
Lemma~\ref{lem:p-qa}     and     the      possible     directions     in
Proposition~\ref{prop:admiss}.   In   particular,  $\alpha$  is   not  a
possible direction  if and only  if there  exists $j\in \II{1,  J}$ such
that    $\alpha_j=0$    and    $0\in   A_j$    or    $\alpha_j=1$    and
$A_0\cup A_j\subset \{0, 1\}$ (those two latter conditions correspond to
$p_{\theta, \alpha}$  being non-trivial). In particular,  if the entries
of   $\alpha$   are  all   positive   then   $\alpha$  is   a   possible
direction. Furthermore,  if $\alpha$ is  a possible direction,  then the
set $I_\alpha\subset [0, +\infty ]$  of possible value for the parameter
$\theta$ is an interval,  see Section~\ref{sec:mean-alpha}.  As observed
in previous works, the existence of a critical parameter $\qc$ such that
$\mu(p_{\qc, \alpha})=1$,  is a key point  to obtain the local  limit of
the conditioned BGW tree.   Proposition~\ref{prop:uniq} asserts that the
mean  function $\theta  \mapsto \mu(p_{\theta,  \alpha})$ is  increasing
when $\mu(p_{\theta,  \alpha})\leq 1$,  and thus there  is at  most one such
critical parameter $\qc$.   Let us stress this result is  not obvious as
the  map $\theta  \mapsto \mu(p_{\theta,  \alpha})$ is  not monotone  in
general (see an example in Remark~\ref{rem:mu-monotone}).
When the critical parameter $\qc$ exists, then the distribution $p$ is
called \emph{generic for the direction $\alpha$}, and we set:
\begin{equation}
   \label{eq:def-pa-intro}
  p_\alpha=p_{\qc, \alpha} \quad\text{(and thus $\mu(p_\alpha)=1$).}
\end{equation}

We provide necessary and sufficient conditions for $p$ to be generic in
the direction $\alpha$ in Theorem~\ref{thm:generic} which are similar to
those obtained in~\cite{ad14a} when $J=1$.
We don't study further the
relation between the  sets $\ca=(A_i)_{i\in \II{1, J}}$ and the fact
that $p$ is generic, and refer to~\cite{ad14a} again to appreciate the
complexity already when $J=1$.

\medskip

Let $\ct_p^*$
denote the Kesten tree associated with the probability distribution $p$
when $\mu(p)=1$, that is, the local limit in distribution of $\ct_p$
conditioned to have height at least $h$, as $h$ goes to infinity.
Before giving the main result of the paper, we recall the hypothesis:
\begin{enumerate}[(H1)]
\item\label{it:H1}
  $p$ is a \emph{non-trivial} probability distribution on $\N$ (there is no
  moment condition). Without loss of generality we assume that the sets
  $\ca=(A_i)_{i\in  \II{1, J}}$ are pairwise disjoint
 subsets of the support of $p$.

\item \label{it:Hdir-poss}
  $\alpha\in \R^J_+$ with  $|\alpha|=1$, is  a
\emph{possible direction} (this condition is always satisfied if all the entries of
$\alpha$ are positive).

\item $p$ is \emph{generic} in the direction $\alpha$.

\item  \label{it:Hlast}   $p$  is  \emph{aperiodic}  in   the  sense  of
  Definition~\ref{defi:period}. (In  particular being  aperiodic depends
  on the    sets $\ca$  and the
  direction $\alpha$, see also Remark~\ref{rem:period0}.)
\end{enumerate}

We are now ready to state the main result of the paper.
Recall  $p_\alpha$ in~\eqref{eq:def-pa-intro}.

\begin{theo*}
  \label{th:main-intro}
Assume that Hypothesis~\ref{it:H1}-\ref{it:Hlast} hold.
We have the following local limit in distribution:
  \[
\dist(\ct_p\,|\, L_{\ca}(\ct_p)=\bn)
\xrightarrow[|\bn|\to\infty]{} \dist(\ct^*_{p_{\alpha}}),
\]
along any sequence $(\bn)$ in  $\N^J$ such that: the sub-populations are
large, that is, $\lim |\bn|=\infty$; the conditioning is legit, that is,
$\P(L_\ca(\ct_p)=\bn)>0$, and  the direction is strictly  $\alpha$, that
is,  $  \lim_{|\bn|\rightarrow  \infty  }  \bn/|\bn|=\alpha$  and,  with
$\alpha=(\alpha_i)_{i\in \II{1,  J}}$ and $\bn=(n_i)_{i\in  \II{1, J}}$,
for all $j\in \II{1, J}$:
\begin{equation}
  \label{eq:cond-zero-intro}
  \alpha_j=0 \implies   n_j= 0.
\end{equation}
\end{theo*}
To be complete, we also refer to Lemma~\ref{lem:a-admiss} on the
existence of a sequences of
$\bn$ in  $\N^J$ satisfying the hypothesis above.

The proof of Theorem~\ref{th:main-intro}  relies on two ingredients. The
first  one is  the use  of Rizzolo's  transformation from~\cite{r15}  to
reduce the problem to the case  $A_0$ empty. The second is the existence
of  a   local  limit  for   multi-type  BGW  tree  conditioned   to  the
sub-populations of  each type to be  large (with proportion given  by the
positive    left   eigenvector    of   the    mean   matrix)    obtained
in~\cite{ad14b}.    One    could    also     use    the    results    of
P\'enisson~\cite{p16}, but  this would require stronger  hypothesis, see
Remark~\ref{rem:related}  for further  comments.   Let  us mention  that
Corollary~3.5 in  Abraham, Delmas  and Guo~\cite{adg18}  gives   Theorem~\ref{th:main-intro} in the very specific case
where $\mu(p)=1$, $p(A_0)=0$,  and the
direction $\alpha$ is the one naturally given by $p$: $\alpha_i=p(A_i)$.

\begin{rem*}[Conditioning on the total size and the number of leaves]
  \label{rem:leaves-intro}
  Motivated by the scaling limits of  BGW trees conditioned on its total
  progeny  and  the  number  of  leaves  to  be  both  large  considered
  in~\cite{k23} and~\cite{km23}, we give as a consequence of the Theorem
  above   the  corresponding   local  limit   of  such   BGW  trees   in
  Remark~\ref{rem:restriction}.  Conditioning on  the total  progeny and
  the number of  leaves amount to consider $A_1=\N^*  \cap \supp(p)$ and
  $A_2=\{0\}$.   Notice  the  directions  $\alpha$  can  be  written  as
  $(a, 1-a)$.  As explained in Remark~\ref{rem:restriction},  we have to
  consider two cases.

If  the
support of $p$ is reduced to two elements say $0$ and $k$ (with $k\geq
2$ as $p$ is non-trivial), then
Hypothesis~\ref{it:Hlast} is not satisfied.  However in this case the
conditioning is equivalent to conditioning on the total size and the
existence of the local limit is then given by~\cite{j12}
and~\cite{ad14a}. Furthermore there is only one possible direction for
which $p$ is generic; it is given by $a=1/k$.

If  the
support of $p$ is not reduced to two elements, then provided the
smallest subgroup in $\Z$ containing  $\{x-y\, \colon\, x, y\in
\supp(p)\cap \N^*\}$ is $\Z$ itself, then  Hypothesis~\ref{it:Hlast} holds. In
this case, Hypothesis~\ref{it:Hdir-poss} is satisfied if and only if
$a\in (0, 1)$. It is
easy to check that $p$ is
generic in the direction $\alpha$ if and only if there exists a positive
finite
root (which is then $\qc$) to the equation:
\[
  g(\theta)=p(0)+ a \theta g'(\theta).
\]
The   critical  probability   measure  $p_\alpha$  is   given  by
$p_\alpha(0)= 1-a$  and $p_\alpha(n)= \qc^{n -1}  p(n)/ g'(\qc)$
for $n\in \N^*$; and we can apply  Theorem~\ref{th:main-intro}.
\end{rem*}

  \begin{rem*}[On the strict convergence of the sequence $\bn$ to the
    direction $\alpha$]
    \label{rem:tbd}
    Assume  that  the  offspring  distribution $p$  is  generic  in  the
    direction  $\alpha$ and  that $\alpha$  has some  zero entries.   We
    provide in  Section~\ref{sec:a=0-n>0} an example where  removing 
    Condition~\eqref{eq:cond-zero-intro}    (that    is,   $n_j=0$    if
    $\alpha_j=0$) on  the sequence of $\bn=(n_i)_{i\in  \II{1, J}}$ such
    that $ \lim_{|\bn|\rightarrow \infty } \bn/|\bn|=\alpha$ prevents to
    get    the   local    limit   of    conditioned   BGW    tree   from
    Theorem~\ref{th:main-intro}.
\end{rem*}

\begin{rem*}[On non-generic distribution]
  \label{rem:gen-intro}
  If $p$ is  not generic in the possible direction  $\alpha$ because of Condition~\ref{item:0inA0-v2} in Theorem~\ref{thm:generic}, then as in
  the case $J=1$ studied in~\cite{ad14b}, we conjecture the existence of
  a condensation phenomenon at the limit: the existence of a node of the
  local limit at finite height with  an infinite degree.  The first step
  to prove  this would be  considering the condensation  for non-generic
  multi-type BGW trees. Notice that the others conditions in Theorem~\ref{thm:generic} might happen for probability distributions  with bounded supports, see Remark~\ref{rem:admi-a}~\ref{item:non-gene-fin}. 
\end{rem*}

\medskip

The  rest  of the  paper  is  structured  as  follows. we  introduce  in
Section~\ref{sec:notation}  the general  notation and  the framework  of
discrete  trees,   BGW  trees   and  Kesten's   tree.   We   define  the
$(p,\ca)$-compatible         probability        distributions         in
Section~\ref{sec:def-pqb} and characterize  all the $(p,\ca)$-compatible
probability distributions in  Section~\ref{sec:compatible} (handling the
degenerate cases $\theta\in \{0, +\infty \}$ and the case $0\not \in A_0$
are delicate). We study in Section~\ref{sec:critical-p} the existence of
the  critical  parameter $\qc$  and  thus  the probability  distribution
$p_\alpha$.  Eventually, we prove the main theorem in Section~\ref{sec:main}
see Theorem~\ref{thm:main-thm} as well as Remark~\ref{rem:tbd}.

\section{Notation}
\label{sec:notation}
\subsection{General notation}
 We denote by $\R_+^*=(0,\infty)$  (resp. $\N^*=\{1, 2, \ldots\}$) the
 set of positive real numbers (resp. integers) and by
$\R_+=[0,\infty)$  (resp. $\N=\{0, 1, \ldots\}$) the set of
nonnegative real numbers (resp. integers). For $i,j\in \N$ such that
$i\leq j$, note $\II{i, j}=\N\cap [i, j]$.  Let $J\in \N^*$. For
$\bx=(x_j)_{j\in \II{1, J}}\in \R^J$, we set $|\bx|=\sum_{j=1}^J |x_j|$. Let:
\[
  \Delta_J = \left\{\bx \in \R_+^J\, \colon\,
|\bx|= 1\right\}.
\]
We set $\un\in  \R^J$ (resp. $\zero\in \R^J$) the vector  of $\R^J$ with
all its coordinates equal to 1 (resp. 0).

\medskip
  Let $p = (p(n))_{ n \in \N} $ be a probability distribution on
  $\N$ and $\supp(p)=\{n\in\N: p(n)>0\}$ be its support.
  For $A \subset \N$, we set:
 \[
   p(A)=\sum_{n\in A} p(n)
   \quad\text{and}\quad
   g_A(r)=\sum_{n\in A} r^n p(n)
   \quad\text{for $r\geq 0$},
 \]
where the sum over an empty set is 0 by convention. In particular, we
have  $g_A=0$ for any set $A\subset \N$
such that $p(A)=0$.
 We also denote by $\rho_A$ the radius of convergence of $g_A$:
 \begin{equation}\label{eq:def-rho-A}
   \rho_A=\sup\{r\geq 1\, \colon\, g_A(r)<+\infty \}.
 \end{equation}
  For simplicity, when $A = \N$, we write $g(r) = g_{\N}(r)$ and
  $\rho=\rho_\N$. We write the mean of $p$ by:
  \[
    \mu(p) = \sum_{n \in \N} np(n).
  \]
We say that a probability distribution $p$ is
\emph{critical}  (resp. \emph{sub-critical}) if   $\mu(p) = 1$
(resp. $\mu(p) < 1$).
The probability distribution  $p$ is \emph{non-trivial}  if:
   \begin{equation}
   \label{assumption-p}
0< p(0) \quad\text{and}\quad
  p(0) + p(1) < 1.
\end{equation}

\subsection{The set of discrete trees}
We consider ordered rooted trees in the framework of
Neveu \cite{n86}. More precisely, let $\cu=\bigcup_{n\geq 0} (\N^*)^n$
be the set of finite sequences of positive integers with the convention
that $(\N^*)^0=\{\emptyset\}$. Note $H(u)=n$ the generation or the height of $u$ if
$u=(u_1,\ldots, u_n)\in (\N^*)^n$. For $u,v\in \cu$, denote by $uv$
the concatenation of $u$ and $v$, with the convention that $uv = u$
if $v = \emptyset$ and $uv = v$ if $u=\emptyset$.
 The set of ancestors of $u$ is the set:
\[
\text{An}(u)=\{v\in\cu: \text{there exists}\ w\in\cu\ \text{such that}\
u=vw\}.
\]
The most recent common ancestor of $\bs\subset\cu$, denoted by $M(\bs)$,
is the unique element $u$ of $\cap_{u\in\bs} \text{An}(u)$ with maximal height $H(u)$.
Let $\prec$ be the usual lexicographic order on $\cu$.

A tree $\bt$ is a subset of $\cu$ that satisfies: $\emptyset \in\bt$; if
$u\in\bt$,  then  $\text{An}(u)\subset\bt$;   for  $u\in\bt$,  there  exists
$k_u(\bt)\in\N^*$, called the out-degree of $u$, such that, for
every $i\in\N^*$,  $ui\in\bt$ if  and only  if $i\in  \II{1, k_u(\bt)}$.
The  vertex  $\emptyset$  is  called  the root  of  $\bt$.   The  vertex
$u\in \bt$  is called a leaf  if $k_u(\bt)=0$.  We set  $k_u(\bt)=-1$ if
$u\not\in \bt$.
Let  $L_A(\bt)$  be the  number of vertices of  the tree $\bt$
whose out-degree   belongs to  $A \subset \N$:
\[
  L_A(\bt)=\Card(\cl_A(\bt))
  \quad\text{with}\quad
  \cl_A(\bt)=\{u\in\bt:\, k_u(\bt)\in A\}.
\]
We simply write $\sharp \bt=L_\N(\bt)$ for the cardinal of $\bt$, and $L_n(\bt)$ for
$L_{\{n\}}(\bt)$ when $n \in \N$.
Remark that we have:
\begin{equation}
  \label{eq:kt=t-1}
\sum_{u\in \bt} k_u(\bt)=\sharp \bt-1.
\end{equation}
Then we get from~\eqref{eq:kt=t-1}:
\begin{equation}
  \label{eq:kt=L0-1}
L_{0}(\bt)=1+ \sum_{k\in \N^*} (k-1) L_{k}(\bt) .
\end{equation}
For $u\in\bt$, we define the  subtree above $u$ by $\{v\in\cu\, \colon\,
uv\in\bt\}$ and the fringe subtree by:
\begin{equation}
  \label{eq:def-fringe}
\bs=\{uv\in \cu\, \colon\,  uv\in\bt\}.
\end{equation}

We denote by $\T$ the set of trees,  $\T_0$ the subset of finite
trees, and $\T_1$  the set of trees with only one  infinite branch:
\[
  \T_1=\left\{\bt\in \T\setminus \T_0\, \colon\, \lim_{n\rightarrow
  \infty } H\big(M(\{u\in \bt\, \colon\, H(u)=n\})\big) =\infty \right\}.
\]
Let $H(\bt)=\sup\{H(u)\, \colon u\in \bt\}$ be the height of the
tree $\bt$; and for $h\in \N^*$, let
$\T^{(h)}=\{\bt\in \T\, \colon\, H(\bt)\leq  h\}$
be the set of trees with height less or equal to $h$.

\subsection{Local convergence of trees}\label{sec:lct}

For $h\in \N$ and a tree $\bt\in \T$, let $r_h(\bt)=\{u\in \bt\,
\colon\, H(u)\leq  h\}$ be the tree $\bt$ truncated at level $h$.
Let $(T_n)_{n\in \N}$ and $T$ be $\T$-valued random variables.
We say that the sequence  $(T_n)_{n\in \N}$ converges locally in distribution
towards $T$ if:
\begin{equation}\label{eq:cv_law}
\forall h\in\N,\ \forall
\bt\in\T^{(h)},\quad  \lim_{n\rightarrow+\infty }
\P\bigl(r_h(T_n)=\bt\bigr)
= \P\bigl(r_h(T)=\bt\bigr),
\end{equation}
and  writing $\text{dist}(T)$ for  the distribution of
the random variable $T$, we denote it by:
\[
\lim_{n\to\infty} \text{dist}(T_n)=\text{dist}(T).
\]

For $\bt\in\T$ and $x\in\cl_0(\bt)$, we consider a  convergence determining
 class of trees $\T(\bt,x)=\{\bt\circledast (\tilde{\bt},x): \tilde{\bt}\in \T\}$,
 where:
\[
\bt\circledast (\tilde{\bt},x)=\{u\in\bt\}\cup\{xv: v\in\tilde{\bt}\}
\]
is the tree obtained by grafting $\tilde{\bt}$ on the leaf $x$ of $\bt$.
We recall from~\cite[Lemma~2.1]{ad14a}, that if $(T_n)_{n\in \N}$ and $T$ are
$\T_0\cup \T_1$-valued random variables, then the sequence $(T_n)_{n\in \N}$
converges locally in distribution
towards $T$ if and only if for all  $\bt\in\T_0$ and $x\in\cl_0(\bt)$:
\[
  \lim_{n\rightarrow \infty } \P(T_n\in \T(\bt, x))= \P(T\in \T(\bt,
  x))
  \quad\text{and}\quad
 \lim_{n\rightarrow \infty } \P(T_n=\bt)= \P(T=\bt).
\]

\medskip

\subsection{BGW trees}

Let $p$  be a  probability distribution on  $\N$.  A  $\T$-valued random
variable  $\tau$  is a  BGW  tree  with  offspring distribution  $p$  if
$k_\emptyset(\tau)$ is distributed as $p$  and the branching property is
satisfied:       for      $n\in       \N^*$,      conditionally       on
$\{k_\emptyset(\tau) = n\}$, the subtrees $(S_1(\tau),\ldots,S_n(\tau))$
are independent and distributed as $\tau$.  We denote by $\ct_p$ the BGW
tree with offspring distribution $p$. For all finite tree $\bt\in \T_0$,
we have:
\begin{equation}
   \label{eq:loi-Tp}
  \P(\ct_p = \bt) = \prod_{u\in \bt} p(k_u(\bt))
  =\prod_{n\in \N} p(n)^{L_n(\bt)},
\end{equation}
with the convention that $0^0=1$.
When~\eqref{assumption-p} holds and $p$ is critical or sub-critical,
then a.s.\ $\ct_p$  is  finite (that is, $\ct_p\in \T_0$) and in this case~\eqref{eq:loi-Tp}
 completely characterizes
the distribution of $\ct_p$.

\subsection{Kesten's tree}
\label{sec:kesten}
Let   $p$    be   a   critical   probability    distribution   on   $\N$
  (and thus  $\mu(p)=1$) satisfying~\eqref{assumption-p}.   We denote  by
$p^*=(p^*(n)=np(n))_{n\in    \N}$    the    corresponding    size-biased
distribution.    The  so   called   Kesten's  tree,   $\ct_p^*$,  is   a
$\T_1$-valued random  tree defined as  the local limit  in distribution,
when $n$  goes to  infinity, of  a BGW tree  conditioned to  have height
larger than $n$:
\[
\lim_{n\to\infty} \text{dist}(\ct_p|\, H(\ct_p)=n)=\text{dist}(\ct^*_p).
\]
Informally, it is the skeleton of a
two-type BGW tree, where: individuals are of type
$\mathrm{s}$ (survivor) or $\mathrm{n}$ (normal);
the root is of type $\mathrm{s}$; each individual
of type  $\mathrm{s}$ has a random number of children with offspring
distribution $p^*$, all of them of type $\mathrm{n}$ but for one uniformly chosen
at random which is of type $\mathrm{s}$;  each individual
of type  $\mathrm{n}$ has a random number of children with offspring
distribution $p$, all of them of type $\mathrm{n}$. Its distribution is
completely characterized by $\P(\ct_p^*\in \T_1)=1$ and:
\begin{equation}
   \label{eq:loi-Tp*}
   \P(\ct_p^*\in \T(\bt,x))=\frac{\P(\ct_p=\bt)}{p(0)}
   \quad\text{for all}\quad
\bt\in\T_0\quad\text{and}\quad
x\in \cl_0(\bt).
\end{equation}

\section{Definition of the distribution $\tilde p_{\theta, \beta}$}
\label{sec:def-pqb}
  Let     $p$    be     a    probability     distribution    on     $\N$
  satisfying~\eqref{assumption-p}. Let $\ca=(A_j)_{j\in \II{1, J}}$, with
  $J\in \N^*$, be pairwise disjoint non-empty subsets of $\supp(p)$.
  Note    $   A_0$   the   complementary   of
$\cup _{j\in \II{1,  J}} A_j$ in $\supp(p)$. Notice $A_0$  may be empty.
For $\cj\subset \II{0, J}$, we set:
\[
  A_{\cj}=\bigcup _{j\in \cj} A_j.
\]

 For a probability
distribution  $q$ on $\N$, we  write:
\[
  q(\ca) = (q(A_1), \ldots, q(A_J))\in [0, 1]^J.
\]
For a finite tree $\bt$, we write:
\[
  L_\ca(\bt) = (L_{A_1}(\bt), \ldots, L_{A_J}(\bt))\in \N^J.
\]

We  extend  the  definition  of generic  probability  distribution  with
respect  to a  subset of  $\N$  in \cite{ad14a}  to generic  probability
distribution with respect to a family of subsets.

\begin{defi}[Generic probability distribution]
  \label{defi:on-p}
  Let     $p$    be     a    probability     distribution    on     $\N$
  satisfying~\eqref{assumption-p}. Let $\ca=(A_j)_{ j\in \II{1, J}}$, with
  $J\in \N^*$, be pairwise disjoint non-empty subsets of $\supp(p)$.
\begin{propenum}		
\item\textbf{$(p,\ca)$-compatible probability distribution.}
    \label{defi:equiv}
 We say that a probability distribution  $p'$ is $(p,\ca)$-compatible
    if $p'$ satisfies~\eqref{assumption-p}, $\supp(p')\subset
    \supp(p)$,   and for all $\bn \in \N^J$ such that
    $\P(      L_{\ca}(\ct_{p'})     =     \bn)>0$      (and     thus
$\P( L_{\ca}(\ct_p) = \bn)>0$), we have:
    \begin{equation}
      \label{equi-cond}
\dist\left(\ct_p\, \big|\,  L_{\ca}(\ct_p) = \bn\right)
= \dist\left(\ct_{p'}\, \big|\,  L_{\ca}(\ct_{p'}) = \bn\right).
\end{equation}

\item\textbf{Generic distribution.}
    \label{defi:generic}
    We say that  $p$ is generic for $\ca$ in  the direction
    $\alpha\in \Delta_J$ if
    there exists a  \emph{critical} $(p, \ca)$-compatible probability
    distribution $p'$ such that:
     \[
       p'(\ca)= \left(1-p'( A_0)\right) \, \alpha
       \quad   \text{with}\quad
       p'(A_0)<1.
    \]

 \end{propenum}

\end{defi}

\begin{rem}[On the one-dimensional case]
  \label{rem:dim=1}
  When $J=1$, Definition \ref{defi:on-p} \ref{defi:generic} reduces to the definition
  of generic probability distribution with respect to the set
  $A_1\subset \N$
 with $\alpha=1$.
\end{rem}

\begin{rem}[Positivity of $\alpha$]
  \label{rem:0inAj}
  If $p$ is generic for $\ca$ in the direction $\alpha\in \Delta_J$ with
  $0\in   A_{j}$   for   some   $j\in   \II{1,   J}$,   then   we   have
  $\alpha_j>0$.   Indeed,  the   probability   distribution  $p'$   from
  Definition~\ref{defi:on-p}~\ref{defi:generic} is $(p, \ca)$-compatible
  and thus $ p'(A_j)= p'(A_0^c) \, \alpha_j$ with $ p'(A^c_0) >0$; since
  $p'$      also satisfies~\eqref{assumption-p},      we     deduce      that
  $p'(A_j)\geq p'(0)>0$ which then  gives $\alpha_j>0$.
\end{rem}

We now introduce a set of parameters which will allow us to describe all
the compatible probability distributions. We set:
\[
  \cj_{\infty }=\{j\in \II{1, J}\, \colon\, \sup A_j <\infty \}.
\]

\begin{defi}[The set of parameters $\param(p, \ca)$]
  \label{defi:p-qb}
   Let $p$ be a probability distribution on $\N$
  satisfying~\eqref{assumption-p}. Let $\ca=(A_j)_{j\in \II{1, J}}$,
  with $J\in \N^*$, be pairwise disjoint non-empty subsets of $\supp(p)$.
 For $\beta=(\beta_1, \ldots, \beta_J)\in  \R_+^J$, we set:
 \begin{equation}
   \label{eq:def-cj}
\cj^*_\beta=\{j\in \II{1, J}\, \colon\, \beta_j>0\}
   \quad\text{and}\quad
   \cj_\beta=\{0\}\cup \cj^*_\beta.
  \end{equation}
The set $\param(p, \ca)\subset [0, +\infty ]\times \R^J_+$ of
parameters is defined as follows.

\begin{enumerate}[(i)]
   \item \textbf{Non degenerate case.}
 The parameter
  $(\theta, \beta)\in \R_+^* \times \R^J_+$ belongs to $\param(p, \ca)$ if:
   \begin{equation}
   \label{eq:def-cond-prob}
    \sum_{j\in \cj_\beta} \beta_j\,  g_{A_j}(\theta)=1
    \quad\text{where}\quad
    \beta_0=\theta^{-1 }.
  \end{equation}

\item \textbf{Degenerate case.}
The parameter $(\theta, \beta)$, with $\theta\in \{0, +\infty \}$
 and $\beta\in  \R^J_+$, belongs to $\param(p, \ca)$ if:
   \begin{equation}
   \label{eq:def-cond-prob-0}
    \sum_{j\in \cj_\beta} \beta_j=1
\quad\text{where}\quad
\beta_0=p(1)\, \ind_{\{1\in A_0\}},
 \end{equation}
 and:
 \begin{align}
   \label{eq:def-cond-struct-0}
&\text{if $\theta=0$, then}  \quad   0\not\in A_0 ,\\
   \label{eq:def-cond-struct-infini}
&   \text{if $\theta=+\infty $, then}\quad
   A_0\subset \{0,1\}
   \quad\text{and}\quad
   \cj^*_\beta\subset   \cj_{\infty }.
   \end{align}
\end{enumerate}
\end{defi}

For simplicity, we shall write $\cj^*$ and $\cj$ for $\cj^*_\beta$ and
$\cj_\beta$. Notice that $\cj^*$ might be empty in the non degenerate
case, and that $\cj^*$ is non empty in the degenerate cases thanks
to~\eqref{assumption-p}.

We now define the family of probability distribution indexed by the
parameter $(\theta, \beta)$.

\begin{defi}[Probability distribution $\tilde{p}_{\theta,\beta}$]
  Let     $p$    be     a    probability     distribution    on     $\N$
  satisfying~\eqref{assumption-p}. Let $\ca=(A_j)_{ j\in \II{1, J}}$, with
  $J\in  \N^*$, be  pairwise disjoint  non-empty subsets  of $\supp(p)$.
  For  $(\theta, \beta)\in  \param(p, \ca)$,  we define  the probability
distribution
  $\tilde{p}_{\theta,\beta}=(\tilde{p}_{\theta, \beta}(n))_{  n\in \N}$ as
  follows, where the unspecified probabilities are set  to be 0.
\begin{enumerate}[(i)]
   \item \textbf{Non degenerate case.} If $\theta\in \R_+^*$, then we set:
\begin{equation}
   \label{eq:def-pt}
 \boxed{ \tilde{p}_{\theta, \beta} (n)=\beta_j \,  \theta^{n}p(n)}
  \quad\text{for  $j\in \cj$ and $n\in A_j$.}
\end{equation}
\item \textbf{Degenerate case at 0.} If $\theta=0$, then we set:
\begin{equation}
   \label{eq:def-pt0}
 \boxed{ \tilde{p}_{0, \beta} (n)=\beta_j }
  \quad\text{for  $j\in \cj$ and $n=\min  A_j$.}
\end{equation}
\item \textbf{Degenerate case at infinity.} If $\theta= +\infty $,
then we set:
\begin{equation}
   \label{eq:def-pt-infty}
 \boxed{ \tilde{p}_{\infty, \beta} (n)=\beta_j }
  \quad\text{for  $j\in \cj$ and $n=\max A_j$.}
\end{equation}
\end{enumerate}
\end{defi}

Conditions~\eqref{eq:def-cond-prob} and
~\eqref{eq:def-cond-prob-0}         insures           that
$\tilde{p}_{\theta,\beta}$   is     indeed       a   probability    distribution.
In the next remark, we consider some particular
cases of the probability
distributions $\tilde{p}_{\theta, \beta}$.  For
convenience,   for the BGW trees,              we                write:
\[
  \ct_{\theta,  \beta}=\ct_{\tilde{p}_{\theta,  \beta}}.
\]
Recall that $g_{A_0}=0$  if $A_0=\emptyset$. We shall  consider the
equation $ g_{A_0}(\theta)=\theta$ on $\R_+$ which has at least one root
and at most two. It is elementary to check that:
\begin{equation}
  \label{eq:def-qmin0}
  \theta_{\min}=\min \{\theta\in \R_+\, \colon\,
g_{A_0}(\theta)=\theta\}\in [0, 1),
\end{equation}
and that the second  root, if it exists, belongs to  $(1, +\infty )$. It
is elementary to check the following result.
\begin{lem}
   \label{lem:q-min0} We have $\theta_{\min}=0$ if and only if
   $0\not\in A_0$.
\end{lem}

\begin{rem}[On particular cases]
  \label{rem:theta=0}
Recall  $\un$ (resp. $\zero$) denotes the vector of $\R^J$ with all its coordinates  equal to
1 (resp. 0).
  \begin{enumerate}[(a)]
 \item \textbf{The case $\theta=1$ and $\beta=\un$.} We trivially have $\tilde
   p_{(1, \un)}=p$, and thus $(1, \un)\in \param(p, \ca)$. In
   particular, we get
  $\ct_p=\ct_{1,\un}$.

\item  \label{item:q=0} \textbf{The  case $\theta=0$.}   The support  of
  $\tilde{p}_{0,\beta}$             is              equal             to
  $\{1\}\cup\{   \min   A_j\,  \colon\,   j\in   \II{1,   J}\}$  or   to
  $\{ \min A_j\,  \colon\, j\in \II{1, J}\}$ according  to $1$ belonging
  to $A_0$ or not.

\item \label{item:q=infty}
\textbf{The  case $\theta=\infty $.}   The support  of
$\tilde{p}_{\infty ,\beta}$             is              equal             to
$\{1\}\cup\{  \max  A_j\,  \colon\,  j\in \cj_\infty \}$ or
$\{  \max  A_j\,  \colon\,  j\in \cj_\infty \}$ according  to $1$ belonging
  to $A_0$ or not.
\item\label{item:b=0}     \textbf{The    case     $\beta=\zero$.}     If
  $(\theta, \beta)\in \param(p, \ca)$, then  we have that $\beta=\zero $
  is equivalent to $  \tilde{p}_{\theta,\zero}(A_0^c)=0$.  In this case,
  we   have   $\theta\in   (0,   +\infty   )$
  (as~\eqref{assumption-p}, which
  implies $p(1)<1$,
  and~\eqref{eq:def-cond-prob-0}       rule      out       the      case
  $\theta\in  \{0,  +\infty  \}$)
  and   that  $\theta$  is  a  root  of
  $g_{A_0}(\theta)=\theta$ by~\eqref{eq:def-cond-prob}.    Thus   $\theta$   can  take   the   value
  $\theta_{\min}$ (only  if $\theta_{\min}>0$ and then  $0\in A_0$), and
  possibly another value,  say $\theta_M\in (1,  +\infty )$.   We also
  have  that $\mu(\tilde  p_{\theta, \zero})=g'_{A_0}(\theta)$,  so that
  $\tilde{p}_{\theta_{\min} ,\zero}$ is  sub-critical and, if $\theta_M$
  exists, $\tilde{p}_{\theta_M ,\zero}$ is super-critical.
  In conclusion, for $ \tilde{p}_{\theta,\zero}$ not to be
  super-critical, we need $0 \in A_0$ and $\theta=\theta_{\min}$.

 \end{enumerate}
\end{rem}

In order (partially) remove  the particular case $\beta=\zero$, we set:
\begin{align}
   \label{eq:def-Pa**}
  \param^{**}(p, \ca)
  &=\{(\theta, \beta)\in \param(p, \ca)\, \colon\,\beta\neq\zero\},\\
  \param ^*(p, \ca)
  &=
    \begin{cases}
      \param^{**}(p, \ca)   &\text{if $0\not\in A_0$},\\
    \param^{**}(p, \ca)\, \cup\{(\theta_{\min}, \zero)\}  & \text{if $0\in
      A_0$}.
  \end{cases}
\end{align}
The introduction of the set $ \param ^*(p, \ca)$ is
motivated by the following result.

\begin{lem}[Conditional finiteness of $\ct_{\theta, \beta}$]
  \label{lem:TinT0}
  Let $(\theta, \beta)\in \param(p, \ca)$.
  We have:
  \[
    \P(\ct_{\theta, \beta}\not \in \T_0,\, \sup_{j\in \cj^*}
    L_{A_j}(\ct_{\theta,
      \beta})<+\infty )=0
    \quad\text{if and only if}\quad
    (\theta, \beta)\in \param^*(p, \ca).
  \]
\end{lem}

\begin{proof}
  If $\beta\neq  \zero$, then  there exists $j\in  \II{1, J}$  such that
  $\beta_j>0$.     This    gives    that     a.s.\    on    the    event
  $\{     \ct_{\theta,     \beta}\not     \in    \T_0\}$     we     have
  $L_{A_j}(\ct_{\theta, \beta})=+\infty  $
  and  thus
  $   \P(\ct_{\theta,   \beta}\not   \in   \T_0,\,   \sup_{j\in   \cj^*}
  L_{A_j}(\ct_{\theta, \beta})<+\infty )=0$.

  If $\beta=\zero$, we deduce from
  Remark~\ref{rem:theta=0}~\ref{item:b=0} that if $(\theta, \zero)\in
  \param(p, \ca)$ then we have either $\theta=\theta_{\min}>0$,
$(\theta, \zero)\in
  \param^*(p, \ca)$,  $\tilde p_{\theta, \zero}$ is
  sub-critical; or
  $\theta>\theta_{\min}$, $(\theta, \zero)\not \in
  \param^*(p, \ca)$ and  $\tilde p_{\theta, \zero}$ is
  super-critical. In the former case, we get $\P(\ct_{\theta, \zero}\not
  \in \T_0)=0$ and in the   latter case:
  \[
    \P(\ct_{\theta, \zero}\not \in \T_0,\, \sup_{j\in \cj^*}
    L_{A_j}(\ct_{\theta,
      \zero})<+\infty )= \P(\ct_{\theta, \zero}\not \in \T_0)>0.
  \]
  This gives the result.
\end{proof}

We shall now restrict our study to the case where $\tilde p_{\theta,
  \beta}$ is non trivial.

\begin{defi}[Compatible parameter]
  \label{defi:param-admiss}
  The parameter $(\theta,  \beta)$ is $(p, \ca)$-compatible if $(\theta,
  \beta)\in \param^*(p, \ca)$   and   the   probability   distribution
  $\tilde{p}_{\theta,   \beta}$ satisfies condition~\eqref{assumption-p}.
\end{defi}

We now characterizes the $(p, \ca)$-compatible parameters.

\begin{lem}[Characterization of the compatible parameters]
  \label{lem:tp-non-degenere}
  The   parameter   $(\theta,    \beta)\in \param^*(p, \ca)$   is
  $(p, \ca)$-compatible if and only if:
\begin{enumerate}[(i)]
\item\label{item:def-cond-struct-ndg} For $\theta\in (0, +\infty )$, we have $0\in A_\cj
  $ and $
      p\left(A_\cj\cap \{0, 1\}^c\right)>0$.
    \item\label{item:def-cond-struct-0-ndg}
      For $\theta=0$, we have $  0\in A_{\cj^*}$ and $
      \max_{j\in\cj^*}\min (A_j)>1$.
    \item \label{item:def-cond-struct-infini-ndg}
      For $\theta=+\infty $, we have $A_{j_0}= \{0\}$ for some $j_0\in
      \cj^*$ and $  \max_{j\in\cj^*}\max (A_j)>1$ (and $A_0\subset
      \{1\}$ and $\cj^*\subset \cj_\infty $).
\end{enumerate}
\end{lem}

\begin{proof}
By  construction, we have
  $\supp(\tilde  p_{\theta, \beta})\subset  \supp(p)$.  Since  $p(0)>0$,
  there  exists  $j_0\in  \II{0,  J}$ such  that  $0\in  A_{j_0}$.

  Let  $\theta\in  (0,  +\infty  )$.   The  condition  $0\in  A_\cj$  in
  Point~\ref{item:def-cond-struct-ndg}      is      equivalent      to
  $\beta_{j_0}>0$  and thus  to  $\tilde  p_{\theta, \beta}(0)>0$.   The
  condition   $p\left(A_\cj  \cap   \{0,   1\}^c\right)>0$  is   clearly
  equivalent  to $\tilde  p_{\theta, \beta}(\{0,  1\}^c)>0$.  Therefore,
  conditions   in~\ref{item:def-cond-struct-ndg}  are   equivalent  to
  $\tilde    p_{\theta,    \beta}$   satisfying    the    non-degeneracy
  condition~\eqref{assumption-p}.

  Let   $\theta=0$  (and   thus   $0\not  \in   A_0$).   The   condition
  $0\in   A_{\cj^*}$    in   Point~\ref{item:def-cond-struct-0-ndg}   is
  equivalent to $\beta_{j_0}>0$ and  to $\tilde p_{\theta, \beta}(0)>0$,
  as               $\min                A_{j_0}=0$.                Since
  $\{\min A_j\,  \colon\, j \in \cj^*\}\subset  \supp( \tilde p_{\theta,
    \beta}) \subset \{\min  A_j\, \colon\, j \in  \cj^*\} \cup\{1\}$, we
  deduce  that   $  \max_{j\in\cj^*}\min   (A_j)>1$  is   equivalent  to
  $\tilde   p_{\theta,   \beta}(\{0,    1\}^c)>0$.    Thus,   conditions
  in~\ref{item:def-cond-struct-0-ndg}       are      equivalent       to
  $\tilde p_{\theta, \beta}$ satisfying~\eqref{assumption-p}.

  Let   $\theta=+\infty   $.   The   conditions   $j_0\in   \cj^*$   and
  $A_{j_0}=\{0\}$      in Point~\ref{item:def-cond-struct-infini-ndg}      are
  equivalent   to  $\tilde   p_{\theta,  \beta}(0)>0$ (notice
  that, by~\eqref{eq:def-cond-prob-0},
  $\tilde p_{\theta, \beta}(A_0)>0$ if and only if $1\in A_0$, and
  then  $0\in A_0$ implies that $\tilde p_{\theta, \beta}(0)=0$).
 Eventually  the
  condition $  \max_{j\in\cj^*}\max (A_j)>1$ is also  clearly equivalent
  to     $\tilde     p_{\theta,    \beta}(\{0,     1\}^c)>0$.      Thus,
  conditions in~\ref{item:def-cond-struct-infini-ndg}   are  equivalent   to
  $\tilde p_{\theta, \beta}$ satisfying~\eqref{assumption-p}.
\end{proof}

\section{The  $(p,\ca)$-compatible probability
distributions}
\label{sec:compatible}

We identify all the  $(p,\ca)$-compatible probability
distributions.
Recall  that   $p$   is     a    probability     distribution    on     $\N$
  satisfying~\eqref{assumption-p} and  $\ca=(A_j)_{j\in \II{1, J}}$, with
  $J\in \N^*$, are pairwise disjoint non-empty subsets of $\supp(p)$.

  \begin{theo}[Characterization of the compatible probability distributions]
    \label{th:p,a-compatibility}
  Let    $    p$    be    a   probability    distribution    on    $\N$
  satisfying~\eqref{assumption-p}.  The probability distribution $p'$ is
  $(p,\ca)$-compatible if and only  if $p' = \tilde{p}_{\theta, \beta}$,
  for some $(p,\ca)$-compatible  parameter $(\theta,\beta)$. 
\end{theo}
As  being  $(p,\ca)$-compatible  implies by  definition  that $p$ is non-trivial (that is, Condition~\eqref{assumption-p} holds), we  shall only  consider
the probability distributions $\tilde  p_{\theta, \beta}$ satisfying the
conditions of Lemma~\ref{lem:tp-non-degenere}.  \medskip

The end  of this  section is devoted  to the proof  of this  result.  We
first prove  that $\tilde  p_{\theta, \beta}$  are $(p,\ca)$-compatible,
provided Condition~\eqref{assumption-p} holds:  see
Lemma~\ref{lem:pA-non-deg}   for   the    non-degenerate   cases   and
Lemma~\ref{lem:pA-deg} for  the degenerate  cases. Then we  prove that
all  $(p,\ca)$-compatible  probability  distributions are  of  the  form
$\tilde p_{\theta,  \beta}$, distinguishing  according to $0\in  A_0$ in
Lemma~\ref{lem:pab-A0}  or  $0\not  \in  A_0$ in
Lemma~\ref{lem:pab-A1}, where the proof of the latter  is  more  technical.

\medskip

We set $\be^0=\zero$ and for $j\in \II{1, J}$:
\[
     \be^j=(e_1^j, \ldots, e_J^j)\in \N^J
     \quad\text{with}\quad
     e^j_i=\ind_{\{i=j\}}.
\]
In  particular, we  have $\sum_{j=1}^J \be^j=\un$.   Notice that  $\P(\ct_p=\bt)>0$
implies  $\bt\in   \T_0$  and   $k_u(\bt)\in  A_{\II{0,  J}}$   for  all
$u\in \bt$.

\begin{lem}
   \label{lem:pA-non-deg}
   Under the  assumptions of Theorem~\ref{th:p,a-compatibility},  if the
   parameter   $(\theta,   \beta)$   is   $(p,   \ca)$-compatible   with
   $\theta\in  (0,   +\infty  )$,  then  the   probability  distribution
   $\tilde p_{\theta, \beta}$ is $(p, \ca)$-compatible.
\end{lem}
\begin{proof}
We suppose that the  parameter $(\theta, \beta)$ is $(p,\ca)$-compatible
(which   implies  that   $\tilde   p_{\theta,   \beta}$  is non-trivial)      and       that
$\theta\in (0, +\infty  )$.  We prove that  the probability distribution
$\tilde    p_{\theta,   \beta}$    is   $(p,\ca)$-compatible.     Recall
$\cj=\{0\}\cup\{j\in \II{1, J}\, \colon\, \beta_j>0\}$.

Notice that $\P(L_{\ca}(\ct_{\theta, \beta}) = \bn)>0$  implies that
 $n_j=0$ for all $j\not\in \cj^*$, where $\bn=(n_1, \ldots, n_J)$.
Using the definition of
 $\tilde{p}_{\theta, \beta}$, we obtain  for
 $\bt\in \T_0$ such that $\P(\ct_p=\bt)>0$ and $L_{A_j}(\bt)=0$ for $j\not\in \cj$ that:
\begin{align*}
\P(\ct_{\theta, \beta} = \bt)
&= \prod_{u\in \bt} \tilde{p}_{\theta,\beta}(k_u(\bt))\\
 &= \prod_{j\in\cj} \, \prod_{u \in \cl_{A_j}(\bt)} \beta_j
 \theta^{k_u(\bt)} p(k_u(\bt)) \\
&= \P(\ct_p= \bt)\,  \theta^{\sharp \bt -1}
                       \prod_{j\in\cj}\beta_j^{L_{A_j}(\bt)}\\
  &= \P(\ct_p= \bt)\, \theta^{-1}  \prod_{j\in\cj^*}
    (\beta_j \theta )^{L_{A_j}(\bt)}.
\end{align*}
where we used~\eqref{eq:kt=t-1}
 for the third equality, and  $\sum_{j\in\cj}  L_{A_j}(\bt)= \sharp \bt$
 as well as $\beta_0 \theta=1$ for the fourth equality.
 As $\P(\ct_p=\bt)=0$ implies
$\P(\ct_{\theta, \beta}=\bt)=0$, we deduce that for $\bn\in \N^ J$:
\begin{align*}
 \P(L_{\ca}(\ct_{\theta, \beta}) = \bn)
  &= \sum_{\bt\in \T_0,\, L_{\ca}(\bt) = \bn} \P(\ct_{\theta, \beta}=\bt)
  + \P(\ct_{\theta, \beta}\not \in \T_0, \, L_{\ca}(\ct_{\theta, \beta}) = \bn)  \\
 &= \sum_{\bt\in \T_0,\, L_{\ca}(\bt) = \bn} \P(\ct_{\theta, \beta}=\bt)\\
&= \P\left(L_{\ca}(\ct_p) = \bn\right)\,  \theta^{-1} \prod_{j\in \cj^*}
(\beta_j\theta)^{n_j},
\end{align*}
where we used Lemma~\ref{lem:TinT0}
for the second equality.

So for every $\bn\in \N^J$
such that $\P(L_{\ca}(\ct_{\theta, \beta}) = \bn)>0$, we have
$\P(L_{\ca}(\ct_p) = \bn)>0$, and  for every
$\bt\in\T_0$ such that $L_{\ca}(\bt)=\bn$, we get:
\[
\P(\ct_{\theta, \beta}=\bt\,|\, L_{\ca}(\ct_{\theta, \beta})=\bn)
=\P(\ct_p=\bt\, |\,  L_{\ca}(\ct_p)=\bn).
\]
Hence, the probability distribution   $\tilde p_{\theta, \beta}$ is
$(p, \ca)$-compatible.
\end{proof}

\begin{lem}
   \label{lem:pA-deg}
Under the assumptions of Theorem~\ref{th:p,a-compatibility},
if the parameter $(\theta, \beta)$ is $(p,
  \ca)$-compatible with $\theta\in \{0, +\infty\}$,
then the probability distribution $\tilde p_{\theta, \beta}$ is $(p,
  \ca)$-compatible.
\end{lem}
\begin{proof}
  We first consider  the case $\theta=+\infty $.   For simplicity, write
  $p'=\tilde p_{\infty  , \beta}$  and $\ct'=\ct_{\infty ,  \beta}$.  We
  suppose    that    the    parameter     $(\infty    ,    \beta)$    is
  $(p,\ca)$-compatible.   As    $p'$ is non-trivial,          we          get          that
  $\beta\in \R^J_+\backslash\{\zero\}$  and thus  $\cj^*\neq \emptyset$.
  In    particular,    we     have    $A_0\subset\{1\}$,    thanks    to
  \eqref{eq:def-cond-struct-infini}                                  and
  Lemma~\ref{lem:tp-non-degenere}~\ref{item:def-cond-struct-infini-ndg}.
  We    prove    that    the   probability    distribution    $p'$    is
  $(p,\ca)$-compatible.

Notice that $\P(L_{\ca}(\ct') = \bn)>0$  implies that
$n_j=0$ for all $j\not\in \cj^*$, where $\bn=(n_1, \ldots, n_J)$,  and,
by~\eqref{eq:kt=t-1}, that:
\begin{equation}
  \label{dege-infinity}
\sum_{j\in \cj^*}n_j(\max A_j-1)=-1.
\end{equation}
To simplify notations, recall we write $L_k$ for $L_{\{k\}}$ for $k\in \N$.
Fix such $\bn\in \N^ J$ and consider the set $\T_{0, n}=\{ \bt\in \T_0\,
\colon\,L_{\ca}(\bt)=\bn \quad\text{and}\quad
\P(\ct'=\bt)>0\}$, which is clearly not empty. Using~\eqref{eq:kt=t-1},
we             get,      for $\bt\in \T_{0, n}$,       that
$ \sum_{j\in \cj^*} \sum_{k\in  A_j} L_{k}(\bt) (k-1)=-1$. We deduce
from~\eqref{dege-infinity}   that  for   $k\not   \in   A_0$  we   have:
$L_{k}(\bt)=n_j$   if    $j\in   \cj^*$ and  $k=\max A_j$,   and
$L_{k}(\bt)=0$ otherwise.

\medskip

Let $\bt\in \T_{0, n}$.
We distinguish two cases according to $1$ belonging to $A_0$ or not.
First, we consider the case
$1 \in A_0$, that is $A_0=\{1\}$, elementary computation gives:
 \[
    \P(\ct'=\bt \, |\,  L_{\ca}(\ct')=\bn)
    = c^{-1} \, p(1)^{L_{1}(\bt)},
  \]
  with, as $\cj^*\neq \emptyset$, $c$ positive and finite  given by:
  \[
    c=\sum_{\bt' \in \T_0 \, \text{ s.t.}\
        L_\ca(\bt')=\bn} p(1)^{L_{1}(\bt')}
+ \P(\ct'\not \in \T_0,  L_{\ca}(\ct')=\bn)
      =\sum_{\bt' \in \T_0 \, \text{ s.t.}\
        L_\ca(\bt')=\bn} p(1)^{L_{1}(\bt')}.
    \]
   Similarly, we have:
 \[
    \P(\ct_p=\bt \, |\,  L_{\ca}(\ct_p)=\bn)
    = c^{-1} \, p(1)^{L_{1}(\bt)}.
  \]
This readily implies that $p'$ is $(p, \ca)$-compatible.
\medskip

Secondly, we consider the case
   $1\not\in A_0$, that is    $A_0=\emptyset$. We       get   that the set
$\T_{0, n}$ is finite. Similarly to the first case, we obtain,       with
$c=\Card(\T_{0, n})\geq 1$, that:
\[
    \P(\ct'=\bt \, |\,  L_{\ca}(\ct')=\bn)
    =  \P(\ct_p=\bt \, |\,  L_{\ca}(\ct_p)=\bn)= c^{-1} .
\]
This readily also implies that $p'$ is $(p, \ca)$-compatible.

\medskip

Eventually, the case $\theta=0$ can be handled  similarly using
$  \sum_{j\in \cj^*}n_j(\min A_j-1)=-1$ instead of~\eqref{dege-infinity}.
\end{proof}

\begin{lem}
   \label{lem:pab-A0}
   Under the  assumptions of Theorem~\ref{th:p,a-compatibility},  if the
   probability   distribution   $p'$   is  $(p,   \ca)$-compatible   and
   $0\in  A_0$, then  we  have $p'=\tilde  p_{\theta,  \beta}$ for
   some    $(p,                    \ca)$-compatible
   parameter $(\theta,\beta)$.
\end{lem}

\begin{proof}
  For $k\in \N$, let $\bt_k$ denote the
tree with the root having $k$ children, all of them being leaves (that
is: $k_\roott(\bt_k)=k$ and $\sharp \bt_k=k+1$).

We assume  that $0\in A_0$  and that $p'$  is $(p, \ca)$-compatible. For
simplicity, we write $\ct'$ for  $\ct_{p'}$.  We have $p(0)>0$ and $p'(0)>0$
since    $p$   and    $p'$      satisfy~\eqref{assumption-p}.    Let
$j\in \II{0,  J}$ be such  that $p'(A_j)>0$. There  exists $k_j\in  A_j$ such
that $p'(k_j)>0$. We have:
\[
  \P(L_\ca(\ct')=\be^j)\ge
   \P(\ct'=\bt_{k_j})=p'(k_j)p'(0)^{k_j}>0.
\]
This also implies that $ \P(L_\ca(\ct_{p})=\be^j)>0$ as
$\supp(p')\subset \supp(p)$.
In particular, thanks to  Equation \eqref{equi-cond}, we have for  any
$k\in A_j$:
\[
\P(\ct'=\bt_k\bigm| L_\ca(\ct')=\be^j)  =\P(\ct_{p}=\bt _ k\bigm|
L_\ca(\ct_{p})=\be^j).
\]
This gives:
\[
  \frac{p'(k)p'(0)^k}{\P(L_\ca(\ct')=\be^j)}
  =\frac{p(k)p(0)^k}{\P(L_\ca(\ct_{p})=\be^j)},
\]
that is:
\begin{equation}
   \label{eq:p=qbp}
  p'(k)=\beta_j \theta^k p(k),
\end{equation}
with                      $\theta=p(0)/p'(0)>0$                      and
$\beta_j=\P(L_\ca(\ct')=\be^j)/\P(L_\ca(\ct_{p})=\be^j)>0$.     Notice
that $\theta$ and  $\beta_j$ does not depend on $k\in  A_j$.  For $j=0$,
we  have  $p'(A_0)>0$   as  $p'(0)>0$.  For  $k=0\in   A_0$,  we  deduce
from~\eqref{eq:p=qbp}  that $\beta_0=p'(0)/p(0)=1/\theta$.  For $j\in  \II{0, J}  $
such  that   $p'(A_j)=0$,  Equation~\eqref{eq:p=qbp}  also   holds  with
$\beta_j=0$. Notice that when $\beta=\zero$, we have $\P(L_\ca(\ct')=\zero)=1$
and $\P(L_\ca(\ct_p)=\zero)< 1,$ which entails that
$\beta_0=\P(L_\ca(\ct')=\zero)/\P(L_\ca(\ct_{p})=\zero)>1$, or
equivalently, $\theta\in (0, 1)$.
This proves that  $p'=\tilde p_{\theta, \beta}$, the latter
being       defined        in~\eqref{eq:def-pt} as $\theta\in (0, +\infty )$.

To conclude, notice that Condition~\eqref{eq:def-cond-prob} holds as $p'$
 is a probability distribution and $\beta_0=1/\theta$.
\end{proof}

\begin{lem}
   \label{lem:pab-A1}
   Under the  assumptions of Theorem~\ref{th:p,a-compatibility},  if the
   probability   distribution   $p'$   is  $(p,   \ca)$-compatible   and
   $0\not \in  A_0$, then  we  have $p'=\tilde  p_{\theta,  \beta}$ for
 some   $(p,                    \ca)$-compatible
   parameter $(\theta,\beta)$ .

\end{lem}
The proof of this  lemma is more technical and relies on the following
result whose proof is postponed at the end of this section.
We introduce the sets:
\begin{align*}
\cj^* & =\{j\in\lb 1,J\rb\,\colon\, p'(A_j)>0\},\\
  \cj^ {**} & =\{j\in \cj^*\, \colon\, \Card(A_j)\geq 2\}.
\end{align*}
Notice that once Lemma~\ref{lem:pab-A1} is proved, then $\cj^*$ coincides with $\cj^*_\beta$ defined in~\eqref{eq:def-cj}. 

\begin{lem}
  \label{lem:technique}
 Assume $0\not\in A_0$ and let  $p'$  be  a  $(p, \ca)$-compatible  distribution.   We  have  the
  following properties.

  \begin{enumerate}[(i)]

  \item \label{item:tech-00}
The map $\ell \mapsto \left(p'(\ell)/p(\ell)\right)^{1/(\ell -1)}$ is
constant over $\{\ell\in A_0\, \colon\,  \ell \geq 2\}$.

\item \label{item:tech-j0}
  Assume there exists $\ell\in A_0$ such that $\ell\geq 2$   and $k', k\in A_j$ with
     $j\in\cj^{**}$ such that  $k'>k\geq 0$.
     Then we have:
     \[
  p'(k')>0
       \quad\Longleftrightarrow\quad
     p'(k)>0 \quad\text{and}\quad
       p'(\ell)>0.
\]
Furthermore if those conditions hold, then we also have, with $\alpha=
\ell-1$ and $\beta=k'-k$:
\begin{equation}
  \label{eq:pp-j0}
  \left(\frac{p'(k)}{p'(k')}\right)^\alpha \, p'(\ell) ^\beta
  =   \left(\frac{p(k)}{p(k')}\right)^\alpha \, p(\ell) ^\beta.
\end{equation}

\item \label{item:tech-jj}
  Assume there exist $i, j\in \cj^{**}$,
  $\ell , k_i\in A_i$ such that $\ell > k_i\geq 0$ and $p'(k_i)>0$, and
 $k_j, k \in A_j$
  such that  $k_j>k\geq 0$  and
  $p'(k_j)>0$. Then we have $p'(\ell)>0$ and $p'(k)>0$ as well as, with
  $\alpha=\ell - k_i$ and $\beta= k_j -k$:
\begin{equation}
  \label{eq:pp-jj}
  p'(k)^\alpha p'(\ell)^\beta =
  \left(\frac{p'(k_j)}{p(k_j)}\right)^\alpha
  \left(\frac{p'(k_i)}{p(k_i)}\right)^\beta \, p(k)^\alpha
  p(\ell)^\beta.
\end{equation}
\end{enumerate}
\end{lem}

\begin{proof}[Proof of Lemma~\ref{lem:pab-A1}]
  We assume  that $0\not \in A_0$, that is without loss  of generality,
  $0\in  A_1$,   and that $p'$  is $(p, \ca)$-compatible.
Since $p$ and $p'$ satisfy~\eqref{assumption-p}, we get $p(0)>0$ and
$p'(0)>0$. So, we can define $\beta_1=p'(0)/p(0)>0$.  By considering trees
with vertices having zero or  one child, we get:
\begin{equation}
   \label{eq:pp'1}
  1\in A_0
 \, \Longrightarrow \, p'(1)=p(1).
\end{equation}
Recall the set $\cj^*=\{j\in \II{1,
  J}\, \colon\, p'(A_j)>0\}$. We set $\beta_j=0$ for $j\in \II{1, J}
\setminus \cj^*$ in accordance with the definition of the $\tilde
p_{\theta, \beta}$. Notice that $\cj^*$ is non empty as $1\in \cj^*$.
For $j\in \cj^*$, let $k_j$ be an element of $A_j$ such that
$p'(k_j)>0$.
For $j\in \cj^*\setminus \cj^{**}$,  we have $A_j=\{k_j\}$ and, whatever
the  value  of   $\theta$  which  will  be  defined  later   on,  we  set
$\beta_j=  \theta^{-k_j} p'(k_j)/  p(k_j)$ if $\theta\in (0, +\infty )$
and $\beta_j=  p'(k_j)/  p(k_j)$ otherwise; this is
also  in accordance  with the
definition of the $\tilde p_{\theta, \beta}$. We now have to consider
the value of $p'(\ell)$ for $\ell\in A_j$ and $j\in \{0\} \cup
\cj^{**}$.

\medskip We first consider the case $\cj^{**}=\emptyset$.  In particular
we    get    that   $A_1=\{0\}$.     Then    we    have   either    that
$p'(A_0)=p'(A_0\cap\{1\})$ or there exists  $k_0 \in A_0$
such that  $k_0\geq 2$ and $p'(k_0)>0$.  In  the former case,  using~\eqref{eq:pp'1} and
$\beta_0=p(1)\, \ind_{\{1\in  A_0\}}$, the probability  distribution can
be  written as  a $\tilde  p_{\theta, \beta}$  with $\theta=0$.   In the
latter    case,    using    Lemma~\ref{lem:technique}~\ref{item:tech-00}
and~\eqref{eq:pp'1},  we  deduce  that~\eqref{eq:def-pt} holds  for  all
$\ell\in   A_0$  with   a   common  $\theta\in   (0,   +\infty  )$ (given by the constant value of the map in   Lemma~\ref{lem:technique}~\ref{item:tech-00}) and
$\beta_0=1/\theta$;  thus  the  probability  distribution  $p'$  can  be
written as a $\tilde p_{\theta, \beta}$ with $\theta\in (0, +\infty )$.

\medskip We now assume that $\cj^{**}\neq \emptyset$. From
Lemma~\ref{lem:technique}~\ref{item:tech-jj}, only three
cases are possible:
\begin{enumerate}[(a)]
\item\label{item:min}
  For all $j\in \cj^{**}$, we have
     $k_j=\min A_j$, $p'(k_j)>0$ and $p'(A_j\setminus \{k_j\})=0$.
   \item\label{item:max}
     For all  $j\in \cj^{**}$, we have   $k_j=\max A_j$, $p'(k_j)>0$ and $p'(A_j\setminus \{k_j\})=0$.
   \item \label{item:all}
     For all  $j\in \cj^{**}$, we have  $p'(k)>0$ for all $k\in A_j$.

\end{enumerate}

We shall  investigate each  case separately. In  case~\ref{item:min}, we
deduce     from    Lemma~\ref{lem:technique}~\ref{item:tech-j0}     that
$p'(\ell)=0$   for   all  $\ell\in A_0$ with $\ell\geq   2$.    Thus,
using~\eqref{eq:pp'1}, we get that the probability distribution  $p'$ can be written as a
$\tilde p_{\theta, \beta}$  with $\theta=0$.

\medskip

In  case~\ref{item:max}, since $p'(0)>0$, we deduce that $A_1=\{0\}$ and
that if $j$ belongs to $\cj^{**}$ (and thus to $\cj^*$) then $\sup A_j$
is finite. Then use   Lemma~\ref{lem:technique}~\ref{item:tech-j0}  to
deduce that there is no element $\ell\geq 2$ in $A_0$, that is
$A_0\subset \{1\}$.  Thus,
using~\eqref{eq:pp'1}, we get that the probability distribution  $p'$  can be written as a
$\tilde p_{\theta, \beta}$  with $\theta=+\infty $.

\medskip

In case~\ref{item:all}, to fix ideas, let $i=\min \cj^{**}$ and consider
$k_i=\min A_i$ and $\ell=\min A_i \setminus \{k_i\}$. This uniquely
determine $\beta_i $ and $\theta\in (0, +\infty )$ solution of, for
$k'\in \{k_i, \ell\}$:
\begin{equation}
  \label{eq:pk=bqpk}
p'(k')=
\beta_i \theta^{k'} p(k').
 \end{equation}
  Then for $k''\in A_i$ larger than $\ell$, use
Lemma~\ref{lem:technique}~\ref{item:tech-jj}, and in
particular~\eqref{eq:pp-jj},  with $j=i$, $k_j=k''$ and $k=\ell$
to deduce that $p'(k'')$ can also be written as in~\eqref{eq:pk=bqpk}.
For  $j\in \cj^{**}$ with $j\neq i$, set $k=\min A_j$ and define
$\beta_j$ by:
\begin{equation}
  \label{eq:pk=bqpk2}
p'(k)=
\beta_j \theta^{k} p(k).
 \end{equation}
 Then    use   Lemma~\ref{lem:technique}~\ref{item:tech-jj},    and   in
 particular~\eqref{eq:pp-jj}, with $k_j>k$ (and $k_j\in A_j$), to deduce
 that~\eqref{eq:pk=bqpk2}  holds for  $k$  replaced by  $k_j$. Then  for
 $\ell\geq 2$ in $A_0$, use Lemma~\ref{lem:technique}~\ref{item:tech-j0}
 to  get  that  $p'(\ell)=  \theta^{\ell -1}  p(\ell)$  (which  is  also
 consistent  with~\eqref{eq:pp'1}).   We  deduce  that  the  probability
 distribution $p'$ can  be written as a $\tilde  p_{\theta, \beta}$ with
 $\theta\in (0,+\infty)$.  This concludes the proof.
\end{proof}

\begin{proof}[Proof of Lemma~\ref{lem:technique}]
Without loss  of generality, we assume $0\in  A_1$.
  In the following description of trees we don't precise the number of
  leaves, as it is determined through
  Equation~\eqref{eq:kt=L0-1}. The argument is based on considering two
  well chosen trees $\bt$ and $\bt'$ such that
  $L_\ca(\bt)=L_\ca(\bt')$. Recall that we write $L_k$ for $L_{\{k\}}$
  and $k\in \N$.
  \medskip

  We prove Point~\ref{item:tech-00}.
 Assume there exist $k,k'\in A_0$ such that $\min(k, k')\geq 2$. Consider a tree $\bt$
 having only leaves and $\alpha=k'-1>0$ vertices
   of out-degree $k$ (that is $L_{k}(\bt)=\alpha$) and a tree $\bt'$
   having only leaves and $\beta=k-1>0$  vertices of out-degree $k'$ (that is
   $L_{k'}(\bt')=\beta$).
Set $a=1+(k'-1)(k-1)> 1$. Thanks to~\eqref{eq:kt=L0-1}, we get $
L_{0}(\bt)=  L_{0}(\bt')=a$,
and thus:
\[
  L_\ca(\bt)=L_\ca(\bt')= \bn
  \quad\text{with}\quad \bn=a\, \be^1.
\]

Assume that $p'(k)> 0$. We have $\P(L_\ca(\ct_{p'})=\bn)\ge
\P(\ct_{p'}=\bt)=p'(k)^\alpha p'(0)^a>0$, and thus
$\P(L_\ca(\ct_{p})=\bn)>0$.
We set:
\[
c=\frac{\P(L_{\ca}(\ct_{p'})=\bn)}{\P(L_{\ca}(\ct_{p})=\bn)}>0.
\]
Then, we apply ~\eqref{equi-cond} to the trees $\bt$ and $\bt'$ to get:
\[
    p'(k)^\alpha\,  p'(0)^{a} =c p(k)^\alpha\,  p(0)^{a}
  \quad\text{and}\quad
   p'(k')^\beta\,  p'(0)^{a} =c p(k')^\beta\,  p(0)^{a}.
\]
This readily implies  that $p'(k')$ is positive (as  $k'\in A_0$ implies
that $p(k')>0$) and that
$\left(p'(k)/p(k)\right)^{1/(k -1)}=\left(p'(k')/p(k')\right)^{1/(k' -1)}$.
This gives Point~\ref{item:tech-00}.
\medskip

 We now  prove Point~\ref{item:tech-j0}.
  Assume there exist $k'>k\geq 0$ which are elements of $A_j$ with
     $j\in\cj^{**}$ and $\ell\geq 2$ which is an element of $A_0$.
 Notice  that $j$ can
  possibly  take the  value $1$  when $1\in  \cj^{**}$ (that is, $\Card(A_1)\geq 2$). Consider a tree $\bt$ having
    only leaves and $\alpha=\ell-1>0$ vertices of
     out-degree  $k'$  (that  is $L_{k'}(\bt)=\alpha$)  and  a  tree
     $\bt'$ having  only leaves,
     $\beta=k'-k>0$ vertices of out-degree $\ell$ and, if $k>0$, $\alpha$ vertices
     of out-degree  $k$ (that is $L_{\ell}(\bt')=\beta$ and, if $k>0$,
     $L_{k}(\bt')=
     \alpha$).

     \medskip
We first assume that $k\geq 1$  (which is automatically satisfied if $j\geq 2$).
Thanks to~\eqref{eq:kt=L0-1}, we get with $a=1+\alpha(k' -1)=1+
\alpha(k-1)+ \beta(\ell -1)$ that $  L_{0}(\bt)=  L_{0}(\bt')=a$,
and thus:
\begin{equation}
  \label{eq:La=La-j0}
L_\ca(\bt)=L_\ca(\bt')= \bn\quad\text{with}\quad \bn=a\, \be^1 + \alpha \be^j.
\end{equation}
Assume that $p'(k')>0$.
 We have $\P(L_\ca(\ct_{p'})=\bn)\ge
\P(\ct_{p'}=\bt)=p'(k')^\alpha p'(0)^a>0$, and thus
$\P(L_\ca(\ct_{p})=\bn)>0$.
We set:
\[
c=\frac{\P(L_{\ca}(\ct_{p'})=\bn)}{\P(L_{\ca}(\ct_{p})=\bn)}>0.
\]
Then, we apply ~\eqref{equi-cond} to the trees $\bt$ and $\bt'$ to get:
\[
  p'(k')^\alpha\,  p'(0)^{a} =c p(k')^\alpha\,  p(0)^{a}
  \quad\text{and}\quad
   p'(k)^\alpha\, p'(\ell)^\beta\,  p'(0)^{a} =c p(k)^\alpha\, p(\ell)^\beta\, p(0)^{a}.
\]
This readily implies  that $p'(k)$ and $p'(\ell)$ are  positive (as
$k\in A_j$ and $\ell\in A_0$ imply
that $p(k)$ and $p(\ell)$ are positive).
Similarly, assuming that $p'(k)$ and $p'(\ell)$ are  positive implies
that $p'(k')$ is positive. Notice also that~\eqref{eq:pp-j0} is
obvious.

\medskip  We now  consider  the  case $k=0$  and  thus  $j=1$ (as $0\in A_1$).  One  has
$L_{A_1}(\bt)=L_{0}(\bt)  + L_{k'}(\bt)=1+  \alpha(k'-1)+ \alpha=  1+
(\ell-1)                             k'$,                            and
$  L_{A_1}(\bt')=L_{0}(\bt')=1+\beta  (\ell  -1)= 1+  (\ell-1)  k'$,
which implies that~\eqref{eq:La=La-j0} still holds.
We then conclude similarly as in the case $k>0$.
This gives Point~\ref{item:tech-j0}.
\medskip

Eventually,  we  prove   Point~\ref{item:tech-jj}.  Assume  there  exist
$i, j\in \cj^{**}$, $k_j>k\geq 0$ which are elements  of $A_j$, $\ell > k_i\geq 0$
which are elements of $A_i$, with $p'(k_j)>0$ and $p'(k_i)>0$. Notice that $i$ and
$j$  can be  possibly equal  and can  possibly take  the value  $1$ when
$1\in   \cj^{**}$.   Consider  a   tree   $\bt$   having  only   leaves,
$\alpha=\ell-k_i>0$  vertices   of  out-degree   $k_j$  and,   if  $k_i>0$,
$\beta=k_j-k>0$     vertices    of     out-degree     $k_i$    (that     is
$L_{k_j}(\bt)=\alpha$ and, if $k_i>0$, $L_{k_i}(\bt)=\beta$) and
a tree  $\bt'$ having only  leaves, $\beta$ vertices of  out-degree $\ell$,
and,   if   $k>0$,  $\alpha$   vertices   of   out-degree  $k$   (that   is
$L_{\ell}(\bt')=\beta$ and, if $k>0$, $L_{k}(\bt')=\alpha$).

  \medskip

  We assume  that $k\geq 1$ and  $k_i\geq 1$, and leave  the cases $k=0$
  (and thus $j=1$) and/or $k_i=0$ (and  thus $i=1$) to the reader as the
  proof can be handled very similarly; see  also the end of the proof of
  Point~\ref{item:tech-j0}.
Thanks to~\eqref{eq:kt=L0-1}, we get with $a=1+\alpha(k_j -1)+ \beta (k_i-1)=1+
\alpha(k-1)+ \beta(\ell -1)$ that $  L_{0}(\bt)=
L_{0}(\bt')=a$
and thus:
\[
  L_\ca(\bt)=L_\ca(\bt')= \bn
  \quad\text{with}\quad \bn=a\, \be^1 + \alpha \be^j+ \beta \be^i.
\]
 We have $\P(L_\ca(\ct_{p'})=\bn)\ge
\P(\ct_{p'}=\bt)=p'(k_j)^\alpha \, p'(k_i)^\beta\, p'(0)^a>0$, and thus
$\P(L_\ca(\ct_{p})=\bn)>0$.
We set:
\[
c=\frac{\P(L_{\ca}(\ct_{p'})=\bn)}{\P(L_{\ca}(\ct_{p})=\bn)}>0.
\]
Then, we apply ~\eqref{equi-cond} to the trees $\bt$ and $\bt'$ to get:
\[
  p'(k_j)^\alpha\, p'(k_i)^\beta\,   p'(0)^{a} =c   p(k_j)^\alpha\,
  p(k_i)^\beta\,  p(0)^{a}
  \quad\text{and}\quad
   p'(k)^\alpha\, p'(\ell)^\beta\,  p'(0)^{a} =c p(k)^\alpha\, p(\ell)^\beta\, p(0)^{a}.
\]
This readily implies  that $p'(k)$ and $p'(\ell)$ are  positive (as
$k\in A_j$ and $\ell\in A_i$ imply
that $p(k)$ and $p(\ell)$ are positive).
Equation~\eqref{eq:pp-jj} is then obvious. This gives Point~\ref{item:tech-jj}.
\end{proof}

\section{Existence of a critical $(p, \ca)$-compatible distribution}
\label{sec:critical-p}
\subsection{Parametrization  of  the $(p,  \ca)$-compatible  probability
  distributions using their direction}

 We  define the direction of a
probability distribution $p'$ with respect to $\ca$.

\begin{defi}[Direction of compatible probability distributions]
  \label{defi:direction}
 The
  direction $\alpha\in \Delta^J$ of  a
  $(p,  \ca)$-compatible  probability distribution $p'$ such that
  $p'(A_0)<1$ is defined by:
\begin{equation}
   \label{eq:direction}
      \alpha= a^{-1} \, p'(\ca)
  \quad\text{with}\quad
  a=1-p'(A_0)>0.
  \end{equation}
\end{defi}
Because by definition  $p'$ is non trivial, see~\eqref{assumption-p},
we shall see below that the set of possible directions is:
\begin{equation}
   \label{eq:possible-dir}
  \Delta_J^*=\Delta_J\backslash \Delta_J^o,
\end{equation}
where the set of ineligible directions are given by:
\[
  \Delta_J^o=\bigcup_{j=1}^J \Big\{\alpha\in\Delta_J: \alpha_j=0
  \,\text{ if }\,
   0\in A_j
   \quad\text{or}\quad
   \alpha_j=1  \,\text{ if }\, A_0\cup A_j\subset \{0,1\}\Big\}.
\]

We shall  use a parametrization  of the probability distribution
$\tilde  p_{\theta,\beta}$ using the parameter $\theta$ and its
direction $\alpha\in \Delta_J^*$. Recall that
$\tilde p_{\theta,\beta}(A_0)=1$  if and only if  $\beta=\zero$,
see Remark~\ref{rem:theta=0}~\ref{item:b=0}. Recall the set of
parameters $\param^{**}(p,
\ca)$ defined in~\eqref{eq:def-Pa**}, where the cases $\beta=\zero$ are
removed,  and the set of $(p,
\ca)$-compatible  parameters given in
Definition~\ref{defi:param-admiss}.
The direction map $\cd: (\theta, \beta)\mapsto (\theta, \alpha)$ given
by~\eqref{eq:direction} is defined on the subset of $\param^{**}(p,
\ca)$ of
$(p,
\ca)$-compatible  parameters;
and it is clearly injective.

\begin{defi}[Compatible parameters]
  \label{defi:direc-compa}
  The  parameter  $(\theta,  \alpha)=\cd(\theta, \beta)$  is
$(p, \ca)$-compatible  if $(\theta, \beta)$ is $(p, \ca)$-compatible  (and
$\beta\neq   \zero$).
\end{defi}
We   shall  write   $p_{\theta,  \alpha}$   for
$\tilde p_{\theta,  \beta}$ when $(\theta,\alpha)=\cd(\theta,\beta)$, and it satisfies~\eqref{assumption-p} when $(\theta, \beta)$ is $(p, \ca)$-compatible.
Notice that for  $\alpha\in \Delta_J^*$ the
set:
\[
\cj^*_\beta=\{j\in \II{1, J}\, \colon \, \beta_j > 0\}
=\{j\in \II{1, J}\, \colon \, \alpha_j > 0\}
\]
is  by   definition  non  empty,  and   we  shall  also  denote   it  by
$\cj^*_\alpha$  or simply  $\cj^*$ when  there  is no  ambiguity on  the
parameter. For the
convenience of the reader, we give explicit formulas for the probability
distributions $p_{\theta, \alpha}$, using~\eqref{eq:direction} and
Lemma~\ref{lem:tp-non-degenere}.  To simplify the expression we set:
\begin{equation}
  \label{eq:def-q1}
q_1=1- p(1) \ind_{\{1\in A_0\}}>0.
\end{equation}

\begin{lem}[The probability distributions $p_{\theta, \alpha}$]
   \label{lem:p-qa}
   Let $\alpha\in \Delta_J^*$ and $\theta\in [0, +\infty ]$ be such that
   the  parameter  $(\theta,  \alpha)$  is  $(p,
   \ca)$-compatible.  The
   non-zero terms  of the probability ditribution  $p_{\theta, \alpha} $
   are            given           as            follows. (Recall that  $0\in A_\cj$ and $
      p\left(A_\cj\cap \{0, 1\}^c\right)>0$.)
\begin{enumerate}[(i)]
\item\label{item:p-alpha-q}
If
$\theta\in   (0,   +\infty   )$, then we have
$g_{A_0}(\theta)<\theta$, $g_{A_j}(\theta)<+\infty $ for $j\in \cj^*$ and:

\begin{align*}
p_{\theta, \alpha}(k) & =\theta^{k-1} \, p(k)\quad  \text{for $k\in A_0$},\\
  p_{\theta, \alpha}(k) & =\alpha_j \frac{\theta -
    g_{A_0}(\theta)}{g_{A_j}(\theta)} \, \theta^{k-1} \, p(k)
  \quad\text{for $k\in A_j$ and $j\in \cj^*$}.
\end{align*}

\item If $\theta=0$, then we have $0\not \in A_0$, $
  \max_{j\in\cj^*}\min (A_j)>1$,
 $p_{\theta, \alpha}(1) =p(1)$ if $1\in A_0$ and:
\[
  p_{\theta, \alpha}(k)=\alpha_j \, q_1
  \quad\text{for $k=\min  A_j$ and $j\in \cj^*$}.
\]

\item If $\theta=+\infty $, then we have  $A_0\subset \{1\}$,  $A_{j_0}= \{0\}$ for some $j_0\in
      \cj^*$, $\cj^*\subset \cj_\infty $,  $  \max_{j\in\cj^*}\max (A_j)>1$,
 $p_{\theta, \alpha}(1) =p(1)$ if $1\in A_0$ and:
\[
  p_{\theta, \alpha}(k)=\alpha_j \, q_1
  \quad\text{for $k=\max  A_j$ and $j\in \cj^*$}.
\]
\end{enumerate}
\end{lem}

For $\alpha\in \Delta^*_J$, we get that  $p_{1, \alpha}=\tilde{p}_{1,\beta}$ with
 $\beta_j=p(A_0^c)\, \alpha_j/p(A_j)$
satisfies~\eqref{assumption-p}, and thus that $(1, \alpha)$ is
$(p,\ca)$-compatible. So $\Delta^*_J$ is indeed the set of all possible directions
 of $(p, \ca)$-compatible probability distributions.
We however complete this picture  with the following result.

\begin{prop}[Possible directions]
  \label{prop:admiss}
  For        every
  $\alpha\in \Delta_J^*$, there  exists $\theta\in (0, 1)$ such that
$(\theta, \alpha)$ is $(p,\ca)$-compatible, that is, such that $(\theta,
\beta)$ is    $(p,\ca)$-compatible
for some $\beta\in  \R^J_+$ with $\beta\neq \zero$ and
 $\alpha$ is the direction of $\tilde p_{\theta, \beta}$.
\end{prop}

\begin{proof}
  We recall the convention $g_{A_0}=0$ if $A_0=\emptyset$.
  Let us  first prove  that there  exists $\theta\in  (0, 1)$  such that
  $\theta>g_{A_0}(\theta)$.  If $0\in A_0$, then $g_{A_0}(0)=p(0)>0$ and
  $g_{A_0}(1)<1$,   so  there   exists   $\theta_0\in(0,1)$  such   that
  $\theta_0=g_{A_0}(\theta_0)$, and then  $g'_{A_0}(\theta_0)<1$. This implies
  $\theta>g_{A_0}(\theta)$    for    $\theta\in    (\theta_0,1)$.     If
  $0\not\in A_0$, then $g_{A_0}(0)=0$ and  $g_{A_0}(1)<1$, so $0$ is the
  only  root   of  $\theta=g_{A_0}(\theta)$   in  $[0,  1]$,   and  thus
  $\theta>g_{A_0}(\theta)$ for $\theta\in (0,1)$.

\medskip

Let us now fix $\theta\in (0, 1)$ such that $\theta > g_{A_0}(\theta)$. We set
for all $j\in \II{1,J}$:
\[
 \beta_j
=\alpha_j\frac{\theta-g_{A_0}(\theta)}{\theta g_{A_j}(\theta)}\ge 0,
\]
and $\beta _0=\theta^{-1}$.
Since $\sum_{j=1}^J \alpha _j=1$, we have:
\[
\sum_{j\in \II{0, J}}\beta_jg_{A_j}(\theta)=\theta^{-1}g_{A_0}(\theta)
+\sum_{j\in \II{1, J}}\alpha_j\theta^{-1}(\theta-g_{A_0}(\theta))
=1.
\]
Hence Condition~\eqref{eq:def-cond-prob} is satisfied.
Moreover we have:
\[
\tilde
p_{\theta,\beta}(A_0^c)=1- \tilde
p_{\theta,\beta}(A_0)
=\theta^{-1}\,
(\theta-g_{A_0}(\theta)) >0.
\]
Finally,  $\alpha\in  \Delta_J^*$ (and  thus  $\alpha\not\in\Delta_J^o$)
insures       that       $\tilde       p_{\theta,\beta}$       satisfies
Lemma~\ref{lem:tp-non-degenere}~\ref{item:def-cond-struct-ndg}, that is,
$\tilde p_{\theta,\beta}$ is non trivial.
\end{proof}

\subsection{Properties of the mean of $p_{\theta, \alpha}$}
\label{sec:mean-alpha}

For $\alpha\in \Delta_J^*$,  we consider the following set:
\begin{equation}
  \label{eq:def-I-alpha}
  I_\alpha=\left\{\theta\in [0, +\infty ]\, \colon\, (\theta,\alpha)
    \text{ is  $(p,
\ca)$-compatible}\right\}.
 \end{equation}

Notice that $1\in I_\alpha$.
  Note  $\rho_\cj$  the radius of convergence
of $\sum_{j\in\cj} g_{A_j}$ or equivalently:
\begin{equation}
   \label{eq:def-rho-J}
  \rho_\cj=\min _{j\in \cj} \rho_{A_j}\in [1, +\infty ].
\end{equation}

We define:
\begin{equation}
   \label{eq:def-qmin-max}
    \theta_{\min}=\inf I_\alpha \in [0, 1)
    \quad\text{and}\quad
  \theta_{\max}=\sup I_\alpha  \in [1, \rho_\cj ] .
\end{equation}
  On  the one  hand, notice  that $\theta_{\min}$  is the  only root  of
  $g_{A_0}(\theta)=  \theta$   in  $[0,  1)$,  so   this  definition  is
  consistent  with~\eqref{eq:def-qmin0}, and  thus $\theta_{\min}$  does
  not  depend  on  $\alpha$.   On  the other  hand,  we  have  that
    $\theta_{\max} $ depends on the support of $\alpha$ as:
\begin{equation}
   \label{eq:q-max}
 \theta_{\max} = \max \big(\rho_{\cj^*}\, ,\,  \sup\{ \theta\in [1, \rho_{A_0}) \,
 \colon\, g_{A_0}(\theta)<\theta\}\big)
 \quad\text{with}\quad
 \rho_{\cj^*}=\min _{j\in \cj^*} \rho_{A_j} .
\end{equation}
  We also have that  $(\theta_{\min}, \theta_{\max})\subset
  I_\alpha$ and that
  $\theta\mapsto p_{\theta, \alpha}$ is continuous on $I_\alpha$ for the
  norm of the total variation.  The  next result is a direct consequence
  of    Lemma~\ref{lem:p-qa}   (the    first    point    is   also    in
  Lemma~\ref{lem:q-min0}).

\begin{lem}
  \label{lem:q-min}
  We  have that:
\begin{enumerate}[(i)]
\item   $\theta_{\min}=0$ if and only if  $0\not\in A_0$,
  \item \label{it:q=0-in-I}
  $\theta_{\min} \in I_\alpha$ if and  only if
  $0\not\in A_0$ and $ \max_{j\in\cj^*}\min (A_j)>1$.
  \item If $0\in A_0$, then we have
$I_\alpha\subset (0, +\infty )$.
\end{enumerate}
  \end{lem}

\medskip

In order to consider finite means, we set for $\alpha\in \Delta^*_J$:
\begin{equation}
  \label{eq:def-q_M}
  I^\mathrm{f}_\alpha=\{\theta\in I_\alpha\, \colon\, \mu_{\theta,
    \alpha} <+\infty \}
  \quad\text{where}\quad
  \mu_{\theta, \alpha}=\mu(p_{\theta, \alpha})\in [0, +\infty ]
\end{equation}
is    the     mean    of    $p_{\theta,    \alpha}$.      Notice    that
$I_\alpha\subset I^\mathrm{f}_\alpha\cup\{\theta_{\max}\}$.

We  are now  interested  in the  existence of  a
critical   probability  distribution   among  the   $(p,\ca)$-compatible
probability distributions  $p_{\theta, \alpha}$  with a  given direction
$\alpha\in \Delta_J^*$, that is in the existence of $\theta\in I_\alpha$
such that $\mu_{\theta, \alpha}=1$.

For  $\alpha\in \Delta^*_J$ and $\theta\geq 0$ such that $\sum_{j\in\cj^*} g_{A_j}(\theta)<+\infty $, we set:
\begin{equation}
   \label{eq:def-H}
H_\alpha(\theta)=\sum_{j\in \cj^*} \alpha_j h_j(\theta),
\end{equation}
where, for  $j\in \cj^*$, $h_j(0)=\min  A_j$ and for $\theta>0$:
\[
h_j(\theta) = \frac{\theta g_{A_j}'(\theta)}{g_{A_j}(\theta)}
= \frac{\E\left[X \theta^X \ind_{\{X \in A_j\}}\right]}
{\E\left[\theta^X \ind_{\{X \in A_j\}}\right]},
\]
where     $X$    is     distributed     according     to    $p$.     For
$\theta\in      I_\alpha\cap      \R_+^*$,       we      have      using
Lemma~\ref{lem:p-qa}~\ref{item:p-alpha-q}:
\begin{equation}\label{mu-expression}
  \mu_{\theta, \alpha}
= g_{A_0}'(\theta) + \frac{\theta - g_{A_0}(\theta)}{\theta}
H_\alpha(\theta).
\end{equation}
Recall $q_1$ from~\eqref{eq:def-q1}. Using Lemma~\ref{lem:p-qa}, if $0\in I_\alpha$  we have that:
\begin{equation}\label{mu-expression-0}
\mu_{0, \alpha}
=(1-q_1) +q_1\sum_{j\in \cj^*}\alpha_j\min A_j.
\end{equation}
and if $+\infty \in I_\alpha$ that:
\begin{equation}\label{mu-expression-infinity}
\mu_{\infty , \alpha}
=(1-q_1) +q_1\sum_{j\in \cj^*}\alpha_j\max A_j.
\end{equation}

We now recall some elementary properties of the function $h_j$, see also
\cite[Lemma 3.1]{j12} for a part of  the proof.  Notice that if $A_j$ is
a singleton, say $\{k_j\}$, then the function $h_j$ is  constant equal to
$k_j$.  Recall  that  $\rho_{A_j}$  is  the  radius  of  convergence  of
$g_{A_j}$ and that $\lim_{x\rightarrow \rho_{A_j}} g_{A_j}(x)= g_{A_j}(\rho_{A_j})$, with the limit being possibly infinite.

\begin{lem}
  \label{lem:hj-monotone}
  Let $j\in \II{1, J}$ with  $\Card(A_j)\geq 2$. The function $h_j$
  defined on  $[0, \rho_{A_j})$ is  $\cc^1$ and increasing,  with
  $h_j'>0$ on $(0, \rho_{A_j})$.
   If $\rho_{A_j}=+\infty $ or if $g_{A_j}(\rho_{A_j})=+\infty $, then we have $\lim_{\theta \rightarrow
   \rho_{A_j}} h_j(\theta )=\sup A_j$.
\end{lem}

As an immediate application, we get the following result (one only needs
to     take      care     of     the     case      $\theta=0$,     where
$H_\alpha(0)=\sum_{j\in  \cj^*}  \alpha_j \min  A_j$,  and  of the  case
$\theta=+\infty                         $,                         where
$H_\alpha(\infty   )=\sum_{j\in    \cj^*}   \alpha_j   \sup    A_j$   if
$\min_{j\in \cj^*} \rho_j=+\infty $).

\begin{cor}[Regularity of $\mu_{\theta, \alpha}$]
     \label{cor:mu-cont}
  The map $\theta \mapsto \mu_{\theta,
    \alpha} $ is continuous on $I_\alpha$,  finite
  on $ I^\mathrm{f}_\alpha$,  and  $\cc^1$ on $(\theta_{\min},
  \theta_{\max})$ and also on $[0, \theta_{\max})$ if $0\in I_\alpha$.
\end{cor}

\subsection{Generic distributions}

Notice that $1 \in I_\alpha$; however we don't assume \emph{a
    priori} that $1\in I^\mathrm{f}_\alpha$ as $p_{1, \alpha}$ might have infinite mean.
 Recall $\rho_\cj=\min_{j\in \cj}
\rho_{A_j}$, see~\eqref{eq:def-rho-J}.

In the next lemmas we give preliminary results on the existence of
$\theta\in I_\alpha$ for $p_{\theta,\alpha}$ to be sub/super/-critical
according to $0$ belonging to $A_0$ or not.

\begin{lem}\label{lem:0-A-0-sub-sup}
Assume that $0\in A_0$  and let $\alpha\in \Delta_J^*$.
\begin{enumerate}[(i)]
\item\label{item:sub-0-in-min}
  There exists $\theta\in I_\alpha$ such that $\mu_{\theta,
       \alpha}< 1$. 

   \item\label{item:sub-0-in-max}
     If  $\rho_\cj=+\infty$, then  there exists  $\theta\in I_\alpha$
     such that $\mu_{\theta, \alpha}> 1$. 
\end{enumerate}
\end{lem}

\begin{proof}
As $0\in A_0$, we  get  using  Lemma~\ref{lem:q-min}  that $I_\alpha\cap  [0,  1]=
(\theta_{\min}, 1]$.
  By
  continuity, we deduce from~\eqref{mu-expression} that:
 \[
 \lim_{\theta\downarrow \theta_{\min}}
  \mu_{\theta,\alpha}=g_{A_0}'(\theta_{\min})<1.
\]
This gives Point~\ref{item:sub-0-in-min}.
\medskip

We now  prove Point~\ref{item:sub-0-in-max}. On  the one hand,  if there
exists $k\in  A_0$ such that $k\geq  2$, then the function  $g_{A_0}$ is
strictly                            convex                           and
$\lim_{\theta \rightarrow\infty  } g_{A_0}(\theta)/\theta= +\infty  $
  as
$\rho_{A_0}\geq \rho_\cj=+\infty $. We deduce that
$g_{A_0}(\theta_{\max})=\theta_{\max}$, and, as $\theta_{\max}>1$, that
$g'_{A_0}(\theta_{\max})>1$.
By
continuity, we deduce from~\eqref{mu-expression} that:
 \[
 \lim_{\theta\uparrow \theta_{\max}}
  \mu_{\theta,\alpha}=g_{A_0}'(\theta_{\max})>1.
\]
On the other hand,  if $A_0\subset \{0, 1\}$, then we get that $
g_{A_0}(\theta)= p(0) + (1-q_1)  \theta$. Thus there is no root of
 $g_{A_0}(\theta)=\theta$ on $(1, +\infty )$, but we have:
\[
  \lim_{\theta\rightarrow\infty } \frac{\theta -
    g_{A_0}(\theta)}{\theta} =q_1>0.
\]
Then, use Lemma~\ref{lem:hj-monotone}
and~\eqref{mu-expression}
to deduce that:
\[
  \lim_{\theta\rightarrow +\infty }
  \mu_{\theta,\alpha}
  =(1-q_1) + q_1 \sum_{j\in \cj^*} \alpha_j \,
  \sup A_j > (1-q_1) + q_1 \sum_{j\in \cj^*} \alpha_j =  1,
\]
 where for the inequality we used that $\sup_{j\in \cj^*} \, \sup
 A_j>1$ as  $p_{\theta, \alpha}$ is
 non trivial and $\sup A_0 \leq 1$. This gives Point~\ref{item:sub-0-in-max}.
\end{proof}

Recall $ \cj_{\infty }=\{j\in \II{1, J}\, \colon\, \sup A_j <\infty
\}$.

\begin{lem}\label{lem:0-A-1-sub-sup}
 Assume that $0\not\in A_0$  and let $\alpha\in \Delta_J^*$.
\begin{enumerate}[(i)]
\item \label{item:sub-0-out-min}
  There exists $\theta\in I_\alpha$ such that $\mu_{\theta,
       \alpha}\leq 1$ if and only if:
     \begin{equation}
       \label{eq:cond-min-alpha}
       \sum_{j\in\cj^*} \alpha_j \min A_j \leq 1.
     \end{equation}
   \item \label{item:sub-0-out-max}
     If $\rho_\cj=+\infty$, then we have   $\mu_{\theta, \alpha}< 1$  for all
    $\theta\in I_\alpha$
   if and only if:
     \begin{equation}
       \label{eq:cond-max-alpha}
    A_0\subset\{1\}, \quad \cj^*\subset \cj_{\infty} ,
    \quad \text{and}\quad \sum_{j\in\cj^*} \alpha_j \max
    A_j < 1.
     \end{equation}
\end{enumerate}
\end{lem}

\begin{proof}
  We  prove   Point~\ref{item:sub-0-out-min}.   We  first   assume  that
  $\sum_{j\in   \cj^*}  \alpha_j   \min  A_j\leq   1$.   By   assumption
  $0\not\in   A_0$  and   $0\in  A_{\cj^*}$   as  $1\in   I_\alpha$  and
  $p_{1,   \alpha}$  satisfies~\eqref{assumption-p}.   According to
  Lemma~\ref{lem:q-min}, we have
    $\theta_{\min}=0$.  
    
    If  $0\in  I_\alpha$,  we  deduce
  from~\eqref{mu-expression-0}                                      that
  $\mu_{0, \alpha}=  1-q_1+ q_1 \sum_{j\in \cj^*}  \alpha_j \min A_j\leq
  1$.        
  
  If      $0\not       \in      I_\alpha$,       that      is
  $ \max_{j\in\cj^*}\min (A_j)\leq 1$,         we         get         that
  $\sum_{j\in \cj^*} \alpha_j \min A_j <1$ as $0\in A_{\cj^*}$, that is,
  $\min A_j=0$  for some $j\in  \cj^*$ and $\alpha_j>0$. We  then deduce
  from~\eqref{mu-expression},      using     that      $q_1$     defined
  in~\eqref{eq:def-q1} is positive, that:
\[
  \lim_{\theta \downarrow 0} \mu_{\theta, \alpha}= (1-q_1) +
  q_1\sum_{j\in \cj^*}\alpha_j\min A_j <1.
\]
So  by Corollary
\ref{cor:mu-cont},   there   exists   $\theta\in   (0,1]$,   such   that
$\mu_{\theta,\alpha}<1$. 

In conclusion,
Condition~\eqref{eq:cond-min-alpha} implies that
$\mu_{\theta,\alpha}\leq 1$ for some $\theta\in I_\alpha$.

\medskip

Let us now assume that $\sum_{j\in \cj^*} \alpha_j \min A_j>1$ (and thus
$0\in I_\alpha$ by Lemma~\ref{lem:q-min}).  Using~\eqref{mu-expression},
the   fact   that   the   functions  $h_j$   are   non-decreasing   (see
Lemma~\ref{lem:hj-monotone} and the fact that $h_j$ is constant equal to
$k_j$ when $A_j$ is reduced to the singleton $\{k_j\}$), we get that for
$\theta\in I_\alpha\cap(0, +\infty )$:
   \[
     \mu_{\theta, \alpha}\geq g_{A_0}'(\theta) + \frac{\theta - g_{A_0}(\theta)}{\theta}
\sum_{j\in \cj^*} \alpha_j \min A_j > g_{A_0}'(\theta) + \frac{\theta -
  g_{A_0}(\theta)}{\theta} =1+\E[(X-1)\theta^{X-1}\ind_{\{X\in A_0\}}],
\]
where $X$  has distribution  $p$. Since $0\not\in  A_0$, we  deduce that
$\E[(X-1)\theta^{X-1}\ind_{\{X\in  A_0\}}]$  is  non-negative  and  thus
$ \mu_{\theta, \alpha}> 1$.   Thanks to~\eqref{mu-expression-0}, we also
have  $\mu_{0, \alpha}>1$,  and  thus $\mu_{\theta,  \alpha}>1$ for  all
$\theta\in I_\alpha$ (notice  that for $\theta=+\infty $,  if it belongs
to   $I_\alpha$,  thanks   to~\eqref{mu-expression-infinity}  one   gets
$\mu_{\infty , \alpha}\geq  \mu_{0, \alpha} > 1$).  This  ends the proof
of Point~\ref{item:sub-0-out-min}.

\medskip

We prove Point~\ref{item:sub-0-out-max}.
If $A_0\subset \{1\}$, we get $\theta_{\min}=0$ by Lemma~\ref{lem:q-min}
and $\theta_{\max}=+\infty $ as $g_{A_0}(\theta)=(1-q_1)\theta<\theta$
and $\rho_\cj=+\infty $, see~\eqref{eq:q-max}.
This gives that $(0, +\infty )\subset I_\alpha$ and  we have:
\begin{equation}
   \label{eq:a0in1}
   \mu_{\theta, \alpha}=1- q_1 + q_1 H_\alpha (\theta)
\quad\text{for}\quad
\theta\in I_\alpha.
\end{equation}

So,    if~\eqref{eq:cond-max-alpha}   holds,    we   have    thanks   to
Lemma~\ref{lem:hj-monotone}    that    $\mu_{\theta,   \alpha}<1$    for
$\theta\in    (0,   +\infty    )\subset   I_\alpha$.    We   also    get
$\mu_{0,   \alpha}<1$,  resp.    $\mu_{\infty  ,   \alpha}<1$,  whenever
$0\in  I_\alpha$,   resp.  $+\infty   \in  I_\alpha$.  This   gives  that
$\mu_{\theta, \alpha} <1$ for all $\theta\in I_\alpha$.

\medskip

Let us assume  that $A_0\subset \{1\}$ and $\sup A_j=+\infty  $ for some
$j\in   \cj^*$ (that is $\cj^*\not \subset \cj_\infty $).
Since $\rho_\cj=+\infty $, Lemma~\ref{lem:hj-monotone} gives that
$\lim_{\theta\rightarrow \infty  } \mu_{\theta,  \alpha} =+\infty $.

We now assume that $A_0\subset \{1\}$, $\cj^* \subset \cj_\infty $ and  $\sum_{j\in
  \cj^*}   \alpha_j  \max   A_j\geq   1$.
Let $j_0\in \cj^*$ such that $0\in A_{j_0}$. If $A_{j_0}=\{0\}$, we get
that $\theta_{\max}=+\infty $ belongs to $I_\alpha$ and~\eqref{eq:a0in1}
implies that $\mu_{\infty , \alpha}\geq 1$.  If $\Card(A_{j_0})\geq 2$,
then we get that $\min_{j\in \cj^*} \max A_j\geq 1$ and $\max_{j\in
  \cj^*} \max A_j> 1$ (as $p_{\theta, \alpha}$
satisfies~\eqref{assumption-p}), and thus $\sum_{j\in
  \cj^*}   \alpha_j  \max   A_j>  1$.
Using~\eqref{eq:a0in1} , we get that
$\lim_{\theta\rightarrow \infty  } \mu_{\theta,  \alpha} >1$.
 In both cases,   there exists $\theta\in I_\alpha$ such that
  $\mu_{\theta, \alpha}\geq 1$.

 We now assume that $A_0\not\subset  \{1\}$. Then the function $g_{A_0}$
 is  increasing  and   strictly  convex.   As
 $\rho_\cj=\infty $, we deduce from~\eqref{eq:q-max} that
 $\theta_{\max}\in (1, +\infty )$ is the maximal root of
 $g_{A_0}(\theta)=\theta$, and thus
 $g_{A_0}'(\theta_{\max}) > 1$.  Then, we get from~\eqref{mu-expression}
 that
 $   \lim_{\theta  \uparrow   \theta_{\max}}   \mu_{\theta,  \alpha}   =
 g_{A_0}'(\theta_{\max}) > 1$.

 In       conclusion,
 if~\eqref{eq:cond-max-alpha}   does   not   hold  then   there   exists
 $\theta\in I_\alpha$ such that $\mu_{\theta, \alpha}\geq 1$.
\end{proof}

Recall $  \cj^ {**}=\{j\in  \cj^*\,  \colon\, \Card(A_j)\geq  2\}$.
We consider the following condition:
\begin{equation}
  \label{eq:patho}
    A_0\cap \{1\}^c\neq \emptyset
  \quad\text{or}\quad
  \cj^{**} \neq \emptyset.
\end{equation}

\begin{prop}[Monotonicity of $\theta\mapsto \mu_{\theta, \alpha}$]
  \label{prop:uniq}
  Let $\alpha\in\Delta_J^*$.
  \begin{enumerate}[(i)]
  \item \label{it:mu=cst}
    Assume~\eqref{eq:patho}  does  not   hold.  Then $I_\alpha=I^\mathrm{f}_\alpha=[0,
    +\infty ]$ and the map $\theta\mapsto p_{\theta, \alpha}$ (as well
    as the map $\theta\mapsto \mu_{\theta, \alpha}$) is constant.
 
  \item          Assume~\eqref{eq:patho}           holds.           Then
    $\partial_\theta   \mu_{\theta,    \alpha}>0$   on    the   interval
    $\{\theta\in I^\mathrm{f}_\alpha\, \colon\, \mu_{\theta, \alpha}\leq
    1\}$.    If  this   set  is   not   empty,  then   its  minimum   is
    $\theta_{\min}$.   Furthermore, there  exists  at  most one  element
    $\theta\in I_\alpha$ such that $\mu_{\theta, \alpha}=1$.
\end{enumerate}
\end{prop}

\begin{rem}[$\mu_{\theta,\alpha}$ independent of $\theta$]
  \label{rem:mu=cst}
Recall $q_1=1- p(1)\ind_{\{1\in A_0\}}$.
  \begin{enumerate}[(a)]
  \item\label{it:mu=cst-1}
    If~\eqref{eq:patho} holds and  the map $\theta\mapsto
\mu_{\theta, \alpha}$ is constant, then $\mu_{\theta, \alpha}>1$.

\item
  If~\eqref{eq:patho}  does  not   hold, that is $A_0\subset\{1\}$ and   $A_j$ is a singleton, say $\{k_j\}$ for
  all  $j\in \cj^*$  (with  one of  the  $k_j$ equal  $0$ and another  larger than $1$ in order  for
  $p_{\theta,     \alpha}$    to     be     non    trivial),     then
  $I_\alpha=[0,        +\infty        ]$       and $  \mu_{\theta,
    \alpha}= (1-q_1)+q_1  \sum_{j\in  \cj^*}  \alpha_j  \, k_j\in (0, +\infty )$.
 Thus, we recover
 Proposition~\ref{prop:uniq}~\ref{it:mu=cst}.

\item Consider the example: $A_0=\{1, k\}$ with $k\geq 2$,
$  \alpha\in \Delta_J^*$ and $A_j=\{k_j\}$ for $j\in \cj^*$ (and thus
$\cj^{**}= \emptyset$) such that $k=\sum_{j\in
    \cj^*} \alpha_j k_j$. Notice that $H_\alpha$ is constant equal to
  $k$ and that the mean $\mu_{\theta, \alpha}$ is constant as:
\[
    \mu_{\theta, \alpha}= p(1) + k \theta^{k-1} p(k) +  (1- p(1) -
  \theta^{k-1} p(k)) H_\alpha(\theta)= 1 + (1-p(1))(k-1).
\]
Thus Point~\ref{it:mu=cst-1} of this remark is not void.
\end{enumerate}
\end{rem}

\begin{rem}[Is $\mu_{\theta, \alpha}$  monotone in $\theta$?]
  \label{rem:mu-monotone}
  Consider the probability $p$ defined by $p(2)=1/2$ and
  $p(0)=p(4)=1/4$; $J=1$ (and thus $\alpha=1$) and $A_1=\{0, 4\}$ (and
  thus $A_0=\{2\}$). We get that $I_1=(0, 2)$ as by
  Lemma~\ref{lem:q-min}~\ref{it:q=0-in-I} $\theta_{\min}=0\not \in I_1$
  and $\theta_{\max}=2$ by~\eqref{eq:q-max}.  It is elementary to check that
  $\lim_{\theta\rightarrow 0+}\mu_{\theta, 1}=0$,  $\mu_{\theta, 1}
  \simeq 1$ for $\theta\simeq 0.36$ and that $\partial_\theta
  \mu_{\theta, 1} <0$ if and only if  $\theta\in [\theta_0, \theta_1]$
  with $\theta_0\simeq 1.24$ and $\theta_1\simeq  1.92$. This provides
  an example where the function $\theta\mapsto \mu_{\theta, \alpha}$ is
  not monotone.
\end{rem}

\begin{proof}[Proof of Proposition~\ref{prop:uniq}]

  Let $\theta\in I_\alpha^\mathrm{f}\cap \R_+^*$. Let $X_\theta$
  be a random variable with distribution $p_{\theta, \alpha}$. We have:
  \[
    \P(X_\theta\in A_0)=\theta^{-1}\, g_{A_0}(\theta)
    \quad\text{and}\quad
    g'_{A_0}(\theta)=\E[X_\theta \ind_{\{X_\theta\in A_0\}}].
  \]
Set:
\[
    f(\theta)= \frac{\theta- g_{A_0}(\theta)}{\theta}=\P(X_\theta\not\in A_0),
 \]
so                                                                  that
$\mu_{\theta ,  \alpha} = g'_{A_0}(\theta) +  f(\theta)
H_\alpha(\theta)$. We get:
\begin{equation}
   \label{eq:mu'}
  \partial_\theta \mu_{\theta, \alpha}= g''_{A_0}(\theta)+ f(\theta)
  H_\alpha'(\theta) + f'(\theta) H_\alpha(\theta).
\end{equation}
(This   expression   can   possibly   be  equal   to   $+\infty   $   if
$\theta=\theta_{\max}$.)  We have $H_\alpha(\theta)\in (0, +\infty ]$ as
$\theta>0$.  We have that $H_\alpha'(\theta)= 0$ if $\cj^{**}=\emptyset$
and,           by            Lemma~\ref{lem:hj-monotone},           that
$H_\alpha'(\theta)\in   (0,  +\infty   ]$  otherwise.    We  also   have
$g''_{A_0}(\theta)=0$     if      $A_0\subset     \{0,      1\}$     and
$g''_{A_0}(\theta)\in  (0, +\infty  ]$  otherwise.   Eventually we  have
$f(\theta)>0$ and:
\begin{equation}
   \label{eq:qf'q-1}
f'(\theta)
=\frac{g_{A_0}(\theta)-\theta\,  g_{A_0}'(\theta)}{\theta^2}
=\theta^{-1} \E\left[(1-X_\theta)\ind_{\{X_\theta\in A_0\}}\right],
\end{equation}
which is finite as $\theta\in I_\alpha^\mathrm{f}$.
Notice that $\mu_{\theta, \alpha}=\E[X_\theta]$, which is finite,  and thus:
\begin{equation}
   \label{eq:qf'q}
\theta\, f'(\theta)
=  \E\left[(X_\theta-1)\ind_{\{X_\theta\not \in A_0\}}\right] + (1-
\mu_{\theta, \alpha}).
\end{equation}

\medskip

\textbf{Case $0\in A_0$.} We first assume  that $0\in A_0$, and thus  $I_\alpha\subset \R_+^*$ and
$\theta_{\min}\not\in I_\alpha$, see Lemma~\ref{lem:q-min}.  We consider
$\theta\in           {I}_\alpha^\mathrm{f}$          such           that
$\mu_{\theta,  \alpha}\leq  1$.   We  deduce  from~\eqref{eq:qf'q}  that
$f'(\theta)=0$ if $\cj^*$ is reduced to a singleton, say $\{j_0\}$, with
$A_{j_0}=\{1\}$  and $\mu_{\theta,  \alpha}=1$, and  that $f'(\theta)>0$
otherwise  (as $0\in  A_0$).  Notice  that it  is not  possible to  have
$A_0\subset\{0,  1\}$, $A_{j_0}=\{1\}$  and $\cj^*=\{j_0\}$  together as
$p_{\theta,\alpha}$     is    non     trivial,     so    at     least
$g''_{A_0}(\theta)\in         (0,          +\infty         ]$         or
$ \E\left[(X_\theta-1)\ind_{\{X_\theta\not  \in A_0\}}\right]  >0$.  The
latter  implies that  $f'(\theta)>0$  as  $\mu_{\theta, \alpha}\leq  1$.
Since    $f(\theta)>0$     and    $H_\alpha(\theta)>0$,     we    deduce
from~\eqref{eq:mu'}  that  $\partial_\theta \mu_{\theta,  \alpha}>0$  on
$\{\theta\in  {I}_\alpha^\mathrm{f}\, \colon\,  \mu_{\theta, \alpha}\leq
1\}$.

\medskip

\textbf{Case $ 0\not\in A_0$.}
We now assume that $0\not\in A_0$. This implies that $\theta_{\min}=0$.
The function $f$ has a continuous extension at $0$ given by
$f(0)=q_1>0$, see~\eqref{eq:def-q1}. We distinguish according to $A_0$
being empty or reduced to $\{1\}$ and $A_0\cap \{1\}^c\neq \emptyset$.

\textbf{Sub-case $ A_0\subset \{1\}$.} We consider the sub-case $A_0\subset \{1\}$. The function $f$  is constant equal to $q_1$, and
$\mu_{\theta,    \alpha}=1-    q_1+     q_1    H_\alpha(\theta)$.
If  $\cj^{**}=\emptyset$ and
thus~\eqref{eq:patho} does not hold, then
we have $p_{\theta,
    \alpha}(k_j)=q_1 \alpha_j$ for $A_j=\{k_j\}$ and $j\in \cj^*$. Thus
Point~\ref{it:mu=cst} is obvious.

If
$\cj^{**}\neq\emptyset$, then  we deduce
from  Lemma~\ref{lem:hj-monotone}  that  the  functions  $H_\alpha$  and
$\theta \mapsto \mu_{\theta, \alpha}$ are increasing  on
$I^\mathrm{f}_\alpha$.

\medskip

% \[ *************\]

\medskip

\textbf{Sub-case $0\not\in  A_0$ and $ A_0\cap  \{1\}^c\neq \emptyset$.}
We get  in particular  that $\theta_{\max}<+\infty  $.  We  introduce an
auxiliary  parameterized function  defined on  $I^\mathrm{f}_\alpha$ for
$\gamma>0$ by:
\begin{equation}
   \label{eq:def-mg}
  m_\gamma(\theta)= g_{A_0}'(\theta)
+  \gamma\, f(\theta) .
\end{equation}
Notice that $m_\gamma(0)=(1-q_1) + q_1 \gamma>0$.
On the one hand, direct computation  yields:
\begin{equation}
   \label{eq:mg}
   m_\gamma(\theta)
   =m_\gamma(0)  + \sum_{k\in A_0\cap\{1\}^c} (k-\gamma)
\theta^{k-1} p(k).
 \end{equation}
 We  deduce  that  if  $\gamma<  \min  (A_0\cap\{1\}^c)$,  and  in
 particular   if
 $\gamma<2$, then $m'_\gamma>0$ on  $I^\mathrm{f}_\alpha$. (Notice that
 $m'_\gamma$ is finite on $I^\mathrm{f}_\alpha$
 except possibly at $\theta_{\max}$ in the case
 where it
 belongs to $I^\mathrm{f}_\alpha$.)
 On the other
 hand,  if  there  exists $\theta_*\in  I^\mathrm{f}_\alpha$  such  that
 $m_\gamma(\theta_*)\leq   1$,  we   deduce  that   $\gamma\leq  1$   if
 $\theta_*=0$ and, from~\eqref{eq:def-mg} that:
\[
  \gamma \leq  \frac{1- g'_{A_0}(\theta_*)}{f(\theta_*)}
  = \frac{\theta_*- \theta_*
    g'_{A_0}(\theta_*)}{\theta_* - g_{A_0} (\theta_*)}<1
  \quad\text{if}\quad \theta_*>0.
\]
In conclusion:
\begin{equation}
  \label{eq:mg-concl}
  \exists \theta_*\in I^\mathrm{f}_\alpha\quad\text{s.t.}\quad
  m_\gamma(\theta_*)\leq 1
  \implies m'_\gamma>0.
\end{equation}

Now, we  go back to  the function $\theta\mapsto  \mu_{\theta, \alpha}$.
Assume   there  exits   $\theta_*\in   I^\mathrm{f}_\alpha$  such   that
$\mu_{\theta_*,  \alpha}\leq  1$.  Set  $\mu_*=\mu_{\theta_*,  \alpha}$,
$\gamma_*=H_\alpha(\theta_*)$  and $m_*=m_{\gamma_*}$.
By construction, we have that $m_*(\theta_*)=\mu_*\leq 1$ and, as
$H_\alpha$ is non-decreasing, for $\theta\in I^\mathrm{f}_\alpha$:
\[
  (\theta - \theta_*)  (  \mu_{\theta, \alpha} -  m_*(\theta))
= (\theta - \theta_*) (H_\alpha(\theta) - H_\theta(\theta_*)) f(\theta)
  \geq 0.
\]
This implies that $\partial_{\theta=\theta_*} \mu_{\theta, \alpha} \geq
m'_* (\theta_*)$, and thus is positive thanks to~\eqref{eq:mg-concl}. We
have obtained that  $\partial_\theta \mu_{\theta,  \alpha}>0$  on
$\{\theta\in  {I}_\alpha^\mathrm{f}\, \colon\,  \mu_{\theta, \alpha}\leq
1\}$.

\medskip

In conclusion, if $A_0\subset \{1\}$ and $\cj^{**}=\emptyset$, then $\theta\mapsto
\mu_{\theta,\alpha}$ is constant; if $A_0\cap\{1\}^c\neq \emptyset$ or
$\cj^{**} \neq \emptyset$,
we have   $\partial_\theta \mu_{\theta,  \alpha}>0$  on
$\{\theta\in  {I}_\alpha^\mathrm{f}\, \colon\,  \mu_{\theta, \alpha}\leq
1\}$,  this  set
 is either empty, or equal to $I_\alpha^\mathrm{f}$ or of the form
$I_\alpha^\mathrm{f}\cap [0, \theta']$, and it contains at most one element  $\theta$ such that $\mu_{\theta, \alpha}=1$.
\end{proof}

Recall  that $\rho_\cj$  is the  radius of  convergence of  the function
$\sum_{j\in\cj} g_{A_j}$ and $I_\alpha$  is an interval of $[0,+\infty]$
defined in~\eqref{eq:def-I-alpha}. We have the following theorem.
Recall  that assuming~\eqref{eq:patho}  is not very restrictive  as
otherwise  the map
$\theta\mapsto p_{\theta, \alpha}$  is constant, see
Proposition~\ref{prop:uniq}.

\begin{theo}[Generic distribution]\label{thm:generic}
   Let     $p$    be     a    probability     distribution    on     $\N$
  satisfying~\eqref{assumption-p}. Let $\ca=(A_j)_{ j\in \II{1, J}}$, with
  $J\in \N^*$, be pairwise disjoint non-empty subsets of $\supp(p)$.
    Let $\alpha\in\Delta_J^*$ and assume that~\eqref{eq:patho} holds.
    The distribution  $p$ is not generic
  for $\ca$ in  the direction  $\alpha$  if and  only if
  one of the  following conditions holds.
  \begin{enumerate}[(i)]
   \item\label{item:0inA0-v2}    $\rho_{\cj}< \infty$,
 $g_{A_0}'(\rho_{\cj}) \leq 1$,   $g_\cj(\rho_{\cj}) <  \infty$ and:
\begin{equation}\label{eq:g-rho-less-infty}
H_\alpha (\rho_{\cj})
< \rho_{\cj}\, \frac{1 - g_{A_0}'(\rho_{\cj})}{\rho_{\cj} - g_{A_0}(\rho_{\cj})}\cdot
\end{equation}
\item \label{item:0niA0-v2} $0\not \in A_0$ and
 $\sum_{j\in\cj^*}
     \alpha_j \min A_j > 1$.
   \item \label{item:0niA0-2-v2}
     $A_0\subset \{1\}$,    $\rho_{\cj} = \infty$  and  $\sum_{j\in\cj^*}
     \alpha_j \sup A_j < 1$.
\end{enumerate}
\end{theo}

\begin{rem}[Direction and generic distribution]
  \label{rem:admi-a}
 A probability distribution might not be generic in all the directions;
 and it may happen that it is generic only in one direction.
\begin{enumerate}[(a)]

   \item \label{item:non-gene-fin} Suppose $\supp(p)=\II{0, 3}$ and
  $\ca=(\{0\},\{2,3\})$ and thus $A_0=\{1\}$.
Consider the direction  $\alpha=(\alpha_1, \alpha_2)$. 
\begin{itemize}
    \item 
The distribution $p$ is generic for $\ca$
in the direction $\alpha$ 
if and only if  $\alpha_2\in [1/3, 1/2]$. 

\item 
If $\alpha_2\in (1/2,1)$, then  Point~\ref{item:0niA0-v2} of
Theorem~\ref{thm:generic} holds since 
$0\not \in A_0$ and  $  \sum_{j=1}^2
\alpha_j\min(A_j)=2\alpha_2>1$; and in this case  all the $p_{\theta, \alpha}$ are sub-critical. 

\item If $\alpha_2\in (0, 1/3)$, then    Point~\ref{item:0niA0-2-v2} of
Theorem~\ref{thm:generic} holds since 
$A_0\subset\{1\}$ and   $  \sum_{j=1}^2
\alpha_j\sup(A_j)=3\alpha_2<1$;  and in this case  all the $p_{\theta, \alpha}$ are super-critical. 

\end{itemize}

\item Suppose  $\supp(p)=\{0, 2\}$,  and
  $\ca=(\{0\}, \{2\})$. Notice that~\eqref{eq:patho}
  does not hold.  In  this example all
  probability   distribution  $p'$   such  that   $\supp(p')=\supp(p)  $   is
  $(p,   \ca)$-compatible.   The   direction   of  $p'$   is  given   by
  $(p'(0), p'(2))$. We  recover that for all  directions in $\Delta_2^*$
  there     exists      a     $(p,      \ca)$-compatible     probability
  distribution.  However, the  probability  distribution $p$  is generic  for
  $\ca$ only in the direction $(1/2, 1/2)$.
\end{enumerate}
\end{rem}

% \[ ************\]
\begin{proof}
We simply  write $M_\alpha$ and  $m_\alpha$ for the  respective supremum
and                              infimum                              of
$\{\mu_{\theta, \alpha}\,\colon\, \theta\in I_\alpha\}$.
Let $\neg$ be the usual logical negation.
\medskip

We first prove that:
\begin{equation}
   \label{eq:non-i}
\rho_\cj<+\infty \quad\text{and}\quad
\neg\ref{item:0inA0-v2} \, \implies \,   M_\alpha\geq 1
  .
\end{equation}
If $\rho_{\cj}  < \infty$ and  $g_{A_0}'(\rho_{\cj}) > 1$, then  we have
$A_0\cap \{0, 1\}^c \neq \emptyset$,  the function $g_{A_0}$ is strictly
convex.  So there
exists   $   \theta^*   \in   (\theta_{\min},   \rho_\cj)$   such   that
$g_{A_0}(\theta^*)<\theta^*$  and  $g_{A_0}'(\theta^*)=1$.  Notice  that
$\theta^*\in  I_\alpha$  and  use~\eqref{mu-expression} to  deduce  that
$\mu_{\theta^*, \alpha}\geq 1$.

If   $\rho_{\cj}   <   \infty$,  $g_{A_0}'(\rho_{\cj})   \leq   1$   and
$   g_\cj(\rho_{\cj})  =   \infty$,  then   we  have   $\rho_\cj>1$  and
$\rho_{\cj}\not\in                     I_\alpha$                     (by
Lemma~\ref{lem:p-qa}~\ref{item:p-alpha-q}).                        Since
$g_{A_0}'(\rho_{\cj})      \leq       1$,      we       deduce      that
$g_{A_0}(\rho_\cj)<\rho_\cj$,   and   thus   $\theta_{\max}=   \rho_\cj$
by~\eqref{eq:q-max}.   Since $  g_\cj(\rho_{\cj}) =  \infty$, we  deduce
there    exists    $j\in\cj^*$    such    that    $\rho_{A_j}=\rho_\cj$,
$g_{A_j}(\rho_\cj)=\infty$ as well as $\sup A_j=+\infty $.  This implies
by  Lemma~\ref{lem:hj-monotone} that  $H_\alpha(\rho_\cj)=+\infty $  and
hence                            by                           continuity
that
$\lim_{\theta\uparrow \theta_{\max}} \mu_{\theta,\alpha}=+\infty $.

If    $\rho_{\cj}< \infty$,
 $g_{A_0}'(\rho_{\cj}) \leq 1$,   $g_\cj(\rho_{\cj}) <  \infty$ and:
\begin{equation}\label{eq:g-rho-less-infty-v2}
H_\alpha (\rho_{\cj})
\geq \rho_{\cj}\, \frac{1 - g_{A_0}'(\rho_{\cj})}{\rho_{\cj} -
  g_{A_0}(\rho_{\cj})},
\end{equation}
then     we    have     $g_{A_0}(\rho_\cj)<    \rho_\cj$     and    thus
$\theta_{\max}=\rho_\cj$         belongs          to         $I_\alpha$.
Use~\eqref{mu-expression}  and~\eqref{eq:g-rho-less-infty-v2} to  deduce
that $\mu_{\theta_{\max}, \alpha}\geq 1$.

This proves that~\eqref{eq:non-i} holds.

\medskip

We    now    consider    the     case    $0\in    A_0$.     Thanks    to
Lemma~\ref{lem:0-A-0-sub-sup}~\ref{item:sub-0-in-min},      we      have
$m_\alpha< 1$.   So $p$ is generic  in the direction $\alpha$  if and
only  if  $M_\alpha\geq  1$. If  Point~\ref{item:0inA0-v2}  holds,  then
$\theta_{\max}=\rho_\cj$   belongs   to   $I_\alpha$,   and   by~\eqref{mu-expression}
and~\eqref{eq:g-rho-less-infty} we get that $\mu_{\theta_{\max}, \alpha} <1$,
which  by Proposition~\ref{prop:uniq} implies that  $M_\alpha< 1$; thus $p$ is not
generic    in    the    direction    $\alpha$.     Now    assume    that
Point~\ref{item:0inA0-v2}  does not  hold.  If  $\rho_\cj<+\infty$, then
use~\eqref{eq:non-i} to  deduce that  $M_\alpha\geq 1$  and that  $p$ is
generic  in the  direction  $\alpha$.  If  $\rho_\cj=+\infty$, then  use
Lemma~\ref{lem:0-A-0-sub-sup}~\ref{item:sub-0-in-max}   to    get   that
$M_\alpha\geq  1$,  and  thus  $p$  is also  generic  in  the  direction
$\alpha$.  This proves the theorem in the case $0\in A_0$.

\medskip   We   now   consider   the   case   $0\not   \in   A_0$.    If
Point~\ref{item:0niA0-v2}       holds,        we       deduce       from
Lemma~\ref{lem:0-A-1-sub-sup}~\ref{item:sub-0-out-min} that $m_\alpha>1$
and thus  $p$ is not  generic in the  direction $\alpha$.

We now assume that $\sum_{j\in\cj^*}  \alpha_j \min A_j \leq 1$.  Thanks
to  Lemma~\ref{lem:0-A-1-sub-sup}~\ref{item:sub-0-out-min}, we  get that
$m_\alpha\leq 1$.   So $p$ is generic  in the direction $\alpha$  if and
only  if  $M_\alpha\geq  1$.  If  $\rho_\cj=\infty  $,  we  deduce  from
Lemma~\ref{lem:0-A-1-sub-sup}~\ref{item:sub-0-out-max}              that
$A_0\subset  \{1\}$ and  $\sum_{j\in\cj^*} \alpha_j  \sup A_j  < 1$  are
equivalent to $M_\alpha<1$, that is, $p$ is not generic in the direction
$\alpha$.   We   eventually   assume  that   $\rho_\cj<\infty   $.    If
Point~\ref{item:0inA0-v2}     holds,
then $\theta_{\max}=\rho_\cj$   belongs   to   $I_\alpha$,
  $\mu_{\theta_{\max}, \alpha} <1$, and     $M_\alpha<    1$     by
Proposition~\ref{prop:uniq},  and  thus  $p$   is  not  generic  in  the
direction $\alpha$.   If Point~\ref{item:0inA0-v2}  does not  hold, then
use~\eqref{eq:non-i} to  deduce that  $M_\alpha\geq 1$  and that  $p$ is
generic in the direction $\alpha$.
\end{proof}

\section{Local limit of large Galton-Watson trees}
\label{sec:main}

We give the  main theorem, see Theorem~\ref{thm:main-thm},  on the local
limit of conditioned BGW tree in the next section. Its proof relies on a
transformation of BGW trees from Rizzolo, see Section~\ref{sec:rizzolo},
and a direct application of local limit theorems for multi-type BGW trees
from~\cite{adg18}, see Section~\ref{sec:proof-main}.

\subsection{Main result}
 Let  $p$ be  a probability  distribution on  $\N$
satisfying~\eqref{assumption-p},   and
let
$\ca=(A_j)_{j\in \II{1,  J}}$, with  $J\in \N^*$, be  pairwise disjoint
non-empty  subsets  of  $\supp(p)$.  Let  $\alpha\in  \Delta^*_J$  be  a
possible direction, see~\eqref{eq:possible-dir}.
We assume the distribution  $p$ is
generic   for  $\ca$   in   the  direction   $\alpha$  (recall   Theorem
\ref{thm:generic}). Thus, there exists a  (unique) $\qc\in I_\alpha$ such that:
\begin{equation}
   \label{eq:def-palpha}
  p_{\alpha}:=p_{\qc, \alpha}
  \quad\text{is critical.}\quad
\end{equation}

Recall  $\ct_q$ denotes  a  BGW  tree with  offspring  distribution $q$  and
$\ct_q^*$ the  corresponding Kesten tree when $\mu(q)\leq 1$.
Recall $|\bn|$ is the $L^1$-norm of $\bn \in \N^J$.

\begin{lem}
  \label{lem:a-admiss}
If  $p$ is generic for $\ca$ in the direction $\alpha\in \Delta_J^*$,
then there exists a sequence $(\bn^{(m)})_{ m\in \N}$ in $\N^J$
such that:
\begin{equation}\label{eq:cond-n_m2}
\P(L_\ca(\ct_p)=\bn^{(m)})>0,
\end{equation}
\begin{equation}
   \label{eq:cond-n_m}
    \lim_{m\rightarrow \infty }
  |\bn^{(m)}|=\infty, \quad \lim_{m\rightarrow \infty }
  \frac{\bn^{(m)}}{ |\bn^{(m)}|}=\alpha,
\end{equation}
and for all $m\in  \N$, $j\in \II{1, J}$, with $\bn^{(m)}=(n_1^{(m)},
\ldots, n_J^{(m)})$:
\begin{equation}
  \label{eq:cond-zero}
  \alpha_j=0 \implies   n^{(m)}_j= 0.
\end{equation}
\end{lem}

\begin{proof}
  Since    $p$    is    generic    for   $\ca$    in    the    direction
  $\alpha\in \Delta_J^*$, there exists  a critical $(p, \ca)$-compatible
  distribution   $p_\alpha$   with   direction   $\alpha$.    Thus,   by
  Definition~\ref{defi:on-p}   of   $(p,  \ca)$-compatible   probability
  distribution, it is  enough to find a sequence  $(\bn^{(m)})_{ m\in \N}$
  such  that~\eqref{eq:cond-n_m2}-\eqref{eq:cond-zero}  hold, with  $\ct_p$
  replaced by $\ct_{p_{\alpha}}$.

  Let   $\ct_{p_\alpha}^{(n)}$   be  distributed   as   $\ct_{p_\alpha}$
  conditioned to have  $n$ vertices.  Notice that for  all finite $M>0$,
  there exists $n>M$  such that the probability of  $\ct_{p_ \alpha}$ to
  have $n$ vertices is positive. Recall $L_k(\bt)$ denotes the number of
  vertices  in $\bt$  with out-degree  $k$.  According  to \cite[Theorem
  7.11]{j12},              along               the              sequence
  $\{n\in  \N\,  \colon\,  \P(\sharp \ct_{p_ \alpha}=n)>0\}$,  the  following
  convergences hold in probability for all $k\in \N$:
 \[
n^{-1} L_k(\ct_{p_\alpha}^{(n)})
\, \xrightarrow[n\to\infty]{\P} \,
p_{\alpha}(k).
\]
Since  $(n^{-1} L_{k}(\ct_{p_\alpha}^{(n)}))_{k\in \N}$ is a random
 probability distribution on $\N$, we also get that:
\begin{equation}\label{eq:conv-la}
\max_{j\in\cj}\val{n^{-1} L_{A_j}(\ct_{p_\alpha}^{(n)})
- p_{\alpha}(A_j)}\xrightarrow[n\to\infty]{\P} 0.
\end{equation}

 \medskip

 In  particular, for $m\in  \N^*$ large enough, we deduce that there exists
 a  tree $\bt^{[m]}$  such  that    $\sharp \bt^{[m]}\geq    m$;
$\P(\ct_{p_\alpha} = \bt^{[m]})>0$; and, with $\bn^{(m)}=(n_1^{(m)}, \ldots,
n_J^{(m)})= L_\ca(\bt^{[m]})$,
$n^{(m)}_j=0$ if $\alpha_j=0$ and for all $j\in \cj^*$:
\[
   \val{ \frac{n^{(m)}_j }{|\bn^{(m)}|} - \frac{p_{ \alpha}
      (A_j)}{1-p_{\alpha} (A_0)} }
  \leq  \frac{1}{m}\cdot
\]
Recall also that $\P(\ct_{p_\alpha}=\bt)>0$ implies that
$\P(\ct_p=\bt)>0$.
This and the fact $\alpha$ is the direction
 of $p_{\alpha}$, that is $p_\alpha(A_j)/(1- p_\alpha(A_0))=\alpha_j$,  end the proof.
\end{proof}

Recall:
\[
  \cj^*=\{j\in \II{1, J}\, \colon\, \alpha_j>0\}
  \quad\text{and}\quad
  \cj=\{0\}
  \cup \cj^*.
\]
For $A\subset \N$, we write  $A-1=\{a-1\,
\colon\, a\in A\}$ and $A-A=\{a-b\, \colon\, a,b\in A\}$. We define:
\begin{equation}
   \label{eq:Gamma-a}
  \Gamma_\alpha= \bigcup_{j\in \cj}\, A'_j,
\quad\text{where}\quad
A'_0=A_0 -1
\quad\text{and}\quad
A'_j=A_j-A_j
\quad\text{for}\quad j\in \cj^*.
\end{equation}

\begin{defi}[Aperiodicity]
  \label{defi:period}
Let  $p$ be  generic for $\ca$ in the direction $\alpha\in \Delta_J^*$.
The probability distribution $p$ is aperiodic for $\ca$ in the direction
$\alpha$ if $\qc\in (0, +\infty )$ and the smallest subgroup of
$\Z$ that contains $\Gamma_\alpha$ is $\Z$.
\end{defi}

\begin{rem}[On aperiodicity]
  \label{rem:period0} $ $
  \begin{enumerate}[(a)]
  \item\label{it:period-0A0}
    If $0\in A_0$,  then $-1\in A'_0$, and thus the
  distribution  $p$ is
  aperiodic for $\ca$ in the direction~$\alpha$.
\item \label{it:patho-singleton}
  If $p$ is aperiodic for $\ca$ in the direction
$\alpha$, then~\eqref{eq:patho} holds.

\item\label{it:period-0infty}   Looking  carefully   at  the   proof  of
  Lemma~\ref{p-ca-critical}~\ref{it:period}, it would be more natural to
  consider      $\gamma'_\alpha$     defined      as     $\Gamma_\alpha$
  in~\eqref{eq:Gamma-a}     but      with     $A_j$      replaced     by
  $A_j\cap \supp (p_\alpha)$.  For $\qc\in (0, +\infty )$ this yields no
  modification   as    $\Gamma'_\alpha=\Gamma_\alpha$.    However,   for
  $\qc\in     \{0,\infty\}$      (which     is     ruled      out     in
  Definition~\ref{defi:period}), the set  $\Gamma'_\alpha$ is reduced to
  $\{0\}$  (as  $A_j\cap  \supp(p_\alpha)$  is  either  a  singleton  of
  empty). So, with the more natural definition that $p$ is aperiodic for
  $\ca$ in the direction $\alpha$ if  the smallest subgroup of $\Z$ that
  contains   $\Gamma'_\alpha$    is   $\Z$,    then   we    still   have
  $\qc\in (0,\infty)$.
  \end{enumerate}
\end{rem}

We now state the main result. Notice we don't assume that $p$ has a
finite mean. 
\begin{theo}[Local limit of conditioned BGW tree]
  \label{thm:main-thm}
  Let    $    p$    be    a    probability    distribution    on    $\N$
  satisfying~\eqref{assumption-p}.  Let $\ca=(A_j)_{j\in \II{1, J}}$,
    with $J\in \N^*$,   be a  family of  pairwise
  disjoint     non-empty    subsets     of     $\supp(p)$    and
  consider the direction
  $\alpha \in \Delta_J^*$.
  We assume that $p$ is generic and aperiodic  for $\ca$ in
  the direction $\alpha$ (and thus $\qc\in (0, +\infty )$).
  If  $(\bn^{(m)})_{ m\in  \N}$  is  a  sequence in  $\N^J$  satisfying
\eqref{eq:cond-n_m2}-\eqref{eq:cond-zero}, then we have:
  \[
\dist(\ct_p\,|\, L_{\ca}(\ct_p)=\bn^{(m)})
\xrightarrow[m\to\infty]{} \dist(\ct^*_{p_{\alpha}}).
\]
\end{theo}

We refer to Remark~\ref{rem:tbd} and the details of its proof given in
Section~\ref{sec:a=0-n>0}
for the usefulness of  the condition~\eqref{eq:cond-zero}, that is, $\alpha_j=0
\implies   n^{(m)}_j= 0$ for the sequence $(\bn^{(m)})_{ m\in  \N}$. 

\begin{rem}[Conditioning on the total size and the number of leaves]
  \label{rem:restriction}
  Notice that conditioning  on the total size and the  number of leaves,
  or equivalently  on the  number of  internal nodes  and the  number of
  leaves, corresponds to $A_1=\N^*\cap \supp(p)$ and $A_2=\{0\}$. Notice
  that                       $A_0=\emptyset$.                        (By
  Remark~\ref{rem:period0}~\ref{it:patho-singleton} notice that the case
  $A_1$     reduced    to     a    singleton     is    excluded     from
  Theorem~\ref{thm:main-thm}, but then it  is equivalent to condition on
  the  number  of  leaves;  and  this  is  considered  in~\cite{ad14a}.)
  Provided the assumptions on genericity and aperiodicity are satisfied,
  , this  case is included  in Theorem~\ref{thm:main-thm} whereas  it is
  excluded \emph{a priori} in Corollary 3.5 in~\cite{adg18}.
\end{rem}

  \begin{rem}[On the case $\cj^{**}=\emptyset$ and $A_0=\emptyset$ or $A_0=\{1\}$]
   \label{rem:singleton}
 Let
$j_0\in \cj^*$ be such that $A_{j_0}=\{0\}$. Notice that $\alpha_{j_0}\in
(0, 1)$.  Since $\cj^{**}=\emptyset$ and $A_0\subset \{1\}$, 
the aperiodic hypothesis is not satisfied. However  the
condition~\eqref{eq:kt=L0-1} implies no choice on $L_0$, so that we can
without loss of generality replace $A_0$ by $A'_0=A_0\cup \{0\}$ and
remove $j_0$ from $\cj^*$, as well as $\alpha$ by $\alpha'\in \R_+^{J-1}$ with
$\alpha'_j=\alpha_j/(1- \alpha_{j_0})$ for $j\in \II{1, J}\setminus\{j_0\}$.
Then the conditioning is the same and the distribution $p$ is aperiodic
for $\ca'=(A_j)_{j\in \II{1, J}\setminus\{j_0\}}$ in the direction $\alpha'$. 

Using this trick, we see that the local convergence of Theorem~\ref{thm:main-thm} holds in this case, eventhough $p$ is not aperiodic. 
 \end{rem}

The  next two  sections are  devoted to  the proof  of the  theorem.  In
Section~\ref{sec:rizzolo},    for    a     tree    $\bt$    such    that
$\cl_\ca(\bt)\neq \emptyset$, we describe a map from $\cl_\ca(\bt)$ onto
a multi-type tree, which is a direct extension of Rizzolo \cite{r15}.
Then, in Section~\ref{sec:proof-main}, we use \cite{adg18} on local
limit of multi-type BGW trees to conclude.

\medskip

From now on the direction $\alpha\in \Delta_J^*$ is fixed and $p_\alpha$
given by~\eqref{eq:def-palpha}
is the  unique critical  probability distribution  $(p, \ca)$-compatible
with     the    direction~$\alpha$.     In    particular,     we    have
$\mu(p_\alpha)=1$.   By  construction,   we   have $p_\alpha(A_j)=0$   if
$\alpha_j=0$ for $j\in \II{1, J}$. Since  only the indices $j$ such that
$\alpha_j$    is    positive    are   pertinent,    for    a    sequence
$\bx=(x_j)_{j\in  \II{1, J}}$,  we shall  consider the  subsequence:
 \begin{equation}
   \label{eq:def-restric*}
    \bx^*=(x_j)_{j\in \cj^*},
   \end{equation}
  where, we recall that $\cj^*=\{j\in \II{1,J}\, \colon\,
  \alpha_j>0\}$.
  For example, we write  $L_{\ca^*}(\bt)=(L_{A_j}(\bt))_{j\in \cj^*}$.

\subsection{Extension of Rizzolo's transformation}
\label{sec:rizzolo}
In  the  following, we  use  the  framework  for multi-type  trees  from
\cite[Section    2]{adg18}.     For    a   tree    $\bt$    such    that
$\cl_\ca(\bt)\neq \emptyset$, we describe a map from $\cl_\ca(\bt)$ onto
a multi-type tree, which is a direct extension of \cite{r15}.

The vertex $u\in\bt$ is said to  have type $j\in\cj$, which we denote by
$e_\bt(u)=j$, if $k_u(\bt)\in A_j$ so  that $(\bt,e_\bt)$ can be treated
as  a $\cj$-type  tree.  Note  $\bt^{[i]}=\{u\in\bt: e_\bt(u)=i\}$.   In
order to remove the $0$-type  vertices, following \cite{r15}, we build a
bijection     $\phi$    (depending     on    $(\bt,     e_\bt)$)    from
$\bt\backslash\bt^{[0]}$  to  a  tree   $\bt^\ca$  with  a  $\cj^*$-type
$e_{\bt^\ca}$, which   preserves  the  types,   that  is,
$e_{\bt^\ca} (\phi(u))=e_\bt(u)\in \cj^*$.   Furthermore, if $\bt^{[0]}$
is empty (which  is automatically the case if  $A_0=\emptyset$), then we
shall  have $  \bt^\ca=\bt$ and  $\phi$ is  the identity  map (and  thus
$e_\bt=e_{\bt^\ca}$).

Let $(\bt, e_\bt)$
be a $\cj$-type tree  such that
$\sharp(\bt\backslash\bt^{[0]})=n\geq 1$.
Following \cite{ad14a}, we define recursively a sequence of growing
$\cj^*$-type trees $(\bt_k, e_{\bt_k})_{k\in \II{1, n}}$ and identify
the last one as $(\bt^\ca, e_{\bt^\ca})$. The map $\phi $ is a
by-product of this construction.
Denote $\prec$ the lexicographic  order on $\cu$.
Let  $u_1\prec\cdots\prec  u_n$  be  the ordered  list  of  vertices  of
$\bt\backslash \bt^{[0]}$. Then, we define recursively:
 \begin{itemize}
 \item $\phi(u_1)=\emptyset$, $\bt_1=\{\emptyset\}$
 and $e_{\bt_1}(\emptyset)=e_{\bt}(u_1)$.
\item For $1<k\leq n$, let  $M(u_{k-1},u_k)\in \{u_1, \ldots, u_{k-1}\}$
be   the most recent  common ancestor of $u_{k-1}$ and $u_k$  and $\bs$ the
  fringe      subtree     of      $\bt$     above      $M(u_{k-1},u_k)$,
  see~\eqref{eq:def-fringe}.

Note $v=\min\{u\in\bs: e_{\bt}(u)\neq 0\}$ (for the lexicographic order).
 Then we set $ \phi(u_k)$ as the concatenation of $\phi(v)$ and
 $(k_{\phi(v)}(\bt_{k-1})+1)$ and consider the tree:
 \[
 \bt_k=\bt_{k-1}\cup \{ \phi(u_k)\},
 \]
 and the type map  $e_{\bt_{k}}$ coincide with $e_{\bt_{k-1}}$ on
 $\bt_{k-1}$ and $e_{\bt_k}(\phi(u_k)) =e_\bt(u_k)$. (This ensures that
 $\phi$ preserves indeed the types.)
\end{itemize}

It is obvious that $(\bt_k,e_{\bt_k})_{k\in \II{1, n}}$ is a sequence of
(increasing)             multi-type              trees.              Let
$(\bt^\ca,e_{\bt^\ca})=(\bt_n,e_{\bt_n})$  and  we   view  $\phi$  as  a
bijection from $\bt\backslash\bt^{[0]}$ to $\bt^\ca$ which preserves the
types. See Fig.~\ref{fig:map} for an example of $\bt$ and $\bt^\ca$ and their
types.

Notice                                                              that
$L_{A_j}(\bt)=\mathrm{Card}\,\{u\in\bt^\ca\,                    \colon\,
e_{\bt^\ca}(u)=j\}$ is the total progeny of  type $j$ (which is equal to
0 if $j\not\in \cj^*$).  For the $\cj^*$-type tree $(\bt^\ca,e_{\bt^\ca})$,
we denote by $\sharp \bt^{\ca}$ the vector  of the total progeny of each type
in $\cj^*$ of $\bt^\ca$:
\begin{equation}
   \label{eq:def-|tA|}
\sharp \bt^{\ca}=L_{\ca^*}(\bt)\in \N^{\cj^*}.
\end{equation}

\begin{figure}[ht]
\begin{center}
\includegraphics{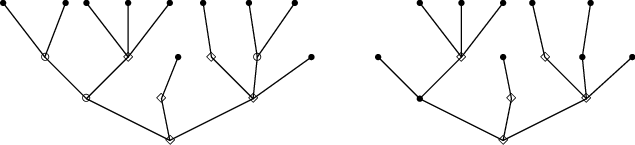}
\end{center}
\caption{ A tree $\bt$ on the left with $A_0=\{2\}$, $A_1=\{1,3\},$
$A_2=\{0\}$ and the tree $\bt^\ca$ after the map on the right.
We represent type 0 with
$\circ$, type 1 with $\diamond $, and type 2 with
$\bullet$.}
\label{fig:map}
\end{figure}

\medskip

Let $\ct_{\alpha}$  be a BGW  tree with critical  offspring distribution
$p_{\alpha}$.   Let  $\ct_{\alpha,*}$  be  distributed  as  $\ct_\alpha$
conditioned to have at least one vertex with out-degree in $A_0^c$, that
is, on $\{L_\ca(\ct_\alpha)  \neq \zero\}$, and, with a  slight abuse of
notation,    we   set    $(\ct_{\alpha}^\ca,e_{\ct_{\alpha}^\ca})$   the
$\cj^*$-type  tree  associated  with $\ct_{\alpha,*}$  by  the  previous
construction.   The proof  of  the next  result, which  is  left to  the
reader, is an adaptation of the proof of~\cite[Theorem 6]{r15}.

\begin{lem}
  \label{lem:J*-BGW}
  The   random  tree   $(\ct_{\alpha}^\ca,e_{\ct_{\alpha}^\ca})$  is   a
  multy-type BGW tree, with types in $\cj^*$.
\end{lem}

The  root of $\ct_{\alpha}^\ca$ is of type $e_{\ct_{\alpha}^\ca}(\emptyset)=j$ with probability
$\alpha_j$ for $j\in \cj^*$.
Let $p^{\ca,\alpha}=(p_j^{\ca,\alpha})_{j\in \cj^*}$ be  the offspring distribution of the
$\cj^*  $-type  BGW tree  $(\ct_{\alpha}^\ca,e_{\ct_{\alpha}^\ca})$,
where  $p^{\ca,\alpha}_j$, a  probability  distribution on  $\N^{\cj^*}$, is  the
offspring distribution of an individual of type $j$.
To describe $p^{\ca,\alpha}_j$, we introduce several intermediate
random variables.
\begin{enumerate}[(a)]
   \item Let  $X$  be  a random variable on $\N$ distributed  according
     to   $p_\alpha$.
   \item Let $X^j$ be distributed as $X$ conditionally on $\{X\in A_j\}$.
   \item Let $(X^0_i)_{i\in\N}$  be independent random
       variables  distributed as  $X-1$ conditionally  on $\{X\in  A_0\}$.
\item Let $N$ be       a  geometric  random
  variable       with        parameter       $p_{\alpha}(A_0^c)$.
\end{enumerate}
We assume that the random variables $X^j$, $(X^0_i)_{i\in\N}$ and $N$
are independent. We adopt the convention $\inf \emptyset=+\infty $.
\begin{enumerate}[resume*]

  \item\label{it:def-T} Set   $  T=\inf\left\{n\in\N^*\,\colon\,  \sum_{i=1}^n
      X_i^0=-1\right\}$.

\item\label{it:loi-Y} Set $Y_j= X^j+\sum_{i=1}^{N-1} X^0_i$ on the event $\{N\le T\}$ and $Y_j=0$ otherwise.
  \item\label{it:loi-Z} Conditionally on the above random variables, let $Z_j$ be   a  binomial
    random  variable with  parameters  $(Y_j, r)$, where:
\begin{equation}
   \label{eq:def-r}
  r=\P(N\leq  T)\in (0, 1].
  \end{equation}
\item\label{it:loi-XA}
  Conditionally on the above  random variables, let
    $X^\ca_j=(X^\ca_j(i))_{i\in \cj^*}$ be a
    multinomial    random   variable     with   parameter
    $(Z_j, \alpha^*)$.
\end{enumerate}
Then, the probability distribution  $p_j^{\ca,\alpha}$  is defined as  the law  of
$X_j^\ca$ conditionally on $\{N\leq
  T\}$.

\medskip

Recall that $p^{\ca,\alpha}$ is said to be  aperiodic, if the smallest
subgroup of $\Z^{\cj^*}$ that contains $\bigcup_{j\in \cj^*}
\big(\supp(p_j^{\ca,\alpha})-\supp(p_j^{\ca,\alpha})\big) $ is $\Z^{\cj^*}$. The
 mean matrix $M=(m_{j\ell})_{ j,\ell\in\cj^*}$ of $p^{\ca,\alpha}$  is defined by:
\begin{equation}
   \label{eq:def-mean-matrix}
m_{j\ell}=\E[X^\ca_j(\ell)\, |\, N\leq T].
\end{equation}

The offspring distribution $p^{\ca,\alpha}$ is  critical if the spectral radius
of the mean matrix $M$ is one.
We have the following properties.
Recall  $\qc$ is the unique $\theta\in [0, +\infty ]$ such that
$p_{\theta, \alpha}$ is critical.

\begin{lem}[Properties of the offspring distribution $p^{\ca,\alpha}$]
  \label{p-ca-critical}
 Let    $    p$    be    a   probability    distribution    on    $\N$
  satisfying~\eqref{assumption-p}, such that it is generic for $\ca$
  in the direction $\alpha\in \Delta^*_J$.
\begin{enumerate}[(i)]
\item\label{it:crit}
  The offspring distribution $p^{\ca,\alpha}$ is  critical and
  $\alpha^*$ is the left eigenvector of the mean matrix associated with
  the eigenvalue 1.
\item\label{it:primitiv}     We    have     $m_{j\ell}>0$    for     all
  $j, \ell \in\cj^*$  and $j\neq j_0$, where $j_0\in \cj$  is defined by
  $0\in A_{j_0}$.  If $j_0\in \cj^*$,  then we have that $m_{j_0\ell}=0$
  for    all    $   \ell    \in\cj^*$     either    if $A_{j_0}=   \{0\}$    and
 $A_0\subset \{1\}$ or if  $\qc=0$, and that $m_{j_0\ell}>0$ for
  all $ \ell \in\cj^*$ otherwise.
\item\label{it:period}
  The offspring distribution $p^{\ca,\alpha}$ is aperiodic if and only
  if $p$ in
aperiodic for $\ca$ is the direction $\alpha$ (and thus  $\qc\not\in \{0, +\infty \}$).
\end{enumerate}
\end{lem}

\begin{proof}
We prove Point~\ref{it:crit} on the criticality of $p^{\ca,\alpha}$.  By
construction and~\eqref{eq:def-mean-matrix},
the entries of the mean matrix $M$  are given by, for
$j,\ell\in \cj^*$:
\[
  m_{j \ell}=\E[X_j^\ca(\ell)\, |\, N\leq T]
=\E[Z_j\, |\, N\leq T]\, \alpha_\ell
=\E[Y_j]\alpha_\ell .
\]
In particular the mean matrix has rank one and
$\alpha^*$ is the left eigenvector
associated with the non-zero eigenvalue, say $\rho$.
Since the mean matrix has nonnegative entries and $\alpha^*$ as positive
entries, we also get that $\rho$ is the Perron-Frobenius eigenvalue and
thus the spectral radius of $M$.
Since $M$ has
rank one, we also get that $\rho$ is the trace of $M$:
\[
\rho=\sum_{j\in \cj^*} \E[Y_j]\, \alpha_j.
\]
We now compute  $r\E[Y_j]$ using~\eqref{eq:def-r}:
\[
\E[Y_j]= \E\left[\left(X^j+\sum_{i=1}^{N-1}
    X^0_i\right)\ind_{\{N\leq T\}}\right]
=r\E[X^j]+\E\left[\sum_{i=1}^{N-1} X^0_i\right]-
\E\left[\sum_{i=1}^{N-1} X^0_i\ind_{\{N\geq T\}}\right].
\]
Using the strong Markov property of $(X_i^0)_{i\in \N^*}$ at the
stopping time $T$ and its definition (see~\ref{it:def-T}), we get:
\begin{align*}
\E[Y_j]
&=r\E[X^j]+\E\left[\sum_{i=1}^{N-1} X^0_i\right]-
(1-r)\left(-1+\E\left[\sum_{i=1}^{N-1} X^0_i\right]\right)\\
&=1+ r\E\left[X^j-1+\sum_{i=1}^{N-1} X^0_i \right]\\
&=1+r\left(\frac{m_j}{p_{\alpha}(A_j)}-1
+\frac{m_0-p_{\alpha}(A_0)}{p_{\alpha}(A_0^c)}\right)\\
&=1+r\left(\frac{m_j}{p_{\alpha}(A_j)}
+\frac{m_0-1}{p_{\alpha}(A_0^c)}\right),
\end{align*}
where $m_\ell= \sum_{k\in A_\ell} k  p_\alpha(k)$ for $\ell\in \cj$.  As
$p_\alpha$ is  critical, we  get that  $\sum_{j\in \cj}  m_j=1$.  Recall
that   $\alpha_j=p_\alpha(A_j)/p_\alpha(A_0^c)$    for   $j\in   \cj^*$.
Therefore, we obtain:
\[
  \rho
  = \sum_{j\in \cj^*}\alpha_j+r\sum_{j\in
    \cj^*}\frac{m_j}{p_\alpha(A_0^c)}
  +r\frac{m_0-1}{p_\alpha(A_0^c)}\sum_{j\in \cj^*}\alpha_j
=1+r\frac{1-m_0}{p_\alpha(A_0^c)}+r\frac{m_0-1}{p_\alpha(A_0^c)}=1.
\]
This ensures that $p^{\ca,\alpha}$ is critical.

\medskip We prove Point~\ref{it:primitiv} on the positive entries of the
mean matrix. Let $j\in \cj^*$. We deduce from~\eqref{eq:def-mean-matrix}
and  from~\ref{it:loi-XA}   (where  $\alpha^*$  has   positive  entries)
and~\ref{it:loi-Z} (where $r>0$)  that $(m_{j\ell})_{\ell\in \cj^*}$ are
all  positive if  $\P(Y_j=0)<1$ and  all zero  if $\P(Y_j=0)=1$.
Notice that  $T=1$ a.s.\  implies  $A_0=\{0\}$ and thus
$Y_j>0$ a.s. on $\{N\leq T\}$, so  $\P(Y_j=0)=1$ implies that $\P(T\geq  2)>0$. We
deduce  that  $\P(Y_j=0)=1$   is  equivalent  to
$\P(X^j=0)=1$  and $\P(N=1)=1$  or $\P(X_i^0=0)=1$.   Thus $\P(Y_j=0)=1$  is
equivalent   to  $0\in   A_j$,   $p_\alpha(A_j\cap   \{0\}^c)  =0$   and
$p_\alpha(A_0)=0$ or $p_\alpha(A_0 \cap\{1\}^c)=0$, that is, $0\in A_j$,
$p_\alpha(A_j\cap  \{0\}^c) =0$  and $p_\alpha(A_0  \cap\{1\}^c)=0$.  To
conclude,   notice  that those conditions are equivalent to either
$A_j=\{0\}$ and $A_0\subset \{1\}$ or $0\in A_j$ and  $\qc=0$.

\medskip

We prove Point~\ref{it:period} on the periodicity  of $p^{\ca,\alpha}$.
Thanks to~\ref{it:loi-XA} and the fact that $\alpha^*$ has positive
entries, we deduce that $p^{\ca,\alpha}$ is aperiodic (in $\Z^{\cj^*}$) if and
only if
the smallest
subgroup of $\Z$ that contains $\Gamma=\bigcup_{j\in \cj^*}
\big(\supp(\text{Law}(Z_j\, |\, N\leq T))-\supp(\text{Law}(Z_j\, |\, N\leq T))\big)
$ is $\Z$.
From the definition of the law of $Z_j$ given in~\ref{it:loi-Z}, we
shall consider the two  cases $r=1$ and $r<1$.
We also  remark   that  $r=1$  if  and  only  if
$T=+\infty   $    a.s.\   or    $N=1$   a.s,   which    corresponds   to
$0\not\in A_0\cap \supp(p_\alpha)$, that is, $0\not\in A_0$ as  $p_\alpha(0)>0$.

In the easy case $0\in A_0$ (and thus $\qc\in (0, +\infty )$), we get on
the one hand
that $X^j>0$ a.e., and as $\P(N=1)>0$, we deduce that $\P(Y_j>0)>0$ and
thus that $\{0, 1\}\subset \supp(\text{Law}(Z_j\, |\, N\leq T))$. This implies that
$p^{\ca,\alpha}$ is aperiodic.
On the other hand, we also get that  $p$ is
  $(\ca, \alpha)$-aperiodic, see
  Remark~\ref{rem:period0}~\ref{it:period-0A0}.

  We  now consider  the  case $0\not\in  A_0$, that  is  $r=1$ and  thus
  $Z_j=Y_j$ and a.s.\  $T=+\infty $.
If  $p_\alpha(A_0)>0$,  we have  $\P(N=k)>0$ for  all $k\in  \N$, and if
 $p_\alpha(A_0)=0$, we have a.s.\ $N=1$ and $A_0\cap \supp(p_\alpha) \subset \{1\}$. In
 both cases, we deduce  from~\ref{it:loi-Y} that:
\[
  \supp(\text{Law}(Z_j\, |\, N\leq T))=\supp(\text{Law}(Y_j\, |\, N\leq T))=\supp(\text{Law}(X^j\, |\, N\leq T))+ \N
(  A_0^\alpha-1),
\]
where  $\N  B=\{n   b\,  \colon,  n\in  \N,  b\in   B\}$
and  $A_0^\alpha=A_0 \cap
\supp(p_\alpha)$.
We set $A_j^\alpha= A_j\cap  \supp(p_\alpha)$ for $j\in \cj^*$ and notice that $A_j^\alpha
=\supp(\text{Law}(X^j\, |\, N\leq T))$.  We then
get that:
\[
  \Gamma=\Big(\bigcup_{j\in \cj^*} (A^\alpha_j- A^\alpha_j)\Big)
  \cup \Z (A_0^\alpha-1).
\]
If $\qc\in \{0, +\infty\}$, we get that $A_0^\alpha\subset \{1\}$ and
$A^\alpha_j$ are singletons for $j\in \cj^*$, so that $\Gamma=\{0\}$.
If $\qc\in (0, +\infty)$, we get that $\Gamma=\Gamma_\alpha$ defined
in~\eqref{eq:Gamma-a}.

Thus  $p^{\ca,\alpha}$ is aperiodic if and only if $\qc\not\in \{0,
+\infty \}$ and the smallest subgroup in $\Z$ containing $\Gamma_\alpha$
is $\Z$, that is, $p$ is
aperiodic for $\ca$ in the direction $\alpha$.
\end{proof}

\subsection{Proof of Theorem~\ref{thm:main-thm}}
\label{sec:proof-main}
Recall
(see Section \ref{sec:lct}) the set of trees $\T(\bt,x)$ for $\bt\in
 \T_0$ and $ x\in \cl_0(\bt)$.
As $p$ is generic for $\ca$ in the direction $\alpha$, there exists by
Lemma \ref{lem:a-admiss} a sequence $(\bn^{(m)})_{ m\in\N}$ satisfying~\eqref{eq:cond-n_m2}-\eqref{eq:cond-zero}.
Since  $p_\alpha$ is $(p,\ca)$-compatible, we have for $m\in\N$, $\bt\in\T_0$
and $x\in\cl_{\{0\}}(\bt)$, with $\ct_\alpha=\ct_{p_\alpha}$ and
$\ct^*_\alpha=\ct^*_{p_\alpha}$, that:
\[
\P(\ct_p\in \T(\bt,x)| L_\ca(\ct_p)=\bn^{(m)})
=\P(\ct_\alpha\in \T(\bt,x)| L_\ca(\ct_{\alpha})=\bn^{(m)}).
\]

\medskip

For $j\in \II{1, J}$, recall $\be^j\in \N^ J$ is the vector with all its entries equal to $0$ but the $j$-th which is equal to $1$, and that $\be^0=\zero$. 
Set  $j_0\in \II{0, J}$ such that $0\in A_{j_0}$ and set:
\[
\bb=\be^{j_0}.
\]
 We have:
\begin{align*}
\P(\ct_{\alpha}\in \T(\bt,x), L_\ca(\ct_{\alpha})=\bn^{(m)})
&=\sum_{\tilde{\bt}\in \T_0}\P(\ct_{\alpha}=\bt\circledast (\tilde{\bt},x))
\ind_{\{L_\ca(\bt\circledast (\tilde{\bt},x))=\bn^{(m)}\}}\\
&=\frac{1}{p_\alpha(0)}\sum_{\tilde{\bt}\in \T_0}\P(\ct_{\alpha}=\bt)\,
\P(\ct_{\alpha}=\tilde{\bt})
\ind_{\{L_\ca(\tilde{\bt})=\bn^{(m)}-L_\ca(\bt)+\bb\}}\\
&=\frac{1}{p_\alpha(0)}\P(\ct_{\alpha}=\bt)\,
\P(L_\ca(\ct_{\alpha})=\bn^{(m)}-L_\ca(\bt)+\bb)\\
&=\P(\ct_{\alpha}^*\in\T(\bt,x))\,\P(L_\ca(\ct_{\alpha})
=\bn^{(m)}-L_\ca(\bt)+\bb),
\end{align*}
where we used that $p_\alpha$ is critical (and thus $\ct_\alpha$ is
a.s.\ finite) for first and third equalities
and~\eqref{eq:loi-Tp*}
for the last.
Recall the notation $\bx^*$ from~\eqref{eq:def-restric*}
which is the restriction of the sequence $\bx$ indexed by $\II{1, J}$ to
the indices $\cj^*$, and that  $L_{\ca^*}(\bt)=(L_{A_j}(\bt))_{j\in \cj^*}$.
We get:
\begin{align}
\P(\ct_{\alpha}\in \T(\bt,x)\,|\, L_\ca(\ct_{\alpha})=\bn^{(m)})
&=\P(\ct_{\alpha}^*\in \T(\bt,x))\, \frac{\P(L_\ca(\ct_{\alpha})
=\bn^{(m)}-L_\ca(\bt)+\bb)}{\P(L_\ca(\ct_{\alpha})=\bn^{(m)})}\nonumber\\
&=\P(\ct_{\alpha}^*\in \T(\bt,x))
\,\frac{\P(L_{\ca^*}(\ct_{\alpha})=\bn^{(m)*}-L_{\ca^*}(\bt)+\bb^*)}
{\P(L_{\ca^*}(\ct_{\alpha})=\bn^{(m)*})},\label{eq:palpha-L-A}
\end{align}
where we used that
$L_{A_j}(\ct_{\alpha})=0$ for $j\not\in \cj$
(and thus $\P(\ct_{\alpha}\in \T(\bt,x))=\P(\ct^*_{\alpha}\in
\T(\bt,x))=0$ if $L_{A_j}(\bt) \neq 0$ for some $j\not\in \cj$) as well
as that $(\bn^{(m)})_{ m\in
  \N}$ satisfies~\eqref{eq:cond-zero} for the second equality.
 Now we apply the  extension of Rizzolo's transformation
 for $\ct_{\alpha}$ to get a $\cj^*$-type BGW tree
  $\ct_{\alpha}^\ca$  such that
  $\sharp \ct_{{\alpha}}^\ca=L_{\ca^*}(\ct_{{\alpha}})$ (see
  definition~\eqref{eq:def-|tA|}).
Hence ~\eqref{eq:palpha-L-A} is equivalent to:
\begin{equation}
   \label{eq:pre-19}
\P(\ct_{\alpha}\in \T(\bt,x)\,|\,  L_\ca(\ct_{\alpha})=\bn^{(m)})
=\P(\ct_{\alpha}^*\in \T(\bt,x))
\, \frac{\P(\sharp \ct_{{\alpha}}^\ca=\bn^{(m)*}-L_{\ca^*}(\bt)+\bb^*)}
{\P(\sharp \ct_{{\alpha}}^\ca=\bn^{(m)*})}\cdot
\end{equation}

We consider the following condition which appears in
Lemma~\ref{p-ca-critical}~\ref{it:primitiv}:
\begin{equation}
  \label{eq:delicat}
  A_0\subset \{1\}
  \quad\text{and}\quad
  A_{j_0}=\{0\}
  \quad\text{for some}\quad
  j_0\in \cj^*.
\end{equation}
We first assume that~\eqref{eq:delicat} does not hold.
Hypothesis  of Theorem~\ref{thm:main-thm}  and Lemma~\ref{p-ca-critical}
ensure  that Assumptions  $(H_1)$ (on  the offspring  distribution being
critical and the mean matrix primitive) and $(H_2)$ (on the aperiodicity
of the offspring reproduction)  hold in~\cite{adg18} and that $\alpha^*$
is  the positive  normalized left  eigenvector of  the mean  matrix (see
hypothesis in Lemma~3.11 in~\cite{adg18} where $a=\alpha^*$ and use that
$(\bn^{(m)},  \,   m\in  \N)$  is   a  sequence  in   $\N^J$  satisfying
\eqref{eq:cond-n_m2}-\eqref{eq:cond-zero}),  so  that the  strong  ratio
theorem  or more  precisely (19)  in~\cite{adg18}  holds, which  entails
that:
\begin{equation}
   \label{eq:19}
\lim_{m\to\infty}\frac{\P(\sharp \ct_{{\alpha}}^\ca=\bn^{(m)*}-L_{\ca^*}(\bt)+\bb^*)}
{\P(\sharp \ct_{{\alpha}}^\ca=\bn^{(m)*})}=1.
\end{equation}

We deduce from~\eqref{eq:pre-19}
and~\eqref{eq:19}
that for all $\bt\in \T_0$ and $x\in \cl_{\{0\}}(\bt)$:
\[
\lim_{m\to\infty}\P(\ct_p\in \T(\bt,x)| L_\ca(\ct_p)=\bn^{(m)})
=\P(\ct_{\alpha}^*\in\T(\bt,x)).
\]
For $\bt\in\T_0$, it is obvious from~\eqref{eq:cond-n_m}
that $
\lim_{m\to\infty}\P(\ct_p=\bt| L_\ca(\ct_p)=\bn^{(m)})
=\P(\ct_{\alpha}^*=\bt)=0.
$
The result thus follows from the fact that  the family $\{(\T(\bt,x),\ \bt\in
\T_0,\ x\in \cl_0(\bt)\}\cup\T_0$ is convergence determining for the
local convergence in $\T_0\cup\T_1$.

\medskip  We now   consider  that~\eqref{eq:delicat}
holds.   We  first  check  that  it  is  enough  to  consider  the  case
$A_0=\emptyset$, where the Rizzolo's transformation is the identity map.
Indeed, if $A_0\neq \emptyset$, that is, $A_0=\{1\}$, then the Rizzolo's
transformation corresponds to discarding individuals with only one
child.  This   amounts  to  replace  the  offspring
distribution     $p$     (resp.     $p_{\theta,\alpha}$)     by     $p'$
(resp. $p'_{\theta,\alpha}=(p')_{\theta,\alpha}$) where $p'(k)=p(k)/q_1$
for  $k\neq  1$   and  $p'(k)=0$ for $k=1$.  Then,  notice   that  $\theta$  s.t.\
$p'_{\theta,\alpha}$ is critical is  exactly $\qc$, so without
confusion, we can also replace $p_\alpha$ by $p'_\alpha=(p')_\alpha$. In
conclusion, using  this modification amounts to only consider  the case:
\begin{equation}
  \label{eq:delicat2}
  A_0=\emptyset
  \quad\text{and}\quad
  A_{j_0}=\{0\}
  \quad\text{for some}\quad
  j_0\in \cj^*.
\end{equation}
Notice this case is  ruled out in~\cite[Corollary~3.5]{adg18}.
However a
slight    modification   of    the   proofs    in~\cite{adg18},  which
we sketch in
Section~\ref{sec:appendix}   (take   $d=\Card(\cj^*)$  and
$d=j_0$ therein) allows to get~\eqref{eq:19}, which we now read as, for
a sequence
$(\bn^{(m)},  \,   m\in  \N)$   in   $\N^J$  satisfying
\eqref{eq:cond-n_m2}-\eqref{eq:cond-zero}:
\begin{equation}
   \label{eq:19-A0=0}
\lim_{m\to\infty}\frac{\P(\sharp \ct_{{\alpha}}=\bn^{(m)*}-L_{\ca^*}(\bt)+\bb^*)}
{\P(\sharp \ct_{{\alpha}}=\bn^{(m)*})}=1,
\end{equation}
where  $\ct_\alpha$ is  seen  as a  multi-type BGW  tree,  where a  node
$u\in  \ct_\alpha$ as  type  $j\in \cj^*$  if $k_u(\ct_\alpha)\in  A_j$.
Notice          the     corresponding      offspring
distribution      is
$\bp=(\bp^{(i)})_{i\in     \cj^*}$     where
$   \bp^{(i)}=(   \bp^{(i)}(\bk))_{\bk\in   \N^{\cj^*}}$  is   a
probability distribution on $\N^{\cj^*}$ whose non-zero terms are given by:
\[
  \bp^{(i)}(\bk)= \frac{ p_\alpha(|\bk|)}{p_\alpha(A_i)}
  \, \mathrm{Mult}(\bk, \alpha^*)
  \quad\text{for}\quad
  |\bk|\in A_i,
 \]
 and $\mathrm{Mult}(\bk, \alpha^*)$ is the probability that a
 multinomial random variable with parameter $(|\bk|, \alpha^*) $
 takes the value $\bk$. Furthermore the type of the root is distributed
 as $\alpha^*$. With this setting, we emphasize
 that~\eqref{eq:19-A0=0} is exactly  (19) in~\cite{adg18}, up to a relabeling.
Once~\eqref{eq:19-A0=0} is established, then we finish the proof  as
in the case where~\eqref{eq:delicat} does not hold.

We now give the properties of the offspring
distribution $\bp$ and the type of the root (recall
that~\eqref{eq:delicat2} holds and that the Rizollo's
transformation is the identity map); Under assumption of
Theorem~\ref{thm:main-thm}, we have:
\begin{enumerate}[(a)]
\item The type of the root is distributed as $\alpha^*$.
\item $\bp^{(j_0)}(\zero)=1$.
\item By Lemma~\ref{p-ca-critical}~\ref{it:primitiv}, the mean
  matrix $M=(m_{ij})_{i,j\in \cj^*}$ is such that for $j\in \cj^*$ we
    have $m_{ij}\in (0, +\infty
    )$ for $i\neq j_0$ and $m_{ij}=0$ otherwise.
  \item By Lemma~\ref{p-ca-critical}~\ref{it:crit}
    $\bp$ is critical     and $\alpha^*$ is the left eigenvector with
    eigenvalue~1.
\item By  Lemma~\ref{p-ca-critical}~\ref{it:period}
  $\bp$ is aperiodic.
\item Since  $p_\alpha$ satisfies~\eqref{assumption-p} and
  $A_0=\emptyset$, we deduce from the definition of $\bp$ that there
  exists a type $j\in \cj^*$ such that individual of type $j$ has two
  children or more with positive probability, that is, $\bp$ is
  non-singular.
\end{enumerate}
In particular, the offspring distribution $\bp$ satisfies
hypothesis~\eqref{eq:hyp-p-d}-\eqref{eq:hyp-p-crit-2} from
Section~\ref{sec:appendix}.
To conclude, we  refer to
Section~\ref{sec:appendix} on how to get
(19) in~\cite{adg18} under this set of hypothesis.

  \begin{rem}[On related work]
\label{rem:related}
  The   case   $A_0=\emptyset$    and   $\Card(A_{j_0})\geq   2$   where
  $j_0\in  \II{1,  J}$  is  such   that  $0\in  A_{j_0}$  (compare  with
  condition~\eqref{eq:delicat})         could         be         handled
  using~\cite[Theorem~5.1]{p16} on  multi-type BGW processes.   (We also
  believe that condition (A5) there, which  amounts to say that for each
  $j\in \II{1,  J}$ there is $k\in  A_j$ such that $k+1  \in A_j$, could
  certainly be  relaxed.)  Notice that the  moments condition considered
  there  does  not  allow  to consider  directions  $\alpha$  such  that
  $\qc=\theta_{\max}$  (this case  might indeed exist).   The possible  vector
  $\ba=(a_1, \ldots, a_J)$ considered in~\cite{p16} (which is associated
  with the critical  BGW multi-type process) corresponds  in our framework
  to $a_j =\qc \alpha_j/ p(A_j)$ for  $j\in \II{1, J}$ and the direction
  $\bar  \bv$  which  appears  in~\cite[Eq.~(5.1)]{p16}  corresponds  to
  $\alpha$. In  our approach, we  first fix the direction  $\alpha$, and
  then  give  sufficient  (and  almost necessary)  conditions  for  the
  existence and uniqueness of the  critical parameter $\qc$ and thus how
  to choose the parameter $\ba$ given the direction $\alpha$.
\end{rem}

\subsection{On  the proof of~\eqref{eq:19-A0=0}}
\label{sec:appendix}

In this  section we quickly  revisit the proof of  (19) in~\cite{adg18},
using slightly different  assumptions in order to take  into account the
particular case~\eqref{eq:delicat} from Section~\ref{sec:proof-main}. In
this section only, we stick to the notations introduced in~\cite{adg18}.
Let  $d\geq   2$  and  set   $[n]=\II{1,  n}$  for  $n\in   \N^*$.   Let
$p=(p^{(i)}, i\in [d])$ with  $p^{(i)}=(p^{(i)}(\bk), \, \bk\in \N^d)$ being
probability  distributions   on  $\N^d$.    We  assume
that:
\begin{equation}
  \label{eq:hyp-p-d}
  \boxed{p^{(d)}(\zero)=1}.
\end{equation}
For $i\in [d]$, let $X_i=(X_i^{(j)}, \,  j\in [d])$ be a random variable
on $\N^d$  with probability distribution  $p ^{(i)}$. In  particular, we
have   that  a.s.\   $X_d=\zero$.  We   consider  the   generating  function
$f=(f^{(i)}, i\in [d])$ of $p$ defined by:
\[
  f^{(i)} (s)=\E\left[\prod_{j\in [d]} s_j ^{X_i^{(j)}}\right],
  \quad\text{where}\quad
  s=(s_j, j\in [d])\in [0, 1]^d.
\]
We  consider   the  mean  matrix   $M=(m_{ij};  \,  i,j\in   [d])$  with
$m_{ij}=\E[X_i^{(j)}]$. We assume that:
\begin{equation}
  \label{eq:hyp-p-mean}
m_{ij}\in (0, +\infty )
\quad\text{for all}\quad
i\in [d-1], \, j\in  [d];
\end{equation}
notice  that $m_{dj}=0$  for  all
$j\in [d]$.  In  particular \emph{the matrix $M $ is  not primitive}, as
there is no $n\in \N^*$ such that $M^n$ has only positive finite
entries; notice that $M$ primitive is  part of assumption (H1)
in~\cite{adg18} (this condition is mainly used to apply Perron-Frobenius
theorem on the existence and uniqueness of a  left and a right
eigenvector having nonnegative entries, and their corresponding
eigenvalue is in fact the spectral radius of $M$).
We recall that $p$ is critical if the spectral radius of $M$ is $1$; and
that $p$ is non-singular if  $f(s)\neq Ms$. We assume that:
\begin{equation}
  \label{eq:hyp-p-crit-1}
  p \quad\text{is critical and non-singular};
\end{equation}
notice this is the
other part  of assumption  (H1) in~\cite{adg18}.
We  also assume  that:
\begin{equation}
  \label{eq:hyp-p-crit-2}
  p \quad\text{is aperiodic},
\end{equation}
that  is,  the
smallest       subgroup       of       $\Z^d$       which       contains
$\bigcup_{i\in [d]} \big(\supp(p ^{(i)})  - \supp(p ^{(i)})\big)$ is $\Z^d$
itself; this correspond to hypothesis (H2) in~\cite{adg18}.
\medskip

For $i\in [d]$, let $\mathbf{e}_i$ denote  the vector of $\R^d$ with all
its entries equal  to $0$ but the  $i$-th which is equal  to $1$.  Using
Perron-Frobenius theorem for the matrix $M$ reduced to the first $d-1$
lines and columns and using
that the  $d$-th line of  $M$ is  zero and the other  entries are  positive, we
deduce that:
\begin{itemize}[-]
   \item the eigenvalue $1$  is simple;
\item  there exists  two left
eigenvectors  with  non-negative   entries:  the  vector  $\mathbf{e}_d$
with  eigenvalue $0$, and a vector $a\in \R^d$ having  positive
with  eigenvalue 1;
\item  there exists a unique right
  eigenvector $a^*=(a^*(i), i\in [d])$ with  eigenvalue 1, and its entries are
  positive but for
the $d$-th wich is zero:  $a^*(d)=0$.
\end{itemize}
This result is the reason why we can remove the primitive assumption
of $M$.

\medskip

Then the  results on  the Dwass  formula for  BGW multi-type  trees from
Section~3.2 in~\cite{adg18}  also hold,  as the expressions  therein are
algebraic in the entries of $p$. (For example Lemma~3.8 holds, but notice
that both terms of the equality therein are zero if $r=d$.)  Now formula
(19) in~\cite{adg18} is  then a direct consequence of  the Dwass formula
and the  technical Lemma~3.11  therein.  This  latter result,  proved in
Section~3.4,  is also  a direct  consequence  of Lemma 3.12,  which
asserts that an  intermediate random variable $Y$ on  $\Z^{2d-1}$ has an
aperiodic distribution,  and of Lemma  4.11, which  is a variant  of the
strong ratio theorem for the  random walk with increments distributed as
$Y$.  Now looking carefully at the proof  of Lemma 3.12, we see that $p$
is    assumed    to    be    aperiodic    (this    is    (H2)    therein
and~\eqref{eq:hyp-p-crit-2} here) and that hypothesis (H1) is only used at
the end of the proof to  get that $\P(X_d=\zero)>0$; but this is clearly
the case  if~\eqref{eq:hyp-p-d} holds.   To conclude, notice  that Lemma
4.11 on the  strong ratio theorem requires  only that the law  of $Y$ is
aperiodic  (which   is  provided  by   Lemma  3.12)  and  that   $Y$  is
integrable. By the  construction of $Y$ given in  Section~3.4, we notice
that $Y$  is integrable if  and only if the  mean matrix $M$  has finite
entries, which  is hypothesis~\eqref{eq:hyp-p-mean}.  In  conclusion, we
obtain that (19)  in~\cite{adg18} holds (notice that the root has  to be of
type $r\neq d$ otherwise the numerator and denominator are both zero).

\begin{rem}[On the extension to the main result of~\cite{adg18} under
hypothesis~\eqref{eq:hyp-p-d}-\eqref{eq:hyp-p-crit-2}]
   \label{rem:Kesten-ADG}
   We  leave   to  the  interested   reader  the  construction   of  the
   corresponding  Kesten tree,  see  Section~2.6 in~\cite{adg18},  where
   here individuals  on the  infinite spine  can not  have type  $d$ (in
   particular,  the  root has  not  type  $d$). (For  example  Lemma~2.9
   therein  holds provided  $i, r$  belong to  $[d-1]$.) Then,  assuming
   hypothesis~\eqref{eq:hyp-p-d}-\eqref{eq:hyp-p-crit-2},   we  have   the
   analogue  of   Theorem~3.1  therein  on  the   local  convergence  in
   distribution,  towards the  Kesten tree  of the  BGW multi-type  tree
   (with  the root  not  being  a.s.\ of  type  $d$  and with  offspring
   distribution $p$) conditioned to have population of type $i$ equal to
   $k(n)_i$           for            $i\in           [d]$           with
   $\lim_{n\rightarrow    \infty   }    k(n)_i/   |k(n)|=a(i)$,    where
   $|k(n)|=\sum_{j\in            [d]}             k(n)_j$            and
   $\lim_{n\rightarrow \infty } |k(n)|=\infty $.
\end{rem}

\subsection{Details for Remark~\ref{rem:tbd} on the condition $n_j=0$ if
  $\alpha_j=0$}
\label{sec:a=0-n>0}

We consider the  following example: a probability  distribution $p$ such
that   $\supp(p)$   containing   but   not  reduced   to   $\{0,   2\}$,
$1\not  \in  \supp(p)$,  $J=2$,  $\ca=(A_1,  A_2)$  is  a  partition  of
$\supp(p)$   (that  is,   $A_0=\emptyset$)  with   $A_1=\{0,  2\}$   and
$A_2\subset 3+2\N$.  Notice that  $A_2\neq \emptyset$. We set:
\begin{equation}
   \label{eq:def-a}
  a= \frac{\sqrt{p(0)}}{\sqrt{p(2)}}\in (0, +\infty ).
\end{equation}

We  consider the
direction  $\alpha=(1, 0)$.   It  is  elementary to  check  that $p$  is
generic    for   $\ca$    in   the    direction   $\alpha$    and   that
$p_\alpha=(p_\alpha(n))_{n\in        \N}$       is        given       by
$p_\alpha(0)=p_\alpha(2)=1/2$.  The  distribution  $p$ is  however  not
aperiodic for $\ca$ in the direction $\alpha$, but thanks to
Remark~\ref{rem:leaves-intro}, we still have the convergence of $\ct$
conditionally on $L_\ca(\ct)=(n,0)$, with $n$ odd going to infinity,
locally in distribution towards the Kesten's tree $\ct^*_\alpha$.

For $n$ odd  going to infinity, we shall check  that the distribution of
$\ct$ conditionally  on $L_\ca(\ct)=(n,1)$ does not  converge locally to
the   distribution    of   $\ct^*_\alpha$,    and   thus
Condition~\eqref{eq:cond-zero} is  required in general to  get the local
limit of conditioned BGW tree from Theorem~\ref{thm:main-thm}. To do so,
we shall simply check the positivity of  the limit, for $n$ odd going to
infinity, of:
 \[
   \P\left(k_{\roott} (\ct) \neq 2\, |\,
     L_{\ca}(\ct)=(n,1)\right)= \frac{B_1(n)}{B_2(n)},
 \]
 where
 \[
   B_1(n)=  \P\left(k_{\roott} (\ct) \neq 2\,,
     L_{\ca}(\ct)=(n,1)\right)
\quad\text{and}\quad
   B_2(n)=  \P\left(
     L_{\ca}(\ct)=(n,1)\right).
 \]
 
 Before going further, we recall that the number of planar binary trees
 with $n$ leaves is:
 \[
   f_{1, n}=\frac{1}{n}\, \binom{2n - 2}{n-1},
 \]
 (in particular $f_{1, n+1}$ is the so called $n$-th Catalan's number) and that
 $f_{1,n}=[z^n]  zC=[z^{n-1}]   C$,  where  we  write   simply  $C$  for
 $C(z)=(1- \sqrt{1-4z})/2z$.  Recall also that $zC^2-C+1=0$.   We deduce
 that the number  of planar forests with $k$ binary  trees and $n\geq k$
 leaves  is  given by  $f_{k,n}=[z^n]  z^kC^k=[z^{n-k}]  C^k$, and  that
 according to~\cite[(B.5)]{deutsch}:
\begin{equation}
  \label{eq:def-cat-forest}
  f_{k, n}=\frac{k}{n}\, \binom{2n - k -1}{n-1}= \binom{2n -k
     -1}{n-1} -  \binom{2n -k
     -1}{n}
 \quad\text{for}\quad
 n\geq k\geq 1.
\end{equation}
We set $f_{k,n}=0$ if $k>n$. 
 \medskip

 Let $n\geq k$  be odd integers.
 On the event $k_\roott(\ct)=k$ and $L_{\ca}(\ct)=(n,1)$,
 we get that $\ct$  can be seen as a forest of $k$  trees grafted on the
 root and  with the forest  having $(n+k)/2$  leaves and $(n-k)/2$
 internal nodes, all of them binary.
 We deduce that:
 \[
   \P\left(k_{\roott} (\ct)=k\,,
     L_{\ca}(\ct)=(n,1)\right)= p(k) p(0)^{(n+k)/2}\,  p(2)^{(n-k)/2}
   f_{k,(n+k)/2},
 \]
 and thus:
\[
   B_1(n)= (p(0) p(2))^{n/2} \sum_{k\in A_2, \, k\leq n}   f_{k, (n+k)/2} \,
   p(k)\,  a^k.
 \]
 A tree  $\bt$ such that  $ L_{\ca}(\bt)=(n,1)$  can be decomposed  as a
 binary tree with $\ell$ leaves, and on one of those $\ell$ leaves one grafts a
 forest   with  $k\in   A_2$   (and  $k\leq   n$)   binary  trees   with
 $(n+k)/2+1 -  \ell$ leaves; and  in total  the tree $\bt$  has $(n+k)/2$
 leaves, $(n-k)/2$ binary  branching nodes and one  node with out-degree
 $k$. Thus, we have:
 \begin{align}
   \label{eq:B2=}
   B_2(n)
   &= \sum_{k\in A_2, \, k\leq n}   p(k)p(0)^{(n+k)/2}
     p(2)^{(n-k)/2}\,
\sum_{\ell=1}^{(n-k)/2+1} \ell f_{1,\ell}\, f_{k, (n+k)/2 + 1 - \ell}
  \\
   \nonumber
   &= (p(0) p(2))^{n/2} \sum_{k\in A_2, \, k\leq n}   F_{k, (n+k)/2}\, p(k)\,  a^k,
 \end{align}
 where for $n\geq k\geq 1$:
 \[
   F_{k, n}= \sum_{\ell=1}^{n-k+1} \ell f_{1,\ell}\, f_{k,
     n + 1 - \ell}
 .
 \]
We give an explicit formula of $F_{k, n}$.
\begin{lem}
  \label{lem:catalan}
We have:
  \[
   F_{k,n}= \binom{2n-k}{n}= \frac{2n-k}{k}\, f_{k,n}
\quad\text{for}\quad
n\geq k\geq 1.
\]
\end{lem}
\begin{proof}
    We have:
\[    F_{k,n}
=[z^{n}] (zC)' z^{k} C^k
= [z^{n}] \inv{k+1} \left(z^{k+1} C^{k+1}\right)'
= [z^{n+1}] \frac{n+1}{k+1} z^{k+1} C^{k+1}
= \frac{n+1}{k+1} f_{k+1,n+1}.
\]
Then, use~\eqref{eq:def-cat-forest} to conclude.
\end{proof}

We shall now consider that $A_2$ is unbounded. Let $\varepsilon\in (0,
1)$ and write:
\begin{align*}
  B_1(n,\varepsilon)
  &=  \sum_{k\in A_2, \,
   k>   \varepsilon n}   f_{k, (n+k)/2} \, p(k)\,  a^k,\\
  B_3(n, \varepsilon)
&=  \sum_{k\in A_2, \, k\leq \varepsilon n}   F_{k,(n+k)/2}\, p(k)\,
                        a^k,\\
   B_4(n, \varepsilon)
&=  \sum_{k\in A_2, \, k> \varepsilon n}   F_{k,
  (n+k)/2}\, p(k)\,  a^k,
\end{align*}
so that using the two latter terms, we can rewrite $B_2(n)$ as:                        
\begin{equation}
   \label{eq:B2=B3-4}
  B_2(n)=(p(0) p(2))^{n/2}\,\left( B_3(n, \varepsilon) + B_4(n, \varepsilon)\right).
\end{equation}
For $k >  \varepsilon n$, we have that:
\[
  F_{k, (n+k)/2} = \frac{n}{k} \, f_{k, (n+k)/2}\leq
\inv{\varepsilon}\, f_{k, (n+k)/2} .
\]
This implies that:
\begin{equation}
   \label{eq:maj-B4}
  B_4(n,\varepsilon)\leq  \inv{\varepsilon} B_1(n,\varepsilon).
\end{equation}

We now assume that $A_2=3+2\N$ and there exists $b\in (0, 1)$ and $M\geq 1$ finite such that $M^{-1}\leq p(k) b^{-k} \leq M$ for $k\in A_2$.
Then, we have with $2m=n+3$ and $k=3+2\ell$:
\begin{align*}
  B_3(n,\varepsilon)
  & \leq M \sum_{k\in 3+2\N, k\leq  \varepsilon n} F_{k,
    (n+k)/2}\,   (ab)^k\\
  & \leq  M(ab)^3 \sum_{\ell\in \N, \ell\leq \varepsilon  m } F_{3+2\ell,
    m+\ell}\,   (ab)^{2 \ell}\\
  & = M(ab)^{-2m+3}\left( 1+(ab)^2\right)^{2m-3}  \sum_{\ell\in \N,
    \ell\leq \varepsilon  m } \binom{2m-
    3}{m+\ell}  \, r^{m+\ell} (1-r) ^{m-3-\ell}
\end{align*}
with $r/(1-r)=(ab)^2$ and thus $r=(ab)^2/\left( 1+(ab)^2\right)$.
As $r<1$, we deduce that:
\[
  \sum_{\ell\in \N,  \ell\leq \varepsilon  m } \binom{2m-  3}{m+\ell} \,
  r^{m+\ell} (1-r) ^{m-3-\ell}= \P(m\leq X \leq (1+\varepsilon) m)\leq
  \P(X \leq (1+\varepsilon) m),
\]
where $X$ is binomial with parameter $(2m-3, r)$.

We now assume that $ab>1$ and that $\varepsilon$ is small enough so that
$ab>(1+\varepsilon)/(1-\varepsilon)$. This yields $2r>1+\varepsilon$, so
that for $m$ large enough, we have $(1+\varepsilon)m\leq r (2m -3)$.
We deduce from~\cite[Theorem~2.1]{zzh22} that, with $j=\lfloor (1+\varepsilon)
    m\rfloor $ and $x=r(2m-3) -j+1$:
\[
  \frac{ \P(X \leq (1+\varepsilon) m)}{ \P(X =j )} \leq  2-x +\sqrt{x^2
    + 4(1-r) j}.
\]
Since $\lim_{m\rightarrow \infty } x=+\infty $ and
 $\lim_{m\rightarrow \infty } (x+j)/j=2r/(1+\varepsilon)$, we deduce
 that:
 \[
   \lim_{m\rightarrow \infty } \frac{ \P(X \leq (1+\varepsilon) m)}{
     \P(X =j )}  =c_0
   \quad\text{with}\quad
   c_0=\frac{r(1-\varepsilon)}{2r -(1+\varepsilon)}\cdot
 \]
 
Recall that $n=2m-3$.  So for $n$ large enough, we have with $k'=3+2\ell'$ and $\ell'=j-m=
 \lfloor\varepsilon     m\rfloor$, and thus  $n+k'=2j$, that:
 \begin{align*}
  B_3(n,\varepsilon)
   &\leq   2c_0 M(ab)^{-2m+3}\left( 1+(ab)^2\right)^{2m-3}
     \binom{2m- 3}{j}  \, r^{j} (1-r) ^{2m-3-j}\\
%   &=   2c_0\, (ab)^{2\ell'} F_{3+2\ell', m+ \ell'}\\
   &= 2 c_0M\, F_{k', (n+k')/2}\,  (ab)^{k'}.
\end{align*}
Notice that $j\geq  (1+\varepsilon)m -1$ and thus $k'\geq  \varepsilon n
+1$, so that:
\begin{equation}
  \label{eq:majo-B3}
   B_3(n,\varepsilon)
  \leq  2 c_0 M^2\, B_4(n,\varepsilon).
\end{equation}
Hence, using~\eqref{eq:maj-B4}, we obtain that:
\[
  B_3(n,\varepsilon)+  B_4(n,\varepsilon)\leq  \frac{1+2c_0M^2}{\varepsilon} B_1(n,\varepsilon).
\]
From the definition of $B_1( n,\varepsilon)$, we get that:
\[
(p(0) p(2))^{n/2}\, B_1(n,\varepsilon)=  \P\left(k_{\roott} (\ct) \geq \varepsilon n\,,
    L_{\ca}(\ct)=(n,1)\right).
\]
We deduce from~\eqref{eq:B2=B3-4} that:
\[
    \liminf_{n\rightarrow \infty }\P\left(k_{\roott} (\ct) \geq
      \varepsilon n\, |\,
     L_{\ca}(\ct)=(n,1)\right)
  =  \liminf_{n\rightarrow \infty } \frac{B_1(n,\varepsilon)}{B_2(n)} \geq
  \frac{\varepsilon}{1+ 2c_0M^2}>0,
\]
provided that there exists $b\in (0, 1)$ and $M\geq 1$ finite such that
$M^{-1}\leq p(k) b^{-k} \leq M$ for $k\in A_2= 3+2\N$ and
$\varepsilon>0$ is chosen so that $ab>(1+\varepsilon)/(1-\varepsilon)$,
with $a$ defined by~\eqref{eq:def-a}.

\medskip

  In conclusion, under the above hypothesis,  the distribution of
$\ct$ conditionally  on $\{L_\ca(\ct)=(n,1)\}$ does not  converge locally  to
the
distribution of  $\ct^*_\alpha$ as $n$ goes to infinity, whereas
 the distribution of
$\ct$ conditionally  on $\{L_\ca(\ct)=(n,0)\}$  converges locally  to
the
distribution of  $\ct^*_\alpha$. Furthermore, conditioning   on
$\{L_\ca(\ct)=(n,1)\}$ and letting $n$ goes to infinity gives  a condensation at the root with positive
  probability.

\bibliographystyle{abbrv}
\bibliography{cgwsub}

\end{document}